\documentclass{article}

\usepackage[utf8]{inputenc}
\usepackage{graphicx,amsmath,amsfonts,amsthm,physics,mathrsfs,enumerate,tikz-cd,svg}
\usepackage{amssymb}
\usepackage{xcolor}
\usepackage[margin=1in]{geometry}
\usepackage{pst-node}
\usepackage[colorlinks,citecolor=cyan,linkcolor=teal]{hyperref}

\graphicspath{ {figures/} }
\DeclareMathOperator{\R}{\mathbb{R}}
\DeclareMathOperator{\Z}{\mathbb{Z}}

\DeclareMathOperator{\A}{\mathbb A}

\DeclareMathOperator{\Aut}{Aut}

\DeclareMathOperator{\PSL}{PSL}

\DeclareMathOperator{\SL}{SL}

\DeclareMathOperator{\AdS}{AdS}

\DeclareMathOperator{\hull}{hull}

\DeclareMathOperator{\supp}{supp}
\DeclareMathOperator{\TAff}{TAff}

\renewcommand{\P}{\mathbb{P}}
\DeclareMathOperator{\SO}{\mathrm{SO}}
\DeclareMathOperator{\RP}{\mathbb{R}\textrm{P}}

\DeclareMathOperator{\HH}{\mathbb{H}}

\newcommand{\fo}{\mathfrak o}
\newcommand{\hk}{\widehat{k}}

\DeclareMathOperator{\CoMin}{CoMin}

\newcommand{\N}{\mathbb{N}}
\renewcommand{\Z}{\mathbb{Z}}
\renewcommand{\R}{\mathbb{R}}

\newcommand{\C}{\mathbb{C}}

\renewcommand{\P}{\mathbb{P}}
\newcommand{\inverse}{^{-1}}

\newcommand{\coaff}{\left(\mathbb A^3\right)^*}

\newcommand{\cM}{\mathcal M}
\newcommand{\cN}{\mathcal N}

\newcommand{\PML}{\mathcal{PML}}
\newcommand{\ML}{\mathcal{ML}}

\DeclareMathOperator{\GL}{GL}
\DeclareMathOperator{\Aff}{SAff}

\DeclareMathOperator{\im}{im}

\DeclareMathOperator{\Hit}{Hit}

\DeclareMathOperator{\diam}{diam}
\DeclareMathOperator{\hol}{hol}

\newcommand{\tlambda}{\widetilde \lambda}
\newcommand{\hlambda}{\widehat \lambda}

\newcommand{\tgamma}{\widetilde \gamma}

\theoremstyle{definition}
\newtheorem{theorem}{Theorem}[section]
\newtheorem{maintheorem}{Main Theorem}

\newtheorem{definition}[theorem]{Definition}
\newtheorem{lemma}[theorem]{Lemma}

\newtheorem{corollary}[theorem]{Corollary}
\newtheorem{proposition}[theorem]{Proposition}
\newtheorem{remark}[theorem]{Remark}

\newtheorem{claim}[theorem]{Claim}

\newtheorem{example}[theorem]{Example}
\newtheorem*{keyquestions*}{Key Questions}

\newcommand{\Addresses}{{
  \bigskip
  \footnotesize
  \noindent M. D. Bobb, \textsc{Max Planck Institute for Mathematics in the Sciences, Leipzig}\par\nopagebreak
  \textit{E-mail address}: \texttt{mbobbmath@gmail.com}
  
  \noindent James Farre, \textsc{Max Planck Institute for Mathematics in the Sciences, Leipzig}\par\nopagebreak
  \textit{E-mail address}: \texttt{james.farre@mis.mpg.de}
  }}

\title{Affine laminations and coaffine representations}
\author{Marit Bobb and James Farre}
\date{}

\begin{document}

\maketitle

\begin{abstract}
    We study surface subgroups of $\SL(4,\R)$ acting convex cocompactly on $\RP^3$ with image in the coaffine group.  
    The boundary of the convex core is stratified, and the one dimensional strata form a pair of \emph{bending laminations}.
    We show that the bending data on each component consist of a convex $\RP^2$ structure and an \emph{affine measured lamination} depending on the underlying convex projective structure on $S$ with (Hitchin) holonomy $\rho: \pi_1S \to \SL(3,\R)$.
    We study the space $\ML^\rho(S)$ of bending data compatible with $\rho$ and prove that its projectivization is a sphere of dimension $6g-7$. \end{abstract}

\section{Introduction}
In this paper we initiate a systematic study of surface groups acting faithfully and strongly convex cocompactly on $\RP^3$.
Every such action is projective Anosov; see e.g.,  \cite{Labourie:Anosov,GW12, BPS19, KLP17, GGKW17,DGK}.
This three-dimensional setting provides a unique opportunity to expand our knowledge of an extremely rich class of Anosov representations and their degenerating behaviors. 

Let $\Aff^*(3,\R)< \SL(4,\R)$ be the stabilizer of a point in $\R^4$, which is called the \emph{special coaffine group}.
We restrict our investigation to the space
\[\{\eta:\pi_1 S\to \Aff^*(3,\R) \text{ faithful and convex cocompact}\},\]
up to $\Aff^*(3,\R)$-conjugation. See \S \ref{subsec: general coaffine}, which extends this discussion to the general coaffine group. 
Necessarily, the linear part $L(\eta) : \pi_1S \to \SL(3,\R)$ is Hitchin (see Lemma \ref{lemma: cccca have Hitchin linear part}), i.e., $L(\eta)$ is the holonomy representation of a properly convex $\RP^2$ structure on $S$ \cite{Choi-Goldman,Hitchin}.
Every such $\eta$ has a minimal convex domain of discontinuity in $\RP^3$ on which it acts with compact quotient homeomorphic to $S\times I$, where $I\subset \R$ is a closed interval, which may be a point.

We work to understand these actions by way of their minimal domains of discontinuity/convex cores and the corresponding boundary data.
There are  bending (geodesic) laminations on the boundary of the convex core, and it is easy to obtain natural convex projective structures on the boundary components, each with holonomy $L(\eta)$. However, the \emph{bending data} transverse to the laminations are more subtle. In developing the technology to understand these data, we reveal a completely new phenomenon. 

The key insight of this paper is that the bending data take values in a flat line bundle over the bending lamination. The holonomy of this bundle encodes delicate dynamical features of the lamination and representation $L(\eta)$.
We summarize our main results informally here.

\begin{maintheorem}\label{mainthm: summary}
    The bending datum for a $3$-dimensional convex cocompact coaffine representation is an \emph{affine measured lamination} on a \emph{convex projective surface}.
    
    That is, given a convex projective structure on a surface and a compatibly constructed affine measured lamination, there is exactly one conjugacy class of coaffine representations so that this is the data at the boundary of one component of the convex core.
\end{maintheorem}

Main Theorem \ref{mainthm: MLrho is a sphere} makes Main Theorem \ref{mainthm: summary} precise.  
\medskip

For context, three classical $3$-dimensional geometries (hyperbolic, anti-de Sitter, and (co)-Minkowski geometry) all have projective models. Their isometry groups are naturally the subgroups of $\SL(4,\R)$ preserving an indefinite quadratic form of different signatures.
In each case, convex cocompact representations  of surface groups  have been studied extensively and are sometimes called \emph{quasi-Fuchsian}. %One main product of this paradigm is that s
Such representations are parameterized by certain ``locally convex pleated surfaces,'' which consist of a hyperbolic metric {bent} along a geodesic lamination according to an {invariant transverse measure} \cite{Thurston:notes, Mess, BF:QFcoMin}. More detailed context is provided in \S \ref{subsec: context and questions}.

Our setting differs crucially in that there is generically no quadratic form on $\R^4$ preserved by the representation. Main Theorem \ref{mainthm: summary} identifies precisely the data necessary to describe general convex cocompact coaffine representations. 
\medskip

In forthcoming work, we analyze the convex core boundary of general deformations of Hitchin representations, not necessarily into the special coaffine group.
We also analyze the boundary of the locus of Anosov coaffine representations and study the structure of their limit sets. 
The theory and techniques developed here constitute a substantial part of the necessary toolkit for both of these tasks.

\subsection{Main results}
Let $\rho: \pi_1S\to \SL(3,\R)$ be a Hitchin representation, i.e., %$\rho$ is in the same path component as the holonomy of a hyperbolic structure on $S$ into $\SO_\circ (2,1)<\SL(3,\R)$.  Then 
$\rho$ preserves a properly and strictly convex domain $\Omega\subset \RP^2$ on which the action is properly discontinuous and cocompact \cite{Choi-Goldman}.  %as well as a dual domain $\Omega^*\subset (\RP^2)^*$ .  
We will be interested in deforming the representation $\rho\oplus 1 : \pi_1S \to \SL(4,\R)$ that fixes a vector in $\R^4$; call this vector $e_4$.  
This representation preserves the cone of all projective lines in $\RP^3$ passing through $\Omega$ and $[e_4]$, and the action of $\rho\oplus 1$ on the cone with $[e_4]$ removed is properly discontinuous. We call this domain of discontinuity $\mathscr C \Omega$ and observe that it has a notion of convexity, as it lies in an affine chart; see Figure \ref{fig: intro_cones}. The quotient by $\rho\oplus 1$ is homeomorphic to $S\times \R$.

Let $\eta: \pi_1S \to \Aut(\mathscr C \Omega)< \Aff^*(3,\R)$ be a coaffine representation with linear part $L(\eta) = \rho$. That is, $\eta$ is a representation of the form
\[
\eta = 
\begin{pmatrix}
    \rho & \bf{0} \\
    Z & 1 \\
\end{pmatrix},
\]
where $Z$ is a function $\pi_1 S \to (\R^3)^*$ satisfying a cocycle condition (see Remark \ref{remark: cocycles make reps}).
%Such a representation $\eta$ preserves a hyperplane in $\RP^3$ if and only if $Z$ is cohomologous to $0$ (Lemma \ref{lemma: coboundaries and reducible reps}). 
Say that $\eta$ is \emph{semi-simple} if it preserves a hyperplane in $\RP^3$ and \emph{non-semi-simple} otherwise.

\begin{figure}[h]
    \centering
    \includegraphics[width=.9\linewidth]{figures/intro_cones.pdf}
    \caption{Left: The cone $\mathscr C \Omega\subset \RP^3$ sliced by a hyperplane, revealing the cross-section which is projectively equivalent to $\Omega$. The cone point is $[e_4]$. Center: The same cone in an affine chart with $[e_4]$ at $\infty.$ Right: The minimal domain $\Xi$ for a convex cocompact representation. The bending laminations appear in dashed lines, $\lambda_+$ labeled.}
    \label{fig: intro_cones}
\end{figure}

It is not difficult to check that $\rho\oplus 1$ is $1$-Anosov; this is an open condition on conjugacy classes of representations \cite{Labourie:Anosov,GW12,KLP17,GGKW17}.  Since $\eta$ may be conjugated in $\SL(4,\R)$ so that it lies arbitrarily close to $\rho\oplus 1$, any coaffine representation with Hitchin linear part is $1$-Anosov, hence enjoys continuous and transverse equivariant boundary maps $\xi$ and $\xi^*$ from the boundary of $\pi_1 S$ into $\RP^3$ and $(\RP^3)^*$, respectively (Lemma \ref{lem: eta is Anosov}).  Then $\xi$ takes values in $\partial \mathscr C \Omega$, the \emph{ideal boundary} of $\mathscr C \Omega$, and has image bounded away from $[e_4]$.

The set \[\Xi_\eta = \textrm{Convex Hull}(\im \xi) \setminus\im \xi\] 
is the minimal convex domain of discontinuity for the action of $\eta$ on $\RP^3$ \cite{DGK}. See Figure \ref{fig: intro_cones}.

We summarize some of the geometrical and topological features of $\Xi_\eta$ and its boundary surfaces coming from its convex structure.  The following theorem is known to experts in the field (see, for example, \cite{Ungemach:thesis}). We include it here to clarify the central geometric object of the paper.

A \emph{geodesic lamination} in this context is a closed subset of $\Omega$ foliated by projective line segments that are inextensible in $\Omega$ (see \S\ref{subsec: surface theory}).

\begin{theorem}\label{thm: pleated disks intro}
    When $\eta$ is non-semi-simple, the boundary $\partial \Xi_\eta \setminus \im (\xi)$ is topologically a union of two disjoint open disks  $\partial\Xi_\eta \setminus \im (\xi)= \partial\Xi_\eta^+\sqcup \partial \Xi_\eta^-$, each of which is \emph{bent} along a geodesic lamination.  That is, there are $\pi_1S$-invariant geodesic laminations $\lambda^\pm\subset \Omega$ and $(\rho, \eta)$-equivariant homeomorphisms
    \[q^\pm :\Omega\rightarrow \partial \Xi_\eta^\pm\]
     that are restrictions of projective transformations on every leaf of $\lambda^\pm$ and every component of $\Omega \setminus \lambda^\pm$.   Furthermore, the maps $q^\pm$ have convex image in $\mathscr C \Omega$.
     %i.e. for all $\gamma\in \pi_1 S$, for all $x\in \Omega$
    %\[q^\pm \circ\rho(\gamma) (x) = \eta(\gamma)\circ q^\pm (x).\]
\end{theorem}
See Theorem \ref{thm: pleated disks} for the proof of the above result.  
We call these geodesic laminations $\lambda^\pm$ the \emph{bending loci} of $\partial \Xi_\eta^\pm$.  We focus on describing the structure of $\partial^+ \Xi_{\eta}$ and the map $q^+$ (see Remark \ref{rmk: positive or negative data}).

\begin{keyquestions*}
\mbox{}
\begin{enumerate}
    \item Which geodesic laminations can appear as the bending locus for $\partial^+ \Xi_\eta$? %for an irreducible coaffine representation $\eta$ with linear part $\rho$?
    \item What are the topological and measure theoretic models for the geometric bending data that describe coaffine representations?
\end{enumerate}
\end{keyquestions*}

In any of the classical settings, every geodesic lamination on $S$ without infinite, isolated leaves admits (up to scale) a simplex of transverse measures along which any hyperbolic structure can be bent into a locally convex pleated surface.
The bending ``angle'' makes sense in these cases because of the existence of a pseudo-Riemannian metric induced by the invariant quadratic form.

Without the rigidity granted in the presence of an invariant pseudo-Riemannian metric, we must examine what structure remains.
For convenience, choose a hyperbolic metric $X$ on $S$ with geodesic flow denoted by $g_t: T^1X \to T^1X$. 
Let $F^\rho \to X$ be the flat $(\R^3)^*$-bundle with holonomy $\rho$ acting naturally on $(\R^3)^*$ and choose a norm $\| \cdot \|$ on $F^\rho$.
Consider also the pullback $E^\rho \to T^1X$ of $F^\rho$ along the tangent projection $\pi: T^1X \to X$.
The Anosov property satisfied by $\rho$ furnishes $E^\rho$ with a dynamical splitting into line bundles
\[E^\rho = E^\rho_1\oplus E^\rho_2\oplus E^\rho_3\]
that are invariant under the geodesic flow.
In particular, each of these line bundles is equipped with a flat connection along each leaf of the geodesic foliation of $T^1X$. If $\rho$ is Zariski-dense, then $E^\rho_i\to T^1 X$ admits no flat connection. 
\medskip

The dynamical splitting of $E^\rho$ is related to the bending data for $\eta$ as follows.
The set of planes in coaffine space form an affine space, so their differences are defined.
In a suitable cohomology theory, the cohomology class of $Z$ has a representative given by taking the differences of support planes for the $2$-strata of $\partial^+\Xi_{\eta}$ (Lemma \ref{lemma: psi is phi}). 
The geometry of the Anosov limit maps $\xi:\partial \pi_1S \to \RP^2$ and $\xi^*: \partial \pi_1S \to \left(\RP^2\right)^*$ implies the following important feature, as in Figure \ref{fig: bending intro}: 
\begin{proposition}[Corollary \ref{cor: convex iff diff in R_k} and Proposition \ref{prop: cocycles localize}]
    A pair of supporting hyperplanes to  $\partial^+ \Xi_\eta$ that intersect in a line very close to a leaf $\ell$ of the bending lamination have difference which is close to $E^\rho_2|_{\ell}$.
\end{proposition} 

\begin{figure}[h]
    \centering
    \includegraphics[width=.7\linewidth]{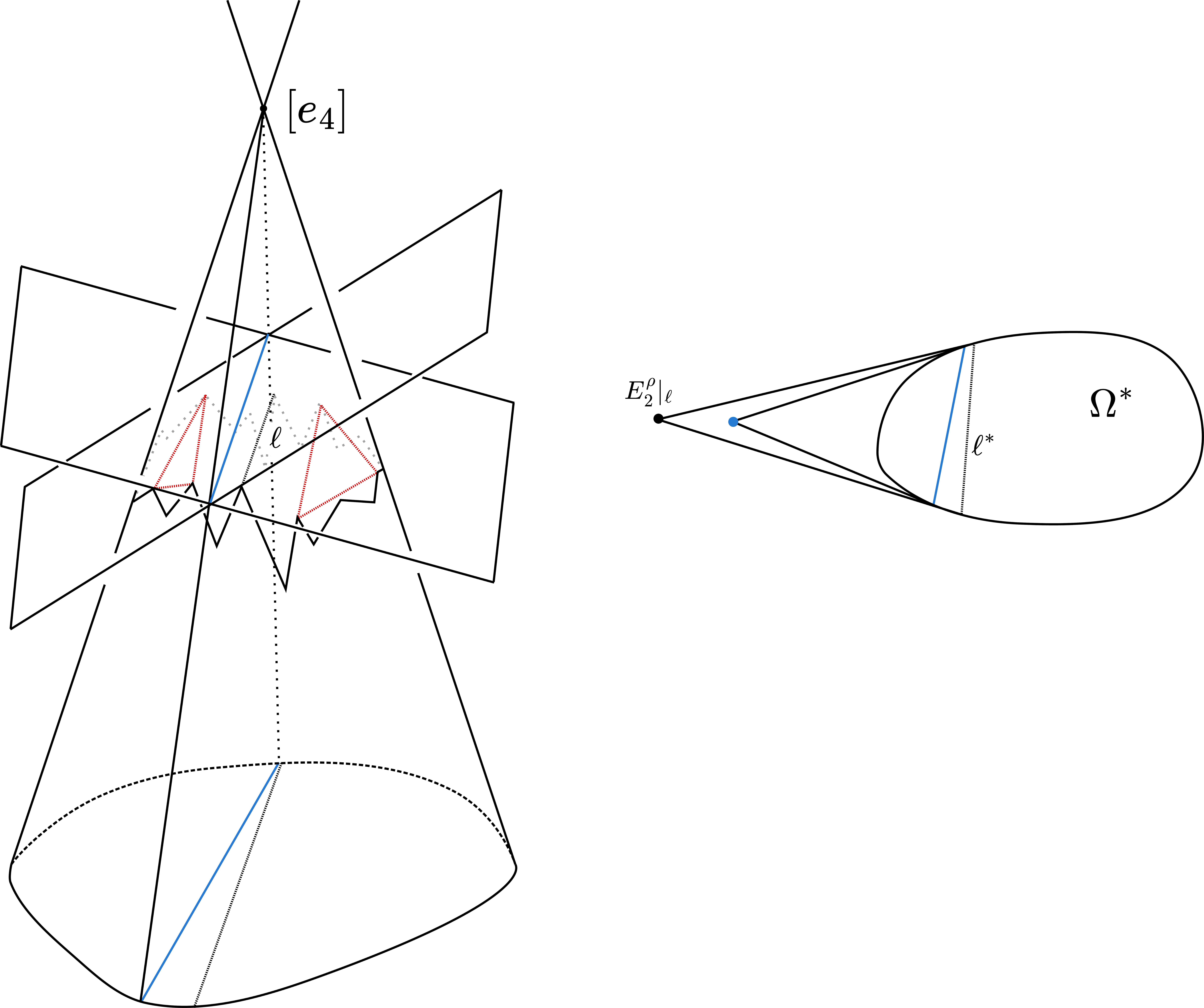}
    \caption{Left: a pair of $2$-strata (red) have supporting hyperplanes intersecting (blue). There is a leaf $\ell$ near their intersection. Right: in $(\RP^2)^*$ the subspace $E_2^\rho\vert_\ell$ is necessarily close to the blue point, dual to the blue line.}
    \label{fig: bending intro}
\end{figure}

Following this observation, we deduce that $Z$ can be encoded as a transverse measure \emph{valued in $E_2^\rho$}, that is, an object that locally looks like a transverse measure to $\lambda^+$ satisfying an equivariance condition tensor a section of $E_2^\rho$.  Making this idea precise takes some careful work, which is carried out in \S\ref{sec: measures and cocycles}. 

\medskip

Having established this relationship between algebraic and geometric information, we seek to further describe the measures obtained in this way, and the bending loci supporting them. 

The simplest geodesic laminations are disjoint unions of simple closed geodesics.  Any lamination decomposes into a finite set of minimal sublaminations and isolated leaves that spiral onto them (see \S\ref{subsec: surface theory}).  
First we understand when minimal laminations are bending loci of a coaffine representation with linear part $\rho$ and then consider the isolated leaves.

The following theorem says that a minimal lamination is the bending locus for some coaffine deformation of $\rho$ if and only if some leaf has non-positive exponential growth in its middle singular direction in the future and in the past.

\begin{theorem}\label{thm: dynamical characterize bending}
    Let $\mu\subset X$ be a minimal geodesic lamination and let $\widehat \mu\subset T^1X$ be its set of tangents.
    Then $\mu$ is the bending locus for a coaffine representation $\eta$ with linear part $\rho$ if and only if there is a $p\in \widehat \mu$ such that 
    \begin{equation}\label{eqn: dynamical condition}
        \limsup_{t\to \infty} \frac{1}{t} \log  {\left\|\left(g_{[0,t]}\right)_*e_2 \right\|_{g_tp}}   \le 0 \text{\hspace{.5cm} and \hspace{.5cm}\ } \limsup_{t\to \infty} \frac{1}{t} \log  {\left\|\left(g_{[0,-t]}\right)_*e_2 \right\|_{g_{-t}p}}   \le 0,
    \end{equation}
    where $e_2 \in E^\rho_2|_p$ is non-zero and $\left(g_{[0,t]}\right)_*$ is the parallel transport for the flat connection along $g_t$-orbits on $E^\rho_2$.
\end{theorem}
Theorem \ref{thm: dynamical characterize bending} appears as Corollary \ref{cor: characterization of bendable laminations} in the text, but the notation is different, for readability. 

\begin{remark}
We note the following.
\begin{itemize}
    \item Theorem \ref{thm: dynamical characterize bending} gives that a simple closed curve $\gamma$ supports bending in coaffine geometry if and only if the middle eigenvalue of $\rho(\gamma)$ is $1$.  This can also be recovered algebraically by examining the centralizer of $(\rho\oplus 1)(\gamma)$ in the cotranslational subgroup of $\Aff^*(3,\R)$.  Compare to the classical geometries, where it is always possible to bend along a simple closed curve.
    \item  The limits in the theorem do not depend on our choice of reference parameterization of the geodesic flow or norm on $E^\rho$.
    \item The conditions are automatically satisfied when $\rho$ is Fuchsian; in this case, $E_2^\rho$ is a trivial flat bundle, hence there is no (exponential) growth along any $g_t$-orbit.
\end{itemize}
\end{remark}

Theorem \ref{thm: dynamical characterize bending} gives a complete answer to Question 1 for minimal geodesic laminations in terms of a dynamical criterion.
Non-orientable, minimal laminations have a time-reversing symmetry, which always guarantees the existence of a point satisfying the hypotheses of Theorem \ref{thm: dynamical characterize bending} for any $\rho$:
    
\begin{theorem}\label{thm: non-orientable minimal intro}
    For every Hitchin representation $\rho$ and every minimal, non-orientable geodesic lamination $\mu\subset X$, there is a coaffine representation $\eta$ with linear part $\rho$ such that the preimage of $\mu$ in $\Omega$ is the bending locus on a component of $\partial \Xi_\eta$.
\end{theorem}

Theorem \ref{thm: non-orientable minimal intro} appears in the text as the first statement of Theorem \ref{thm: minimal laminations support measures} combined with the results in  \S\ref{sec: measures and cocycles}. 
Minimal, non-orientable laminations are abundant (e.g., \S\ref{subsection: geodesic laminations}). 
In future work, we use Theorem \ref{thm: non-orientable minimal intro} to produce analytic paths of Zariski-dense representations $\eta_t\in \mathrm{Hom}(\pi_1S, \SL(4,\R))$ with $\eta_0 =\rho\oplus 1$ and prescribed non-orientable minimal lamination as the bending locus on $\partial^+ \Xi_{t}$.
\medskip

To explain the relevance of Equation \eqref{eqn: dynamical condition}, we outline a construction of a \emph{bending measure}. 
Roughly, this is a transverse measure to a geodesic lamination valued in the line bundle $E_2^\rho$. 
For a minimal lamination $\mu\subset X$ and point $p \in \widehat \mu$, let $v \in E_2^\rho|_p$ and define a Borel measure $\nu_t$ on $T^1X$ with support contained in $\widehat \mu$ taking values in $E_2^\rho$ by the rule  
\begin{equation}\label{eqn: bundle measure intro}
    \nu_t=\frac{1}{N(t)}{\int_{-t}^t \left(g_{[0,s]}\right)_*v \otimes \delta_{g_sp}~ds},
\end{equation}
where $N(t)$ is a normalizing factor making this measure bounded.  
The quantity inside of the integral is a dirac mass times a vector in the bundle $E_2^\rho\vert_\mu$, parallel transported forward and backward via the flat connection along a leaf of $\mu$. 

The hypotheses of Theorem \ref{thm: dynamical characterize bending} describe constraints on how quickly the size of this vector can grow in order to be able to extract convergent subsequences in \eqref{eqn: bundle measure intro} that  satisfy a crucial quasi-invariance property.  
Locally disintegrating the measure along arcs transverse to the leaves of $\mu$ forgets the part of the measure that lives in the direction of the geodesic flow and defines a \emph{transverse measure} to $\mu$ with values in $E_2^\rho$ (\S\ref{subsec: eq transverse measures}).  Such a transverse measure can then be integrated to a developing map from $\Omega$ onto a convex surface in $\mathscr C\Omega$ bent along the leaves of the preimage of $\mu$ in $\Omega$ invariant under some coaffine representation $\eta$. 
We elaborate on this procedure and make this sketch rigorous in \S\S\ref{sec:equivariant measures}--\ref{sec: measures and cocycles}. 
\medskip

Another class of examples of bending laminations, first given in Ungemach's thesis \cite{Ungemach:thesis}, consists of ``atomic'' bending measures giving full mass to isolated, infinite leaves.  We explain this phenomenon and variants in more detail in \S\S\ref{subsec: more on bending locus} and \ref{appendix}.  These results give a partial but satisfying answer to Question 1, and we now move on to Question 2.  

\medskip

A priori, a measure valued in a bundle need not admit any reasonable description, e.g., in terms of some finite combinatorial data. 
Note that the involution on $T^1X$ induces a quotient $\overline E^\rho_2\to \P T^1 X$ of $E^\rho_2\to T^1X$ (see \S\ref{section: flat bundles over laminations}), and there is a natural embedding of $\mu$ into $\P TX$ by taking tangents, so that $\overline E^\rho_2|_\mu$ is defined by pullback.  When $\mu$ is oriented, then $\overline E^\rho_2\vert_\mu$ is just a connected component of $E^\rho_2\vert_{\widehat \mu}$.

Miraculously, the bundle $\overline E^\rho_2\vert_{ \mu}$ admits a flat connection \emph{transverse} to the leaves of $ \mu$, which in turn allows for a surprisingly concise description of the transverse measures valued in $E^\rho_2$.

\begin{theorem}\label{thm: slithering connection intro}
    There is a H\"older continuous flat connection on $\overline E^\rho_2\vert_\mu$ called the \textit{slithering connection}.  Its holonomy can be recorded as a representation     \[\chi: \pi_1 \tau\to \R^\times,\]
    where $\tau$ is a snug train track carrying $\mu$.
\end{theorem}
For the unfamiliar reader, a snug train track is the information of a sufficiently small neighborhood of $\mu$ in $X$ (with some additional structure). 

The connection given by the representation $\rho$ provides a meaningful sense of parallel transport of vectors in $E_2^\rho$ along the leaves of $\mu$---the surprising thing is that this connection extends in a compatible way transverse to the leaves.
We name this extension the \emph{slithering connection} as it is constructed directly from Bonahon--Dreyer's slithering map \cite{BonDre}. Flatness follows from the naturality properties of this map. 
This connection gives local transverse trivializations of $\overline E^\rho_2\vert_\mu$ on collections of fellow-traveling leaves, and the existence of the holonomy representation follows. We clarify this in \S\ref{section: flat bundles over laminations}.

Flatness of the bundle permits the following nice description of the bending measure.

\begin{maintheorem}\label{maincor: intro affine measure}
Bending measures valued in $\overline E^\rho_2|_\mu$ are \emph{affine} (see \S\ref{sec: affine laminations} for definitions).  Consequently,
\begin{itemize}
    \item These measures may be written as a finite collection of positive real numbers on a \textit{train track with stops} and holonomy $\chi$ given by Theorem \ref{thm: slithering connection intro}, 
    \item If $k$ is a suitable transversal to $\mu$, then geodesic flow induces a first return map to $\hk \cap \widehat \mu$, the tangents to $\mu$ over $k$.
    Integrating an affine measure on $\hk\cap \widehat \mu$ 
    produces an \textit{affine interval exchange transformation} (AIET) with an involutive symmetry (and flips) that is measurably semi-conjugate to this first return system; see \S\ref{subsec: AIET}.
\end{itemize}
\end{maintheorem}
We remark that the flips are introduced for convenience, and are not dynamically essential.  
See also \cite{HO:affine} for a discussion of affine measured laminations.
In contrast to their setting, our flat line bundles are only defined over (a neighborhood of) the bending lamination, and this bundle does not in general extend to a flat line bundle over the whole surface (see \eqref{eqn: slithering holonomy}).

Loosely speaking, without knowing that the bending measure is affine, the first return system (the map obtained by following oriented leaves forward until they come back to a transversal segment, equipped with the quasi-invariant bending measure) defines a \emph{generalized} interval exchange transformation with H\"older continuous derivative, where defined.  
Main Theorem \ref{maincor: intro affine measure} upgrades the regularity of the derivative: it is locally constant. See \S\ref{subsection: affine measures from Hitchin world} for a rigorous treatment.

Main Theorem \ref{maincor: intro affine measure} answers Question 2.

\medskip
We now have the language to give a more precise account of Main Theorem \ref{mainthm: summary}.
The bending data for special coaffine representations with linear part $\rho$ organize themselves into a topological space $\ML^\rho(S)$. The objects in the space are pairs consisting of flat line bundles over laminations with 
slithering connection determined by $\rho$ and affine measures valued in those bundles.
These line bundles and their holonomy representations vary continuously in the Hausdorff topology on geodesic laminations.  
There is a weak-$*$ topology on $\ML^\rho(S)$, described in \S\ref{subsec: continuity of weights}.

\begin{maintheorem}\label{mainthm: MLrho is a sphere}
    There exists a natural map
    \[\beta_+: \{\eta: \pi_1 S\rightarrow \Aff^*(3,\R)~\text{non-semi-simple}\mid L(\eta) = \rho\}/\TAff^*(3,\R) \to \ML^\rho(S) \] 
    recording the affine bending lamination on the top component of the convex hull $\Xi_\eta$ from a coaffine representation with linear part $\rho$. 
    The map $\beta_+$ is a homeomorphism that is homogeneous with respect to positive scale.
    Consequently, the projectivization $\PML^\rho(S)$ is a sphere of dimension $6g-7$.
\end{maintheorem}

In principle, one could have defined and studied the space $\ML^\rho(S)$ independently of coaffine geometry. 
This abstract perspective would obscure the fact that $\ML^\rho(S)$ is a manifold.
The geometric viewpoint afforded by coaffine geometry proves that its projectivization is a sphere.  
Compare this discussion to  \cite{HO:affine}.

This concludes our summary of the main results.

\subsection{More on the bending locus}\label{subsec: more on bending locus}

Ungemach proved in his thesis that the bending loci of coaffine representations could be laminations admitting no invariant transverse measure of full support \cite{Ungemach:thesis}.  We now give an outline of this construction.

Consider a Hitchin representation $\rho$ and a simple closed curve $\gamma$ with nontrivial holonomy in $E_2^\rho$, i.e., such that the middle eigenvalue of $\rho(\gamma)$ is not one.
One can find a simple complete geodesic $\ell$ spiraling onto it in the direction such that  
$\|(g_{[0,t]})_* v\|$ 
is decreasing exponentially for any nonzero $v\in E_2^\rho\vert_\ell$ as $t$ tends to $\pm\infty$ (see Figure \ref{fig: prototypes}). As a result, the measure $\nu_t$ defined in \eqref{eqn: bundle measure intro} with $N(t) \equiv 1$ is essentially bounded by a geometric series and converges with no renormalization.  The result is a bending measure whose mass is concentrated along the isolated, spiraling leaf, but whose support is the closure.  
This example can be readily extended to any collection of leaves that spiral in the correct direction around $\gamma$ or a suitable multi-curve; the maximum number of such leaves is $6g-6$.  See Figures \ref{fig: prototypes} and \ref{fig: isolated leaves}.

\begin{proposition}[Ungemach]\label{prop: Ungemach}
    Dirac masses supported on maximal families of isolated geodesics account for full-dimensional simplices in the space of bending measures compatible with $\rho$.
\end{proposition}
The union of examples constructed in this way appears to be dense in this space, which we describe in more detail in \S\ref{sec: affine laminations}.
In other words,  in coaffine geometry, bending loci where all of the bending is concentrated on spiraling leaves are abundant.  

Laminations with isolated infinite geodesics are \emph{never} the bending locus in the classical geometries, as they do not carry transverse measures of full support.  
We point out that not only do bending laminations constructed with spiraling isolated leaves with contracting holonomy not carry ordinary \emph{invariant} transverse measures of full support, but also they are not even chain recurrent.  That is, they are not Hausdorff limits of simple closed multi-curves.  
\medskip

The two examples discussed in this introduction --- minimal laminations with a leaf having non-positive exponential growth (Theorem \ref{thm: dynamical characterize bending}) and Ungemach's construction (Proposition \ref{prop: Ungemach}) --- are dynamically extreme and do not account for every possible bending locus. Ungemach's examples generalize easily to orientable laminations where all leaves have a similar uniform exponential contraction property.  One can then find an isolated leaf that spirals in such a way that the parallel transport contracts the norm of a vector in $E_2^\rho$ exponentially and construct similarly ``atomic'' examples, where all of the bending mass is concentrated on a single or finite collection of leaves of the bending lamination; again see Figure \ref{fig: prototypes}.

Given an oracle for the dynamical behavior of $E_2^\rho\vert_\ell$ for every leaf $\ell$ of a lamination $\mu$, it should be possible to account for all possible equivariant transverse measures on that lamination, but stating this precisely seems to be more complicated than it is useful.
Nevertheless, we detail some further constructions of bending loci as well as some gaps in our understanding in Remark \ref{rmk: can and cant do}. 
=

\begin{figure}[h]
    \centering
    \includegraphics[width=.9\linewidth]{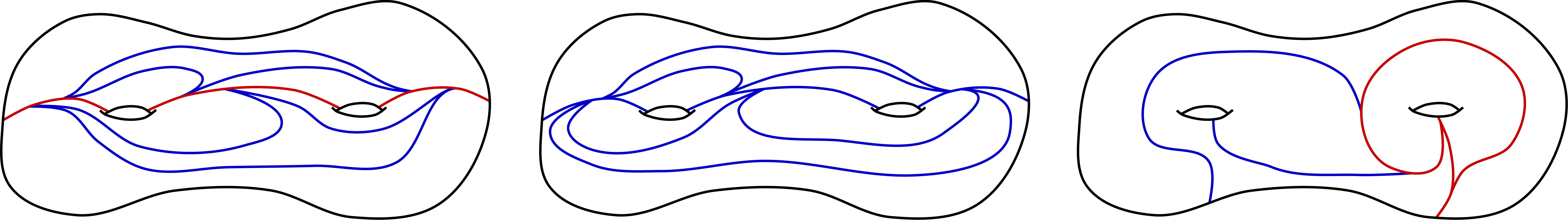}
    \caption{The three prototypical types of bending laminations are depicted by approximating train tracks. Minimal components are shown in red while isolated spiraling leaves are blue.} 
    \label{fig: prototypes}
\end{figure}

Summarizing, we identify the following three distinct behaviors of bending loci and bending measures:
\begin{itemize}
    \item $\mu$ is a minimal lamination with a leaf accumulating non-positive exponential growth (Theorem \ref{thm: dynamical characterize bending}).%\mb{This one should come first.}
    \item Ungemach's examples: $\mu$ is an augmented pants decomposition where the holonomy along each isolated leaf is exponentially decreasing (Proposition \ref{prop: Ungemach}). 
    \item $\mu$ is an orientable lamination augmented by an isolated leaf along which the holonomy is exponentially decreasing (Lemma \ref{lem: spiraling leaves} and Theorem \ref{thm: minimal laminations support measures}). 
\end{itemize}

\subsection{Context}\label{subsec: context and questions}

There are three classical settings of convex cocompact surface groups acting on homogeneous (pseudo)-Riemannian $3$-manifolds, all of which are sub-geometries of $(\SL(4,\R), \RP^3)$:
\begin{enumerate}
    \item  hyperbolic space, $\mathbb H^3$ by isometries in $\SO_\circ(3,1)\cong \PSL(2,\C)$
    \item  co-Minkowski space, $\CoMin^3$ by isometries in $\SO_\circ(2,1)\ltimes (\R^{3})^*$
    \item  anti-de Sitter space, $\AdS^3$ by isometries in $\SO_\circ (2,2)\cong \left(\SL(2,\R)\times \SL(2,\R)\right)/(\Z/2\Z)$
\end{enumerate}

 These three settings are all of independent interest. The first sheds light on the topology and geometry of hyperbolic three-manifolds \cite{Thurston:notes, Thurston:bulletin_3-manifolds}, while the last links more closely to Teichm\"uller theory by way of Mess' beautiful proof of Thurston's Earthquake Theorem \cite{Mess, Thurston:EQ_universal,NotesOnMess,DS:EQ}. The middle sibling, the co-Minkowski space, may be interpreted as an infinitesimal version of the others; Danciger describes how suitable conjugacy limits of shrinking convex cocompact $\mathbb H^3$ or $\AdS^3$ surface group actions limit to convex cocompact $\CoMin^3$ actions \cite{Danciger:transition}.
Every convex cocompact co-Minkowski action is, in particular, one of the coaffine actions studied in this paper.

Away from the Fuchsian locus, the path metrics on the boundary components of the minimal convex domain of discontinuity 
define hyperbolic structures on $S$, and the laminations are endowed with the data of \textit{invariant transverse measures}. The measures record the amount by which the supporting hyperplanes tilt or bend as one passes from stratum to stratum. 
Conversely, a hyperbolic metric and suitable measured lamination identify a unique conjugacy class of convex cocompact actions.

Other data are also known to determine uniquely a conjugacy class of convex cocompact actions in some cases.  For example, Thurston's Bending Conjecture, recently resolved by Dular--Schlenker \cite{DS:bending_conjecture}, states that pairs of compatible bending measures parameterize convex cocompact surface group actions on $\HH^3$.  By work of Bonahon, certain pairs of measured laminations related to Kerckhoff's \emph{lines of minima} parameterize $\CoMin^3$ convex cocompact surface group actions \cite{Bonahon:bending,Kerckhoff:lines}; see also \cite{BF:QFcoMin}.

\subsection{Questions}
It would be interesting to understand how our work relates to the work of Seppi and Ni on surfaces of constant affine Gaussian curvature \cite{SN:1}, as well as to recent work of Antoine Ablondi \cite{Ablondi}, who treats exactly the representations we examine here, but from the dual perspective of differential affine geometry.  

We would be very interested to know what the \emph{typical} behaviors for the bending data are.  Via Main Theorem \ref{maincor: intro affine measure}, there are analogous open questions about AIETs and possible connections to half dilation surfaces (e.g., \cite{Selim:tori,Wang:thesis}).
For example, we would like to know when our affine transverse measures are purely atomic (so that the corresponding AIET  has a wandering interval \cite{AIETflips,Cobo:AIET,MMY:AIET}), which seems to be a dense and open condition.  It is not clear whether or not this set has full Lebesgue measure.
It would be interesting also to identify which pairs of affine laminations appear on the boundary of the convex core of a convex cocompact coaffine surface group action, extending what's known at the Fuchsian locus \cite{Bonahon:bending}.

\subsection{General coaffine convex cocompact representations}\label{subsec: general coaffine}
While working on a follow-up paper, the authors realized that there is a straightforward extension of our results to convex cocompact representations of surface groups into the \emph{general coaffine group}, which is the stabilizer of $e_4$ in $\GL^+(4,\R)$.  That is, 

\[\textrm{Aff}^*(3,\R) = \left\{ 
\begin{pmatrix}
    A & \mathbf{0} \\
    \tau & 1
\end{pmatrix} \colon A \in \GL^+(3,\R) \textrm{ and } \tau \in (\R^3)^*
\right\}.\]

In this case, a convex cocompact coaffine representation is of the form 
\[\eta = 
\begin{pmatrix}
    \rho e^{-\varphi}& \mathbf{0} \\
    Z & 1
\end{pmatrix}
\]
where
\begin{enumerate}
    \item $\rho$ is a Hitchin representation;
    \item $Z:\pi_1 S\rightarrow (\R^3)^*$
satisfies the cocycle condition
\[Z(\gamma_1\gamma_2) = Z(\gamma_1)\rho(\gamma_2) e^{-\varphi(\gamma_2)} +Z(\gamma_2)\]
for all $\gamma_1, \gamma_2 \in \pi_1S$;
\item and $\varphi \in H^1(S,\R)$ satisfies \[\sup_{\gamma\not=id \in \pi_1S} \frac{\varphi(\gamma)}{\lambda_1^\rho(\gamma)}<1,\]
where $\lambda_1^\rho(\gamma)$ is the logarithm of the largest eigenvalue of $\rho(\gamma)$. 
\end{enumerate}
Item 3 is satisfied if and only if $\eta$ is $1$-Anosov; see \cite{Bar10} or \cite{Lahn:ReducibleSuspensions} for similar statements.  The precise formulation we give here is proved in forthcoming work of the authors \cite{BF:BoundaryAnosov}.

\textbf{In this setting,  all results in the paper hold as well with judicious additions of $e^{-\varphi}$.}
Since this generalization only changes the proofs cosmetically, we chose to outline the results here rather than burden the paper notationally. 
We give the more general versions of Main Theorems \ref{maincor: intro affine measure} and \ref{mainthm: MLrho is a sphere} here.
\begin{maintheorem}
    Let $\eta: \pi_1S \to \textrm{Aff}^*(3,\R)$ be a non-semi-simple convex cocompact representation with linear part $L(\eta) = \rho e^{-\varphi}$.  Then bending measures valued in $\overline{E}^{L(\eta)}_2|_\mu$ are affine.

    Furthermore, there exists a natural map 
    \[\beta_+: \{\eta: \pi_1 S\rightarrow \textrm{Aff}^*(3,\R)~\text{non-semi-simple}\mid L(\eta) = \rho e^{-\varphi}\}/\TAff^*(3,\R) \to \ML^{\rho e^{-\varphi}}(S) \] 
    recording the affine bending measured lamination on the top component of the convex hull $\Xi_\eta$ from a coaffine representation with linear part $\rho e^{-\varphi}$. 
    The map $\beta_+$ is a homeomorphism that is homogeneous with respect to positive scale.
    Consequently, the projectivization $\PML^{\rho e^{-\varphi}}(S)$ is a sphere of dimension $6g-7$.
\end{maintheorem}

\begin{remark}
    Recall the holonomy representation of $\overline E^\rho_2|_\mu$ from Theorem \ref{thm: slithering connection intro} is
    \[\chi: \pi_1 \tau\to \R^\times.\]
    If $L(\eta) = \rho e^{-\varphi}$, then the holonomy of  $\overline E^{L(\eta)}_2|_\mu$ is
    \[e^{-\varphi}\chi: \pi_1 \tau\to \R^\times.\]
\end{remark}

In the followup paper, we develop structural features of coaffine representations of the fundamental groups of properly and strictly convex closed projective manifolds on the boundary of the locus of coaffine representations.  We give a detailed account of the geometry of their limit sets in projective space.

\subsection{Outline of the paper}

Section \ref{sec: background} establishes notation, background, and the fundamental objects and structures of study. 
Notably, the basics of coaffine geometry are developed in \S\ref{subsec: coaffine geometry}, and  \S\ref{subsec: Anosov property} establishes notation for Anosov representations, including the flat bundles $F^\rho\to S$ and $E^\rho \to T^1S$, and the relationship between them.

Section \ref{sec: coaffine geometry} explores the data which may be extracted from a convex cocompact coaffine representation $\eta$ through purely geometric techniques. 
Section \ref{sec:equivariant measures} defines spaces  of flow equivariant (transverse) measures supported on a geodesic lamination and gives a dynamical characterization for when a geodesic lamination can support such a measure; 
Theorems \ref{thm: dynamical characterize bending} and \ref{thm: non-orientable minimal intro} are proved in \S\ref{subsec: building equivariant measures}. 
In Section \ref{sec: measures and cocycles}, we construct bending cocycles from transverse measures and vice versa.

Section \ref{section: flat bundles over laminations} is independent from the previous sections and contains results that could be of separate utility. In \S\ref{subsec: tt and flat connections} we develop a definition of flat connections over geodesic laminations extending a given connection along the leaves, and in \S\ref{subsection: Hitchin bundles are flat} we show that the bundle $\overline E^\rho_2\vert_\mu$ is flat by way of constructing the \emph{slithering connection}.
In this section we use $F^\rho_2\to \mu$ to denote $E^\rho_2|_\mu$.

In Section \ref{sec: affine laminations}, we prove our Main Theorems \ref{mainthm: summary}, \ref{maincor: intro affine measure}, and \ref{mainthm: MLrho is a sphere}. 
In \S\ref{subsec: AIET}, we make precise how an affine measure relates to an affine interval exchange transformation (AIET), while \S\ref{subsec: affine laminations} relates our notion of affine measured laminations with the notion studied by Hatcher and Oertel.

Finally, the Appendix is a motivating example. It is designed to be maximally independent from the rest of the paper, except the background material, and is intended to be readable by a proficient graduate student familiar with the basic context.

\section*{Acknowledgements}
The authors would like to thank Jeff Danciger, Steve Kerckhoff, Jane Wang, Rick Kenyon, Fanny Kassel, and Weston Ungemach for interesting discussions related to this work. 
We also extend our gratitude to Lena Coleman for her support and the seed out of which this collaboration blossomed.

The authors appreciate constructive feedback from the referee on an earlier draft.

The first-named author would like to thank IHES, Yale University, and Universit\"at Heidelberg for their hospitality while completing portions of this work.

This project was funded by DFG – Project-ID 281071066 – TRR 191.
The first named author received support from NSF Grant DMS Award No.
2002230.

\setcounter{tocdepth}{2}
\tableofcontents{}

\section{Preliminaries                       }\label{sec: background}

\subsection{Vector spaces}\label{subsec: vector spaces}
Throughout we will work with finite dimensional vector spaces and their dual spaces, often simultaneously. Recall that the general (and special) linear groups of a vector space $V$ act on the dual space $V^*$ by
\[A.\alpha = \alpha \circ A\inverse\]
for $A\in \GL(V)$ and $\alpha\in V^*$. This action is natural: it is the unique action preserving the evaluation between $V$ and $V^*$, and required no additional data. For finite dimensional vector spaces, the groups $\GL(V)$ and $\GL(V^*)$ are therefore naturally identified.

Given a representation $r:\Gamma \rightarrow \GL(V)$, there is thus an action of $\Gamma$ on both $V$ and $V^*$. We denote these vector spaces with their $\Gamma$-actions as $V_r$ and $(V^*)_r$ respectively.

A choice of (ordered) basis $\{e_1,\dots,e_d\}$ for a finite-dimensional vector space $V$ gives a standard choice of basis $\{e^1,\dots,e^d\}$ for $V^*$ by the defining property $e^i(e_j)=\delta^i_j$.

\subsection{Special coaffine geometry}\label{subsec: coaffine geometry}

Let $\Aff^*(3,\R)<\SL(4,\R)$ denote the group of \textit{special coaffine} transformations acting on $\RP^3$. That is, $\Aff^*(3,\R)$ is defined (up to conjugacy in $\SL(4,\R)$) by
\[\Aff^*(3,\R) = \left\{ 
\begin{pmatrix}
    A & \mathbf{0}\\
    \tau & 1
\end{pmatrix}
\mid A\in \SL(3,\R), \tau\in \left(\R^3\right)^*
\right\},\]
where $\mathbf{0}$ denotes the zero vector in $\R^3$. 
The special coaffine group is stabilizer of a point for the action of $\SL(4,\R)$ on $\R^4$. 
In the parametrization used at present, this point is (any non-zero scalar multiple of) $e_4\in \R^4$. The \textit{coaffine space} $\RP^3\setminus \{[e_4]\}$ is a homogenous space for $\Aff^*(3,\R)$. 

\begin{definition}\label{def: coaffine geometry}
    In the language of $(G,X)$-structures, \textit{special coaffine geometry} (in this context, simply \textit{coaffine} geometry), is the pair 
    \[\left(\Aff^*(3,\R), \RP^3 \setminus\{[e_4]\}\right).\]
\end{definition}

There is a homomorphism $\Aff^*(3,\R)\rightarrow \SL(3,\R)$ which takes the matrix $A$ from the definition above; this is the \textit{linear part} of an element of $\Aff^*(3,\R)$. We will say that an element $T\in\Aff^*(3,\R)$ with trivial linear part is a \textit{cotranslation}, and let $\TAff^*(3,\R)$ be the normal subgroup of cotranslations; it is a $3$-dimensional vector space.
There is a short exact sequence
\[1\to \TAff^*(3,\R) \to \Aff^*(3,\R) \to \SL(3,\R) \to 1,\]
which splits, so that 
%The subgroup of translations is normal in $\Aff^*(3,\R)$, the linear part being the quotient to $\SL(3,\R)$. Indeed, the quotient splits, and so $\Aff^*(3,\R)$ is a semi-direct product:
\[\Aff^*(3,\R) \cong \TAff^*(3,\R) \rtimes \SL(3,\R).\]
Then $\SL(3,\R)$ acts naturally on $\TAff^*(3,\R)$ by conjugation.
%we will write $A.\tau = \tau A\inverse$.

The point $[e_4]\in \RP^3$ defines a hyperplane in the dual projective space:
\[[\ker e_4] = \left\{[\alpha] \in \left(\RP^3\right)^* : \alpha(e_4) = 0\right\}.\] 
Denote by 
\[\coaff = \left(\RP^3\right)^* \setminus [\ker e_4],\] 
accompanied by the action by $\Aff^*(3,\R)$, which preserves $\coaff \subset (\RP^3)^*$.
It is isomorphic to the usual $3$-dimensional affine space $\A^3$, i.e., the set of hyperplanes in $\RP^3$ which do not intersect $[e_4]$ form a copy of the $\SL(3,\R)$-affine space under the action of $\Aff^*(3,\R)$. 

Let us choose the vector $e_4 \in \R^4$ as a representative of its projective class, and consider the affine chart for $\coaff$  given by 
\[\left\{\alpha \in \left(\R^4\right)^* : \alpha(e_4) =1\right\}.\]  
Abusing notation, we write $\alpha \in \coaff$ to mean that $\alpha(e_4) =1$ (and therefore $[\alpha]\in (\A^3)^*$).

For $\alpha, \beta\in \coaff$, we have $(\alpha -\beta) \in \ker e_4 \cong (\R^3)^*$.
We may thus view $\ker e_4$ as the (normal) subgroup of cotranslations in $\Aff^*(3,\R)$ by 
\[\tau \in \ker e_4 \mapsto 
\begin{pmatrix}
    I_3 & \mathbf 0\\
    \tau & 1
\end{pmatrix}.\]
For any $\delta \in \R\setminus \{0\}$, the affine chart 
\[\left\{\alpha \in \left(\R^4\right)^* : \alpha(e_4) = \delta\right\}\]
effects our identification of the cotranslation subgroup of $\Aff^*(3,\R)$ via conjugation by \[\begin{pmatrix}
    \delta^{-1/3} I_3 &  \\
     & \delta
\end{pmatrix} \in \SL(4,\R).\]

\begin{remark}\label{rmk: R3 star is TAff}
    Any equivariant identification between $\TAff^*(3,\R)$ and $(\R^3)^*$ requires a choice.  However, we will frequently use `$(\R^3)^*$' as shorthand notation for $\TAff^*(3,\R)$, and `$(\RP^2)^*$' for $\P\TAff^*(3,\R)$. 
\end{remark}

\begin{example}\label{example: coaffine translation}
    To understand the geometric meaning of a cotranslation on the coaffine space, let us use coordinates. For any pair $\alpha\not= \beta\in (\A^3)^*$, the kernels $[\ker \alpha]$ and $[\ker \beta]$ intersect in a projective line $L \subset \RP^3$ which itself does not contain $[e_4]$.  We will exhibit the cotranslation taking $[\ker \alpha]$ to $[\ker \beta]$. See Figure \ref{figure: cone_1}.
    
    Choose a pair of distinct vectors $e_1$ and $e_3$ with projective classes in $L$ as two vectors of a basis of $\R^4$, and let $e_2$ be any vector so that $[e_2]\in [\ker \alpha] \setminus L$. In these coordinates, the unique cotranslation taking $\alpha$ to $\beta$ is
    \[Z_t = 
    \begin{pmatrix}
        1 & 0 & 0 & 0\\
        0 & 1 & 0 & 0\\
        0 & 0 & 1 & 0 \\
        0 & t & 0 & 1
    \end{pmatrix} \in \Aff^*(3,\R),\]
    for some particular $t$. Note also that $\beta-\alpha=te^2$.;

    Conversely, every non-trivial cotranslation $\tau\in \TAff^*(3,\R)  = \ker e_4$ acts by the identity on a unique plane $[\ker \tau] \subset \RP^3$  which passes through $[e_4]$. Consider the subgroup $T = \langle \tau\rangle_\R \le \TAff^*(3,\R)$ generated by the real line through $\tau$.  Then $T$ fixes $[\ker\tau]$ pointwise, and for every line $L$ in $[\ker\tau]$ not passing through $[e_4]$, $T$ acts simply transitively on the set of projective hyperplanes intersecting $[\ker \tau]$ in $L$.
    Up to coordinate transformation taking $\tau$ to $e^2$, $T = \{Z_t: t \in \R\}$.

    % Up to a coordinate change, $\tau = [\ker e^2]$. The subgroup $Z_t$ of $\TAff^*(3,\R)$ fixes $[\ker e^2]$ pointwise. 
    % For each line $L$ in $[\ker e^2]$ which does not contain $[e_4]$, $\{Z_t: t\in \R\}$ acts simply transitively on the set of projective hyperplanes intersecting $[\ker e^2]$ in $L$.

    This is the dual picture to the action of the translation subgroup of the $3$-dimensional affine group acting on $\A^3$. Any pair of points in the affine space differ by a unique translation. Each one-parameter subgroup of the translation subgroup has a unique fixed point at infinity and partitions the affine space into lines on which it acts simply transitively.

    % Choose $L$ any line in this plane which does not intersect $[e_4]$. 
    
    % Up to a coordinate change, we assume that $L$ is the span of $e_1$ and $e_3$. \mb{corrected this paragraph}

    % Consider a positive diagonal element $A\in \Aff^*(3,\R)<\SL(4,\R)$,

    % \[A = \begin{pmatrix}
    %     a_1 & 0 & 0  & 0\\
    %     0 & a_2 & 0 & 0\\
    %     0 & 0 & a_3 & 0\\
    %     0 & 0 & 0 & 1
    % \end{pmatrix}\]
    
    % which preserves the line $L$. Observe (by computation) that the centralizer of $A$ in $\Aff^*(3,\R)$ contains no cotranslation when the linear part of $A$ has eigenvalues all distinct from $1$ (except, necessarily, the eigenvalue of $e_4$). Suppose that the second eigenvalue of $A$ were equal to $1$. Then the centralizer of $A$ contains the subgroup $\{Z_t\}$. This will be relevant later for the deformation theory of coaffine representations of surface groups.
    
\end{example}

\subsection{Surfaces and geodesic laminations}\label{subsec: surface theory}
The reader looking for an introduction to surface theory and geodesic laminations is advised to consult \cite{Thurston:notes, CB, PH, BZ:GL}.
\subsubsection{Negatively curved metrics and flows}
Let $S$ be a closed oriented surface of genus $g$ at least $2$, and let $X$ denote the data of $S$ equipped with a negatively curved Riemannian metric. 
Let $T^1X$ be the unit tangent bundle and let 
\[g_t : T^1X\to T^1X\]
be the time $t$ map for the geodesic flow.
The Levi-Civita connection on $X$ gives us a way to lift the Riemannian metric on $X$ to one on $T^1X$.

Denote by $\partial \widetilde X$ the visual boundary of the universal cover $\widetilde X$ of $X$.
Given a point $x \in \widetilde X$, the exponential map induces a homeomorphsim $T^1_xX \cong \partial \widetilde X$.  
The angular metric on $T^1_xX$ gives $\partial \widetilde X$ a metric, where the identity map on  $\partial \widetilde X$ is a bi-Lipschitz equivalence for the angular/visual metrics for points $x, y \in X$. 

Negative curvature gives that for any pair $z, w \in \partial \widetilde X$ of distinct points there is a unique geodesic line from $z$ to $w$.
The space of geodesic lines is the quotient of $\partial \widetilde X \times \partial \widetilde X \setminus \Delta$ by the order two symmetry swapping the future for the past, where $\Delta$ denotes the diagonal.
%The space of unoriented geodesic lines is the quotient .
Then the family of angular metrics on $X$ gives the space of geodesics a bi-Lipschitz class of metrics.

Suppose $Y$ is another metric on $S$ such that $\widetilde Y$ is a uniquely geodesic visibility space and that the identity map $X\to Y$ is a bi-Lipschitz equivalence.  For example, $Y$ is another negatively curved metric or the Hilbert metric for a properly convex projective structure (see \S\ref{subsec: convex proj}).  The Gromov product can be used to define a H\"older class of metrics on $\partial \widetilde Y$.

The  identity map lifts to a quasi-isometry of universal covers and extends to a bi-H\"older continuous map between $\partial X\cong \partial Y$.
Thus the spaces of (un)oriented geodesic lines in $X$ and in $Y$ have a well defined bi-H\"older class of metrics.
There is also a bi-H\"older continuous 
\begin{equation}\label{eqn: holder orbit equiv}
    \varphi : T^1X \to T^1Y
\end{equation}
orbit equivalence between the two geodesic flows \cite[\S19.1]{KH:dynamicalsystems}.

\subsubsection{Geodesic laminations}\label{subsection: geodesic laminations}

A \emph{geodesic lamination} $\lambda$ on $X$ is a closed subset of $X$ equipped with a foliation by complete geodesics, called its \emph{leaves}.
A geodesic lamination $\lambda$ is \emph{orientable} if its leaves can be continuously oriented.
Every lamination $\lambda$ has a two-fold orientation cover $\hlambda \to \lambda$.  %, where $\hlambda$ is an orientable lamination. 
Concretely, we identify $\hlambda$ with the set of unit tangent vectors to the leaves of $\lambda$, i.e., $\hlambda \cong T^1\lambda\subset T^1X$.
%This cover is not a subset of the surface $S$, but may be realized as a lamination in a double cover of some subsurfaces (see \mb{reference the relevant Bonahon}).
%For our purposes, we view $\hlambda$ as the set of tangents $T^1\lambda \subset T^1X$, with which it is naturally identified.

Examples of geodesic laminations are furnished by disjoint unions of simple, closed geodesics on $X$ and the full preimages of such in $\widetilde X$.
Any family of disjoint complete geodesic lines in $\widetilde X$ that is invariant under the action of $\pi_1S$ by deck transformations descends to a geodesic lamination.
Using the boundary map $\widetilde X \cong \widetilde Y$, we therefore obtain a bijective correspondence between the geodesic laminations on $X$ and the geodesic laminations on $Y$.

A geodesic lamination $\lambda$ is \emph{minimal} if it contains no proper sublamination.
Equivalently, $\lambda$ is minimal if every leaf in $\lambda$ is dense in $\lambda$.
%A geodesic lamination $\lambda$ is \emph{minimal} if any geodesic ray whose image lies in $\lambda$ is dense in $\lambda$. This is the usual dynamical definition of minimality, that tangent vectors to $\lambda$ are dense under the geodesic flow.\mb{adjusted}
For a $C^1$-arc $k$ that is transverse to the leaves of $\lambda$, i.e., a \emph{transversal arc}, if $\lambda$ is minimal and has an infinite leaf, then the $k\cap \lambda$ is a cantor set. 

The structure theory for geodesic laminations gives that $\lambda\subset X$ decomposes as a union of at most $3g-3$ minimal sublaminations and at most $6g-6$ many infinite isolated leaves that spiral onto the minimal components.
We say that $\lambda$ is \emph{maximal} if it is not contained in any other geodesic lamination on $X$.
Every geodesic lamination is contained in a maximal geodesic lamination, called a \emph{maximal completion}.
\medskip

The projective tangent bundle $\mathbb P TX$ is the quotient of $T^1X$ by the fiberwise antipodal involution and carries a metric making the quotient map a local isometry.
Every geodesic lamination $\lambda\subset X$ can be embedded in $\mathbb P TX$ as its tangent  line field $\mathbb P T\lambda$.
The following lemma can be deduced from \cite[Lemma 1.1]{K:NR}.
\begin{lemma}\label{lem: lamination geometry}
    The map $x\in \lambda \mapsto [T_x\lambda]\in \mathbb PT\lambda$ is a bi-Lipschitz homeomorphism.  
\end{lemma}

There is a \emph{Hausdorff metric} $d_X^H$ on the set of geodesic laminations coming from the restriction of the Hausdorff metric on closed subsets of $X$.
This metric gives the set of  geodesic laminations the structure of a compact metric space.
While the orbit equivalence from \eqref{eqn: holder orbit equiv} need not induce a H\"older map $\P TX \to \P TY$, it does however induce a H\"older map between closed invariant sets of the geodesic foliations of $\P TX$ and $\P TY$ with their Hausdorff metrics.
Thus, Lemma \ref{lem: lamination geometry} together with \eqref{eqn: holder orbit equiv} show that the Hausdorff metric on geodesic laminations depends bi-H\"older continuously on $X$ (see also \cite[Lemma 7]{BZ:GL}).

\medskip

Let $\epsilon>0$ be given, and consider the $\epsilon$-neighborhood $\cN_\epsilon(\lambda)\subset X$.  For $\epsilon$ small enough, this neighborhood can be foliated by contractible $C^1$ arcs called \emph{ties} that meet $\lambda$ transversely, for example using Thurston's \emph{horocycle foliation} \cite{Thurston:notes, Thurston:stretch} or the construction of the \emph{orthogeodesic foliation} studied in \cite{CF, CFII}.\footnote{Although both constructions are given only for hyperbolic surfaces, i.e., only when $X$ has constant negative curvature, they both apply in full generality.}
Such a foliated neighborhood is called a \emph{train track neighborhood}.

If every component of $X\setminus \mathcal \cN$ is a deformation retract of the component of $X\setminus \lambda$ containing it, then we say that $\cN$ is a \emph{snug} for $\lambda$.  If $\epsilon$ is small enough, then $\cN$ is snug for $\lambda$.
A snug neighborhood $\cN$ has the property that if $k$ is a tie, and $J$ is a complementary component of $k\setminus \lambda$ not containing an endpoint of $k$, then the leaves corresponding to the endpoints of $J$ are asymptotic in $\lambda$; this can be proved directly from the definition of snugness.
The orientation covering $\hlambda\to \lambda$ extends to a two-fold covering $\widehat \cN \to \cN$ if $\cN$ is snug.

The leaf space of the foliation of $\cN = \cN_\epsilon(\lambda)$ by ties, denoted by $\tau$, has the structure of a \emph{train track}, i.e., a graph with a $C^1$ structure at its vertices (called \emph{switches}) and edges (called \emph{branches}) satisfying a number of non-degeneracy conditions that can be $C^1$-embedded in $\cN$ transverse to the ties.  See \cite[\S1.1]{PH} or \cite[\S8.9]{Thurston:notes}. 
%A lamination $\lambda'$ which is contained in a train track neighborhood $\mathcal N_\epsilon(\lambda)$ meeting all ties transversely is \emph{carried} by $\tau$.  
%Every lamination that is close enough to $\lambda$ in the Hausdorff topology is carried by $\mathcal N_\epsilon(\lambda)$.
%If $\lambda'$ is carried by $\tau$ (denoted $\lambda\prec \tau$), then there is a $C^1$ map $\lambda' \to \tau$ induced by collapsing the ties.
%The orientation double cover $\widehat \lambda \to \lambda$ induces a double cover $\widehat {\mathcal N} \to \mathcal N$.
%Train tracks are useful for a number of reasons when working with geodesic laminations.

%The following can be used to show that every geodesic lamination on a closed surface has Hausdorff dimension 1.  
For topological reasons, the $2$-dimensional Lebesgue measure of a geodesic lamination on a closed surface is zero.
For a $C^1$ arc $k$ transverse to $\lambda$, Fubini then implies that the $1$-dimensional Lebesgue measure of $k\cap \lambda$ is zero.
A proof of the following lemma can be adapted from \cite[Proposition 4.1]{Farre:Hamiltonian} or from the arguments in \cite[\S1]{Bon:SPB}.
The lemma implies also that the Hausdorff dimension of $\lambda$ is $1$; see also \cite{BS}.
%It uses the fact that the intersection with a transversal to a geodesic laminations on closed hyperbolic surfaces has zero $1$-dimensional Lebesgue measure.  %Hausdorff dimension $1$ \cite{BS}.
\begin{lemma}\label{lem: tt geometry}
    For any given $C^1$ arc $k$ transverse to $\lambda \subset X$ and $\epsilon>0$, the connected components of $k\cap \cN_\epsilon(\lambda)$ each have diameter $O(\epsilon)$, and there are at most $O(\log1/\epsilon)$ such components.
\end{lemma}

%The geodesic flow on $T^1X$ preserves $\hlambda$, and $\hlambda$ contains the support of a flow invariant finite Borel measure $\mu$ on $T^1X$.
%Denote by $\ML(T^1X)$ the space of finite flow invariant Borel measures on $T^1X$ that are invariant under the antipodal involution of $T^1X$ and whose support is tangent to a geodesic lamination.
%Then $\ML(T^1X)$ is equipped with the weak-$*$ topology. 

A \emph{transversely measured geodesic lamination} $\mu$ with support $\lambda$ is the assignment of, for every arc $k$ transverse to $\lambda$, a finite positive Borel measure $\mu_k$ with support equal to $\lambda\cap k$.  The assignment $k \mapsto \mu_k$ is required to be natural under restriction and to be slide invariant, i.e., if $k'$ is homotopic via $H$ to $k$ through arcs transverse to $\lambda$, then $H_*\mu_k = \mu_{k'}$.
If $\lambda$ has isolated leaves, then it is not the support of a transverse measure.
Thurston's space $\ML(S)$ of measured geodesic laminations is a manifold whose positive projectivization  $\PML(S)$ is a sphere of dimension $6g-7$ \cite{Thurston:bull}.
%Certain train tacks can be used to give coordinate charts for $\ML(S)$.
Thurston constructed a natural measure (called the Thurston measure) on $\ML(S)$ in the class of Lebesgue.
%Via disintegration,  $\PML(S)$ is homeomorphic to the space of probabilty measures in $\ML(T^1X)$; see \cite{} or Proposition \ref{prop: disint is equiv}, below.

A minimal geodesic lamination $\lambda$ is called \emph{uniquely ergodic} if the simplex of measures in $\ML(S)$ whose support is $\lambda$ is a ray.
For the Thurston measure, the support of almost every measured lamination is maximal, minimal, non-orientable, and uniquely ergodic \cite{Ker:ss}.

\subsection{Hitchin representations and convex $\RP^2$-geometry}\label{subsec: convex proj}

We require some general facts about the Hitchin component of surface group representations into $(\mathrm P)\SL(3,\R)$, and the geometry of convex projective structures on the surface $S$. For the uninitiated, there are two helpful surveys on convex projective structures, one by Choi, Lee, and Marquis \cite{CLM} and one by Benoist \cite{Benoist_survey}.

\medskip
A subset of the real projective space $\RP^d$ is \textit{properly convex} if its closure is contained in an affine chart and it is convex in such a chart.
A \textit{(real, properly) convex projective structure} on a surface $S$ is an (incomplete) $(\SL(3,\R), \RP^2)$-structure so that the developing map is a homeomorphism to some open properly convex subset $\Omega\subset \RP^2$. 

Goldman and Choi demonstrated that the set of holonomies (considered up to conjugacy) of convex projective structures on a surface form a component of the $\SL(3,\R)$-character variety. This component is usually called the \textit{Hitchin component} and  is denoted $\Hit_3$ \cite{Choi-Goldman, Labourie:Anosov}. 

A hyperbolic structure on a surface is therefore an example of a convex projective structure: the developing map of such a structure is a homeomorphism to $\HH^2$, which may be realized as the projective (Klein) model. This is a properly convex domain in $\RP^2$. Any properly convex domain is a metric space with the \textit{Hilbert} metric; the projective lines are geodesics for this metric.

For a representation $\rho$ with $[\rho]\in \Hit_3$, the domain $\Omega=\Omega(\rho)$ is strictly convex and the boundary $\partial \Omega$ has regularity $C^{1+a}$  for some $a\in (0,1]$ \cite{Benoist_CDI}. It is true in general that if a domain $\Omega$ is $C^1$ and strictly convex, then $\Omega^*$ is as well ($C^1$ and strict convexity are dual notions).

As in the case of hyperbolic geometry, there is a $(\pi_1 S,\rho)$-equivariant H\"older-continuous identification 
\begin{equation}\label{eqn: boundary map}
    \xi=\xi(\rho): \partial \pi_1 S \to \partial \Omega\subset \RP^2.
\end{equation}

In the dual projective space $(\RP^2)^*$, the action of $\rho$ preserves a convex projective domain $\Omega^*$, which is defined by
\[\Omega^* = \{[\alpha]\in (\RP^2)^* \mid [\ker \alpha] \cap \overline \Omega = \emptyset\}.\]
It is not difficult to check that $\Omega^*$ is a properly convex projective domain, and so also is furnished with a limit map 
\[\xi^*=\xi^*(\rho): \partial \pi_1 S \to \partial \Omega^*\subset (\RP^2)^*.\]
The maps $\xi$ and $\xi^*$ have the property that for all $z\in\partial\pi_1 S$, $\xi^*(z)(\xi(z))=0$, i.e., the hyperplane $\xi^*(z)$ is the unique supporting hyperplane to $\Omega$ at the point $\xi(z)$.

\subsection{The Anosov property}\label{subsec: Anosov property}

We require some of the theory of Anosov representations from a closed surface group into $\SL(d,\R)$ following Labourie \cite{Labourie:Anosov}.
For further development of the theory of Anosov representations, see \cite{GW12, BPS19, KLP17}.

Let $\rho: \pi_1 S \to \SL(d,\R)$ be a representation.
We construct a flat $(\R^d)^*$-bundle over $S$ by pushing forward the flat connection on the trivial bundle $\widetilde S \times (\R^d)^*$ to the quotient by the diagonal action:
\[F^\rho = \left(\widetilde S \times (\R^d)^*\right) / \left[(x, v) \sim (\gamma.x, \rho(\gamma). v)\right].\]
%\mb{we should maybe write $\rho(\gamma).v$? it isn't left matrix multiplication... let's discuss. }
The flat connection on $F^\rho$ is called $\nabla^\rho$.

Choose for reference a hyperbolic metric $X$ on $S$ and a continuously varying inner product on $F^\rho$, and consider the flat bundle $E^\rho \to T^1X$ obtained by pulling back $F^\rho$ along the projection $\pi: T^1X \to X \cong S$.  By abuse of notation, we denote also the flat connection on $E^\rho$ by $\nabla^\rho$.
To define the Anosov property, we consider the parallel transport along flow lines for the geodesic flow.

\begin{definition}\label{def: Anosov}
    We say that $\rho$ is \emph{Anosov} if there are constants $c_1, c_2>0$ a H\"older continuous splitting $E^\rho_1\oplus ... \oplus E^\rho_k$ of $E^\rho$ such that for $i < j$ we have 
    \[\frac{\|\left(g_{[0,t]}\right)_* v_i \|}{\|\left(g_{[0,t]}\right)_* v_j \|}\le c_2e^{-c_1t}, \]
    for all $t\ge 0$ and unit vectors $v_i\in E_i^\rho$ and $v_j \in E_j^\rho$, where $\left(g_{[0,t]}\right)_*$ is the parallel transport along the segment $g_{[0,t]}$.
    When $k=d$, $\rho$ is \textit{Borel}-Anosov. When $\dim E^\rho_1 =1$, $\rho$ is \textit{projective}-Anosov.
\end{definition}

For a representation into $\SL(3,\R)$, a representation is projective-Anosov if and only if it is Borel Anosov.

There is an analogously defined bundle\footnote{The minor annoyance of having an $\R^d$-bundle decorated with a `$*$' is made up for by how often we will work with $E^\rho$.} $(E^\rho)^*$  constructed from the natural action of $\rho$ on $\R^d$. Its Anosov-splitting is dual to the Anosov splitting of $E^\rho$ for dynamical reasons.

The antipodal involution $p\mapsto \overline p$ satisfies $g_t\overline p = \overline{g_{-t}p}$, and so induces a symmetry 
\[E^\rho_i|_p = E^\rho_{k-i}|_{\overline p},\]
so that the dimensions of the factors in the splitting match up in pairs: $\dim E^\rho_i = \dim E^\rho_{k-i+1}$.

\begin{remark}
    Anticipating what is to come in \S\ref{sec:equivariant measures}, we point out that if $\rho$ is a representation to $\SL(3,\R)$ that is Borel Anosov, and if $v\in E^\rho$ satisfies
    \[\limsup_{t\to \infty} \frac{1}{t}\log \left\|\left(g_{[0,t]}\right)_*v\right\| =0,\]
    then $v\in E_{2}^\rho$.
\end{remark}

Because the splitting $(E^\rho)^* = (E^\rho_1)^*\oplus \dots \oplus (E^\rho_k)^*$ is $g_t$-invariant, it provides a \textit{limit curve} with image in $\RP^{d-1}$ when $\rho$ is projective- (or Borel-) Anosov:
\[\xi: \partial \pi_1 S \rightarrow \RP^{d-1}, \]
by assigning to $z\in \partial \pi_1 S$ the projective class of $(E^\rho)^*_k|_p$ for any $p\in T^1X$ satisfying $\lim_{t\rightarrow\infty}(g_t p) =z$.

The map $\xi$ is $\rho$-equivariant, and is the same curve from Equation \eqref{eqn: boundary map} in the case that $[\rho]\in \Hit_3$. 

We will take a special interest in the bundle $E^\rho_2$ for $[\rho]\in \Hit_3$. The following lemma will be helpful in the constructions of Section \ref{sec:equivariant measures}.

\begin{lemma}\label{lemma: E_2 has sections}
    Let $\rho:\pi_1 S\to \SL(3,\R)$ be so that $[\rho]\in \Hit_3$, and let $E^\rho=E^\rho_1\oplus E^\rho_2 \oplus E^\rho_3$ be the Anosov splitting of the associated flat bundle. Then the line-bundle $E^\rho_2$ admits a global H\"older continuous (non-zero) section.
\end{lemma}

\begin{proof}
    % Because the existence of the section is a topological statement, it is sufficient  (by the Thurston-Ehresmann principle) to demonstrate it in the case that $\rho$ is Fuchsian: when $\rho(\pi_1 S)<\SO(2,1)<\SL(3,\R).$ 

    Since the image of the holonomy is in the special linear group, the bundle $E^\rho$ is orientable (i.e. the associated orientation bundle admits a global section, $o$). Due to Choi-Goldman \cite{Choi-Goldman}, there is a properly convex closed cone $C_v$ in every fiber of $E^\rho\vert_v$. These cones are locally canonically identified by the flat section, so we call them all $C$. %The boundary of $C$ is a cone on a topological circle with H\"older regularity. 
    
    Since the sub-bundles $E^\rho_1$ and $E^\rho_3$ lie in the boundary of $C$, they are globally oriented: choose the direction that lies in $C$ and take the unit vector with respect to the chosen norm on $E^\rho$. Call these sections $x$ and $z$ respectively. These orientations $E^\rho$ co-orient $E^\rho_2$, define a section by insisting that $y\in E^\rho_2$ is a unit vector satisfying
    \[[x\wedge y \wedge z] =o\]
    at every point. 
    The regularity statement follows from the regularity of the Anosov splitting; the lemma is proven.
\end{proof}

\subsection{Cohomology with coefficients}\label{subsec: cohomology}
Our primary reference for this material is \cite{JM:deformations}.  %The modern theory (for the dual) seems to go back at least to Steenrod \cite{Steenrod:homology}.
Let $\Gamma$ be a group and $M$ be a connected CW-complex with fundamental group isomorphic to $\Gamma$, let $V$ be a vector space, and let $\rho: \Gamma \to \GL(V)$ be a representation.

%The cohomology of $\Gamma$ with coefficient module $V$ equipped with its $\rho$-action appears, e.g., when studying the space of infinitesimal deformations of $G$-conjugacy classes of representations of $\Gamma$ into a Lie group $G$; in this case, $V$ is the Lie algebra $\mathfrak g$ and $\rho$ is the adjoint representation of some $\eta: \Gamma \to G$.

We will only be interested in the $1$-dimensional cohomology defined as follows.
The $1$-cocycles, denoted $Z^1(\Gamma, V_\rho)$ are functions $\varphi: \Gamma \to V$ satisfying the cocycle relation
\[\varphi(\alpha \beta) = \varphi(\alpha) + \rho(\alpha)\varphi (\beta), ~ \forall \alpha, \beta \in \Gamma.\]
See Definition \ref{def: coaffine rep} for examples of a $1$-cocycles.  
The coboundaries $B^1(\Gamma, V_\rho) \le Z^1(\Gamma, V_\rho)$ are cocycles $\varphi$ for which there exist some $v \in V$ 
satisfying \[\varphi(\gamma) = \rho(\gamma) v -v, ~ \forall \gamma \in \Gamma.\]
The cohomology $H^1(\Gamma,V_\rho)$ is the vector space quotient of cocycles by coboundaries.
%The quotient vector space 
%\[H^1(\Gamma, V_\rho) = \frac{Z^1(\Gamma, V_\rho)}{B^1(\Gamma, V_\rho)}\]
%is the cohomology.

Let $\widetilde M$ be the universal cover of $M$.
There is a flat $V$-bundle $F^\rho \to M$ constructed by pushing forward the (trivial) flat connection on $\widetilde M \times V$ to the quotient by the diagonal action:
\[F^\rho = (\widetilde M \times V)/ [(x, v) \sim (\gamma.x, \rho(\gamma)v)].\]
Now we discuss the cohomology of $M$ with coefficients in $F^\rho$.
Denote by $[0,1,2]$ the standard oriented $2$-simplex, so that, algebraically, we have 
\[\partial [0,1,2] = [1,2]-[0,2]+[0,1].\]
The $1$-cochains $\psi$ assign to each singular $1$-simplex $\sigma: [0,1] \to M$ an element $\psi(\sigma)$ in the fiber of $F^\rho$ over $\sigma(0)$.
A $1$-cochain $\psi$ is a cocycle if for every singular $2$-simplex $\tau : [0,1,2] \to M$, we have
\[\psi(\tau|_{[0,2]}) = \psi(\tau|_{[0,1]}) + \overline{\left(\tau|_{[0,1]}\right)}_*\psi(\tau|_{[1,2]}),\]
where $\overline{\left(\tau|_{[0,1]}\right)}_*: F^\rho|_{\tau(1)} \to F^\rho|_{\tau(0)}$ is the parallel transport. 
We say that $\psi$ is a $1$-coboundary if there is a (set theoretic) section $s: M \to F^\rho$ such that 
\[\psi(\sigma) = \overline \sigma_* s(\sigma(1))  -  s(\sigma(0))  \]
for all singular $1$-simplices $\sigma: [0,1] \to M$.
The space of cocycles is denoted by $Z^1(M,F^\rho)$ and the subspace of coboundaries is $B^1(M,F^\rho)$.  The cohomology $H^1(M,F^\rho)$ is the vector space quotient of cocycles by coboundaries.

An identification of $\Gamma$ with $\pi_1(M,p)$ determines an isomorphism
\begin{equation}\label{eqn: cohomology iso}
    H^1(M,F^\rho) \to H^1(\Gamma, V_\rho)
\end{equation}
by evaluating cocycles only on loops based at $p$.

\section{Coaffine geometry and convex cocompact representations}
\label{sec: coaffine geometry}

In this section, we study convex cocompact representations $\eta: \pi_1S \to \Aff^*(3,\R)$ and show that the linear part $L(\eta) = \rho$ is Hitchin (Lemma \ref{lemma: cccca have Hitchin linear part}). The converse, that every coaffine representation with Hitchin linear part acts convex cocompactly, is Lemma \ref{lemma: def is cccca}. The primary geometric object associated to these representations is their minimal convex domain of discontinuity, introduced in \S\ref{subsec: minimal domain coaffine}. 

The main result of this section is that the boundary of this minimal domain hosts a pair of geodesic laminations on the underlying surface $S$ (Theorem \ref{thm: pleated disks}) and that there exist macroscopic bending data which is transverse to these laminations (Lemma \ref{lemma: psi is phi}); the values that the bending cocycle attains are quantitatively close in a dual projective plane to the dual of the lamination (Proposition \ref{prop: cocycles localize}).

\subsection{Coaffine Hitchin representations}\label{subsection: coaffine gp and Hit}

We begin by showing that every convex cocompact special coaffine representation acting on $\RP^3$ has Hitchin linear part.
There is a natural projection $Q: \R^4\rightarrow \R^4 / [e_4]$, inducing a map 
\begin{equation}\label{eqn: def [Q]}
    [Q]: \RP^3\setminus \{[e_4]\}\rightarrow \P(\R^4/[e_4])\cong\RP^2.
\end{equation}
The map $Q$ is equivariant with respect to taking the linear part of a transformation $A\in \Aff^*(3,\R)$, i.e.
\[L(A)\circ Q =  Q \circ A.\]
Dual to $Q$, there is an inclusion $Q^*:\TAff^*(3,\R) \cong \ker e_4\hookrightarrow \left(\R^4\right)^*$, inducing \[[Q^*]: \P\TAff^*(3,\R) \cong \left(\RP^2\right)^* \hookrightarrow \left(\RP^3\right)^*.\]
The action of $\Aff^*(3,\R)$ preserves $\P\TAff^*(3,\R)$. 

\begin{lemma}\label{lemma: cccca have Hitchin linear part}
    Suppose $\eta: \pi_1 S \to \Aff^*(3,\R)$ is faithful and acts convex cocompactly on $\RP^3$. Then the linear part $\rho=L(\eta)$ is Hitchin.
\end{lemma}

\begin{proof}
    Since $\eta$ is faithful, acts convex cocompactly, and $\pi_1S$ is hyperbolic, $\eta$ is projective-Anosov \cite{DGK}. 
    The point $[e_4]$ cannot be in the image of the Anosov limit map $\xi$, as $\xi$ is equivariant and dynamics preserving. Indeed, if $[e_4]$ were in $\im \xi$, the action of $\pi_1S$ on $\partial \pi_1S$ would have a global fixed point, which would be a contradiction.

    By assumption, the action of $\eta$ on the minimal convex domain 
    \[\Xi =\left(\hull \im \xi\right) \setminus \im \xi\]
    is proper, so $\Xi$ also does not contain the point $[e_4]$. As a result, there exist a hyperplane $\alpha\subset \RP^3$ separating $[e_4]$ and $\Xi$.
    
    The action of $\eta$ on $(\R^4/[e_4])$ induced by the quotient $Q$ is $\rho = L(\eta)$. The representation $\rho$ is discrete and faithful for the following reason. If a sequence of distinct group elements $\gamma_n\in\pi_1 S$ existed so that $\lim_{n\to \infty} \rho(\gamma_n)=I$, then the cotranslation part of $\eta(\gamma_n)$ tends to infinity, as $\eta$ is discrete. In this case, the orbit of some point $x\in \Xi$ tends to $[e_4]$ under $\eta(\gamma_n)$, contradicting the fact that $\Xi$ is invariant and bounded away from $[e_4]$. This shows that $\rho$ is discrete, and because $\pi_1S$ has no finite normal subgroups, $\rho$ is faithful by the same argument.
    
    Since $\rho$ is discrete, faithful, and preserves an open proper convex projective domain 
    \[\Omega = [Q](\Xi) \subset \RP^2,\]
    the quotient is homeomorphic to $S$. This is a properly convex projective structure on $S$, and thus $\rho$ is Hitchin by \cite{Choi-Goldman}.
\end{proof}

To reiterate, since $\rho$ is Hitchin, it preserves a properly convex $C^1$ and strictly convex projective domain 
\[\Omega \subset \P(\R^4/[e_4])\] 
and its dual action preserves the dual properly convex $C^1$ and strictly convex domain
\[\Omega^* \subset \P\TAff^*(3,\R).\]

In light of Lemma \ref{lemma: cccca have Hitchin linear part}, we turn our attention to representations of the following prescribed form. Recall from Remark \ref{rmk: R3 star is TAff} that $(\R^3)^*$ means $\TAff^*(3,\R)$.
\begin{definition}\label{def: coaffine rep}
    Let $\rho:\pi_1 S \rightarrow \SL(3,\R)$ be a Hitchin representation, and let $\varphi\in Z^1(\pi_1 S,(\R^3)^*_\rho)$. Define $\eta = \eta_\varphi:\pi_1 S \rightarrow \Aff^*(3,\R)$ by
\begin{equation}\label{eqn: coaffine reps}
    \eta(\gamma) = 
\begin{pmatrix}
    \rho(\gamma) & \mathbf 0\\
    \varphi(\gamma^{-1} ) & 1
\end{pmatrix}.
\end{equation}
    Then $\eta$ is a \emph{coaffine representation with Hitchin linear part}.
\end{definition}

One checks that this is a representation using the cocycle condition; Lemma \ref{lemma: cccca have Hitchin linear part} states that all convex cocompact coaffine representations are of this form.

\begin{remark}\label{remark: cocycles make reps}
    It is a straightforward matrix computation to see that if $\eta$ is a representation of the block form 
    \[\eta(\gamma)=\begin{pmatrix}
    \rho(\gamma) & \mathbf 0\\
    Z(\gamma) & 1
    \end{pmatrix}.\]
    then necessarily $Z(\gamma) = \varphi(\gamma\inverse)$ for some cocycle $\varphi$ as in Definition \ref{def: coaffine rep}.
\end{remark}

\medskip
There is a relationship between semi-simple representations and coboundaries in this setting. 
Recall from \S\ref{subsec: coaffine geometry} that 
\[(\A^3)^*= \{[\alpha]\in (\RP^3)^* \mid \alpha(e_4)=1\}.\]
When $\eta$ is semi-simple and of the form in \eqref{eqn: coaffine reps}, it preserves a hyperplane $[\ker \alpha ]$ for some $\alpha \in \coaff$. Call such a representation $\eta_\alpha$. Conversely, every hyperplane $[\ker\alpha]$ which does not contain $[e_4]$ defines a linear section of the quotient $\R^4\rightarrow \R^4/[e_4]$, and an accompanying semi-simple representation $\eta_\alpha$ preserving this hyperplane. Note that the natural projection $Q: \R^4\rightarrow \R^4/[e_4]$ restricted to $[\ker\alpha]$ is a linear isomorphism conjugating $\eta_\alpha$ to $\rho$, i.e.,
\[ Q\circ\eta_\alpha =  \rho \circ Q.\]

For any pair of semi-simple representations $\eta_\alpha$ and $\eta_\beta$, the difference $\tau_{\beta\alpha} = \beta-\alpha\in \ker e_4$ is an element of the cotranslation subgroup of $\Aff^*(3,\R)$, and
\[\tau_{\beta\alpha}\eta_\alpha (\tau_{\beta\alpha})^{-1} = \eta_\beta.\]

\begin{lemma}\label{lemma: coboundaries and reducible reps}
    $\eta_\varphi$ is semi-simple if and only if $\varphi\in Z^1(\pi_1 S,(\R^3)^*_\rho)$ is a coboundary, i.e., $\varphi \in B^1(\pi, (\R^3)^*_{\rho})$. 
\end{lemma}

\begin{proof}
    Recall from the definition that $\varphi \in B^1(\pi, (\R^3)^*_{\rho})$ when there exists some $\tau\in (\R^3)^* = \ker(e_4)$ such that for all $\gamma\in \pi_1 S$,
    \[\varphi(\gamma)= \tau- \rho(\gamma).\tau \]
    One checks by a direct matrix computation that in this case, $\eta_\varphi$ preserves the hyperplane $[\ker(\tau+ e^4)]$:
    \[ \eta_\varphi(\gamma).(\tau+e^4) = 
    (\tau + e^4) 
    \begin{pmatrix}
        \rho(\gamma)\inverse & 0 \\
        (\tau - \tau \rho(\gamma)\inverse) & 1
    \end{pmatrix}
    = \tau + e^4
    \]

    The converse is a similar computation.
\end{proof}

One perspective on Lemma \ref{lemma: coboundaries and reducible reps} is that the space of coboundaries $B^1(\pi_1 S,(\R^3)^*_\rho)\cong \TAff^*(3,\R)$ is a vector space acting simply transitively (as an abelian group) on the set of semi-simple representations with linear part $\rho$, by way of the fact that such a semi-simple representation is uniquely determined by the hyperplane it preserves.

The correspondence given above passes to cohomology.
\begin{lemma}\label{lem: cohomology parameterizes reps up to conj}
    The vector space $H^1(\pi_1 S,\TAff^*(3,\R)_\rho)$ acts simply transitively (as an abelian group) on 
    \[ \{\eta: \pi_1 S\rightarrow \Aff^*(3,\R)\mid L(\eta)\in\SL(3,\R).\rho\}/\Aff^*(3,\R).\]
\end{lemma}

\begin{proof}
    The centralizer of $\rho$ in $\SL(3,\R)$ is always trivial. In particular if $\rho' = g \rho g\inverse$ for some unique $g\in\SL(3,\R)$, there is an induced isomorphism between the cohomology groups 
    \[H^1(\pi_1 S,\TAff^*(3,\R)_\rho)\cong_g H^1(\pi_1 S,\TAff^*(3,\R)_{\rho'}).\]
    Therefore it is sufficient to consider the case when $L(\eta)=\rho$.
    
    First, we claim that the space of cocycles $Z^1(\pi_1 S,\TAff^*(3,\R)_\rho)$ acts simply transitively on the set 
    \[ \{\eta: \pi_1 S\rightarrow \Aff^*(3,\R)\mid L(\eta)=\rho\},\]
    the action being by addition of the cocycle to the cotranslation part. It is then a computation to check that $\TAff^*(3,\R)$-conjugacy affects $\eta$ by adding a coboundary to the cotranslation part, so that the action of $H^1(\pi_1 S,\TAff^*(3,\R)_\rho)$ on $\TAff^*(3,\R)$-conjugacy classes is well-defined and transitive. 
\end{proof}

\begin{figure}[h]
\includegraphics[width=.95\textwidth]{cone_1.pdf}
\centering
\caption{The cone $\mathscr C \Omega\subset \RP^3$ with two hyperplanes intersecting it. The line $\ell$ is shown in $\P(\R^4/[e_4])$, and the dual data in $\Omega^*\subset\P(\TAff^*(3,\R))$.}
\label{figure: cone_1}
\end{figure}

Consider $\mathscr C\Omega = [Q]\inverse(\Omega)$, which is an open cone in $\RP^3\setminus \{[e_4]\}$.
Topologically, the closure of this cone in $\RP^3$ is the one point compactification of the product $\overline{\Omega}\times \R$, the exceptional point $[e_4]$ being the compactifying point. 
See Figure \ref{figure: cone_1}.

Since $\Omega\subset\P(\R^4/[e_4])$, $\Omega$ is actually \textit{equal} to the set of lines which pass through $[e_4]$ and which intersect the cone $\mathscr C \Omega$. Let us call such lines the \textit{vertical} lines in $\mathscr C \Omega$.

Choosing an orientation on a hyperplane $[\ker \alpha]$ where $\alpha \in (\A^3)^*$ amounts to choosing a lift of $[\alpha]$ which evaluates either positively or negatively on $e_4\in [e_4]$. %, our fixed choice of representative for the fixed point of $\Aff^*(3,\R)$. 

We consistently oriented $[\alpha]\in (\A^3)^*$ by insisting that $\alpha(e_4)=+1$. Call these the (positively) \textit{oriented} hyperplanes. Observe that the group of cotranslations $\TAff^*(3,\R)$ acts simply transitively on the positively oriented hyperplanes. The vertical lines are then oriented to point into the positive half-space of every positively oriented hyperplane.

\begin{remark}
Choose an affine chart for $\RP^3$ containing $\mathscr C \Omega$; $[e_4]$ is at infinity in such a chart. The cone $\mathscr C \Omega$ is a cylinder: $\Omega \times \R$ which is more evident in this chart. An orientation on this chart and on the $\R$-factor co-orients the hyperplanes which intersect $\mathscr C \Omega$.
This cylinder is the generalization of Danciger's `half-pipe' model for co-Minkowski space \cite{Danciger:transition} (it is exactly this when $\Omega$ has boundary an ellipse). Ungemach calls this space $\mathbf{HP}(\Omega)$ \cite{Ungemach:thesis}.
\end{remark}

\subsection{The minimal domain of a convex cocompact coaffine representation}\label{subsec: minimal domain coaffine}

This subsection is devoted to describing the geometry of the action of $\eta$ on $\RP^3\setminus \{[e_4]\}$, which (from our perspective) is governed by its minimal convex domain of discontinuity.

%Recall from \S\ref{subsec: convex cocompact} the definition of convex cocompactness for surface groups acting by projective transformations on $\RP^3$.

\begin{lemma}\label{lem: eta is Anosov}
    For any $\varphi \in Z^1(\pi_1 S, (\R^3)^*_{\rho})$, the representation $\eta$ defined in \eqref{eqn: coaffine reps} is projective Anosov and the limit curve $\im \xi$ is contained in $\partial \mathscr C \Omega \setminus \{[e_4]\}$.
    %and acts convex cocompactly on $\RP^3$. 
\end{lemma}

\begin{proof}
    Consider the path of representations $\eta_{t \varphi}$ for $t\in[0,1]$. When $t=0$, $\eta_0$ is semi-simple; its Anosov limit curve is contained in the hyperplane $[\ker(e^4)]\subset \RP^3$ and is the boundary of the domain $\Omega$ inside of this embedded $\RP^2$.
    Therefore the Lemma holds in this case.

    Since the property of being projective-Anosov is open, there exists some $\epsilon>0$ so that the representation $\eta_{t\varphi}$ is projective-Anosov for all $t<\epsilon$. Since the limit map $\xi$ varies continuously in the representation, for sufficiently small $\epsilon$, $\im \xi$ is contained in $\partial \mathscr C \Omega \setminus \{[e_4]\}$.
    
    Observe that the conjugate of $\eta_{t\varphi}$ by the element 
    $\textrm{diag}(s^{-1/2},s^{-1/2},s^{-1/2}, s^{3/2})\in \SL(4,\R)$ is equal to $\eta_{(st)\varphi}$. Therefore the Lemma holds for arbitrary $\varphi \in Z^1(\pi_1 S, (\R^3)^*_{\rho})$.
\end{proof}

There is a minimal convex domain of discontinuity for the action of $\eta$ on $\RP^3$.
\[\Xi=\hull(\im(\xi))\setminus \im (\xi).\]
%where $\Omega$ is an open domain in $\RP^3\setminus \im\xi^1$. 
Then $\Xi$ can also be realized as the intersection of all convex domains of discontinuity for the action of $\eta$ on $\RP^d$.

\begin{lemma}\label{lemma: def is cccca}
    When $\eta$ is not semi-simple, the action of $\eta(\pi_1 S)$ on $\Xi$ is cocompact with quotient homeomorphic to $S\times [0,1]$.
\end{lemma}

\begin{proof}
    It follows from Lemma \ref{lem: eta is Anosov} that $[e_4]$ is not contained in $\Xi$. Therefore there exist a pair of hyperplanes $H_1$ and $H_2$ bounding $\Xi$ from $[e_4]$. The projection $[Q]$ is $(\eta,\rho)$-equivariant, so the action of $\eta$ on $\Xi$ is proper.
    
    Since $\Xi$ is convex and not contained in a hyperplane, the pre-image of a point $p\in \Omega$ is a closed interval. Since $\eta$ preserves the orientations on the vertical lines of $\mathscr C \Omega$, this proves that 
    \[\Xi/\eta \cong \left(\Omega/ \rho\right) \times [0,1].\]
    The Lemma is complete.
\end{proof}

This is the converse of Lemma \ref{lemma: cccca have Hitchin linear part}; convex cocompact coaffine representations and coaffine representations with Hitchin linear part are the same objects.

\begin{remark}
    Coaffine representations with Hitchin linear part are also \emph{strongly convex cocompact} in the sense of \cite[Definition 1.1]{DGK}. One should generally be concerned about the notions of `naive' and `strong' projective cocompactness, as discussed in \cite{DGK}. Since surface groups are non-elementary Gromov hyperbolic groups, the two notions coincide.  See Theorem 1.15 of that paper.
\end{remark}

%For the remainder, let $\Xi = \Xi(\eta)$ denote the minimal convex domain of discontinuity for $\eta$ acting on $\RP^3$.
Henceforth we fix $\eta$ and study the structure of $\Xi$.
The following definition describes the sort of object that may be found naturally bounding $\Xi$.

\begin{definition}\label{def: topological pleated disk}
    A \textit{bent domain} is a triple $(D,\lambda, h)$ of a domain $D\subset \RP^2$ with strictly convex boundary, a geodesic lamination (see \S\ref{subsec: surface theory}) $\lambda$ on $D$, and a continuous map $h:D\rightarrow \RP^3$ which satisfies
    \begin{itemize}
        \item For every leaf $\ell \subset \lambda$, $h\vert_\ell$ is a projective isomorphism to its image,
        \item For every connected component $P\subset D\setminus \lambda$, the map $h\vert_P$ is a projective isomorphism to its image,%$\im(h\vert_P)$ is a proper subset of a projective plane in $\RP^d$.
        \item The map $h$ has a continuous extension to $\partial\lambda\subset \partial D$.
    \end{itemize}
\end{definition}

\begin{remark}    
    The image of a bent domain does not have any obvious (intrinsic) global projective structure. 
\end{remark}

We will encounter maps which are simultaneously bent domains and sections of the cone $\mathscr C \Omega \rightarrow \Omega$. In particular, we wish to characterize when such maps have convex graphs in $\mathscr C \Omega$.

Because the vertical lines in $\mathscr C \Omega$ are oriented, it is coherent to discuss the positive and negative sides of a continuous section 
\[q:\Omega \rightarrow \mathscr C \Omega.\]
Such a section divides $\mathscr C \Omega$ into two connected components, and any vertical line $L$ is oriented into the positive half of $\mathscr C \Omega \setminus \im q$.

This gives a concise definition of convexity in the cone $\mathscr C \Omega$ that matches the intuitive notion one might expect.

\begin{definition}\label{def: function convex}
    A continuous section $q:\Omega \rightarrow \mathscr C \Omega$ is convex up (resp. down) if for all $x,y\in \Omega$, the segment $[q(x),q(y)]$ is contained in the closure of the positive (resp. negative) side of $\mathscr C \Omega \setminus \im q$. 
\end{definition}

Suppose that $q:\Omega \rightarrow \mathscr C \Omega$ is both a bent domain $(\Omega, \lambda,q)$ and a section (so that $[e_4]\notin 
\overline{\im q}$). For each connected component $P\subset \Omega\setminus \lambda$ let $\alpha(P)=\alpha(q,P)$ be the unique oriented hyperplane containing $q(P)$.

Then the following characterization of convexity holds

\begin{lemma}\label{lemma: two notions of convexity}
    Let $q: \Omega \rightarrow \mathscr C \Omega$ be a bent domain and a section.
    The graph $\im q\subset \mathscr C \Omega$ is convex up if and only if for all $X,Y$ connected components of $\Omega\setminus \lambda$,
    \[\alpha(Y) (q(X)) \geq 0\]
    where the inequality is read to mean that $q(X)$ is on the positive side of $\alpha(Y)$. A symmetric result holds for `convex down'.
\end{lemma}

\begin{proof}
    This amounts to a fact about functions from $\Omega'\subset \R^d\rightarrow \R$ where $\Omega'$ is strictly convex: such a function is convex if and only if its graph is contained in the boundary of its own convex hull. The condition is then the requirement that $\alpha(Y)$ be a supporting hyperplane to the convex hull of $\im q$.
\end{proof}

We can relate this condition for convexity to the geometry of $\Omega$.
Given a geodesic lamination 
\[\tlambda\subset \Omega,\]
the boundary pairs defining its leaves in the space of geodesics identify a geodesic lamination 
\[\tlambda^*\subset \Omega^*.\]
More explicitly, the composition of the limit curves $\left(\xi^*\circ \xi\inverse\right)$ defines a bijection between the set of geodesics in $\Omega$ and the set of geodesics in $\Omega^*$: a geodesic $\ell\subset\Omega$ with endpoints $\ell_\pm$ is mapped to the leaf $\ell^*\subset\Omega^*$ with endpoints $\ell^*_\pm$ which are the unique supporting hyperplanes to $\ell_\pm$. The image of a geodesic lamination under this map is a geodesic lamination.

\begin{definition}
    Let $\ell_1$ and $\ell_2$ be a pair of non-intersecting geodesics in a strictly convex and $C^1$ domain $\Omega\subset \RP^2$, and $\ell^*_1, \ell^*_2$ the dual geodesics in $\Omega^*$. Define the closed set 
    \[R(\ell_1,\ell_2)=\{[v]\in (\RP^2)\mid [\ker v]\cap\Omega^* ~\textrm{between}~ \ell^*_1 ~\textrm{and}~ \ell^*_2.\}\]
    Similarly define $R(\ell^*_1,\ell^*_2)$ symmetrically as a subset of $(\RP^2)^*$.
\end{definition}

By \emph{between} in the definition above, we mean that a geodesic in $\Omega$ is between $\ell_1,\ell_2$ if it separates their respective endpoints and does not intersect either geodesic. 
%In other words, a non-trivial element $v$ of $[v]$ is between $\ell_1$ and $\ell_2$ if it evaluates positively on $\ell_1^*$ and negatively on $\ell_2^*$ in some lift of $\Omega$ to the projective sphere. 

Let $[v_1]$ be the intersection of the two tangent lines to $\Omega$ at the endpoints of $\ell_1$ so that $\ell^*_1\subset [\ker v_1]\subset (\RP^2)^*$.
Define $[v_2]$ similarly.

\begin{lemma}
    Suppose $\ell_1$ and $\ell_2$ do not share an endpoint.
    The subset $R(\ell_1,\ell_2)$ is equal to the closure of the connected component of 
    \[\RP^2\setminus \left([\ker(\ell^*_1)_+]\cup [\ker(\ell^*_1)_-] \cup [\ker(\ell^*_2)_+] \cup [\ker(\ell^*_2)_-]\right)\]
    which is characterized by simultaneously being a $4$-gon, not containing $\Omega$, and having $[v_1]$ and $[v_2]$ in its closure (Figure \ref{figure: intersecting cones}).

    When $\ell_1$ and $\ell_2$ share an endpoint, then $R(\ell_1, \ell_2)$ is the projective segment $[[v_1],[v_2]] \subset \RP^2\setminus \overline \Omega$.
\end{lemma}

\begin{proof}
    We outline the proof and refer the reader to Figure \ref{figure: intersecting cones}. The property defining $[v]\in R(\ell_1, \ell_2)$ is that in a lift of $\Omega^*$ to the projective sphere, a lift $v$ of $[v]$ evaluates (without loss of generality) positively on $\ell_1^*$ and negatively on $\ell_2^*$. Since project segments are convex, it is equivalent to consider the signs of these evaluations on the endpoints of $\ell_1^*$ and $\ell_2^*$, which define four hyperplanes in $\RP^2$. Thus $R(\ell_1, \ell_2)$ is defined exactly as (the closure of) a connected component of the complement of these four hyperplanes in $\RP^2$. 

   That $R(\ell_1, \ell_2)$ does not contain $\Omega$ is due to the fact that the points contained in $R(\ell_1,\ell_2)$ define hyperplanes in $(\RP^2)^*$ which necessarily pass through $\Omega^*$, i.e.\ they are not elements of $\Omega = ((\Omega)^*)^*$. The fact that $[v_1]$ and $[v_2]$ are in the closure of $R(\ell_1, \ell_2)$ follows from the notion of between given after the definition of $R(\ell_1, \ell_2)$.
\end{proof}

For $P\subset \Omega\setminus \lambda$, let $\mathscr E P \subset \TAff^*(3,\R)$ be the set of cotranslations $\tau$ so that $\tau.\alpha(P)$ evaluates positively on $P$. It is immediate that the set $\mathscr E P$ is a proper convex positive cone. 
Let $P_1$ and $P_2$ be a pair of distinct connected components in $\Omega\setminus \tlambda$, and let $\ell_1$ and $\ell_2$ be the extremal elements in the ordered set of leaves of $\tlambda$ between $P_1$ from $P_2$. 

The following corollary is a restatement of Lemma \ref{lemma: two notions of convexity}:

\begin{corollary}\label{cor: convex iff diff in R_k}
    The graph $\im q \subset \mathscr C \Omega$ is convex if and only if for all pairs of distinct connected components $P_1, P_2$ in $\Omega\setminus \tlambda$, the difference
    \[\alpha(P_2) - \alpha(P_1) \in \mathscr E P_2 \cap (-\mathscr E P_1).\]
    Also, the projectivization of $\mathscr E P_2 \cap (-\mathscr E P_1)$ is equal to $R(\ell^*_1,\ell^*_2).$
\end{corollary}

\begin{figure}[h]
\includegraphics[width=.8\textwidth]{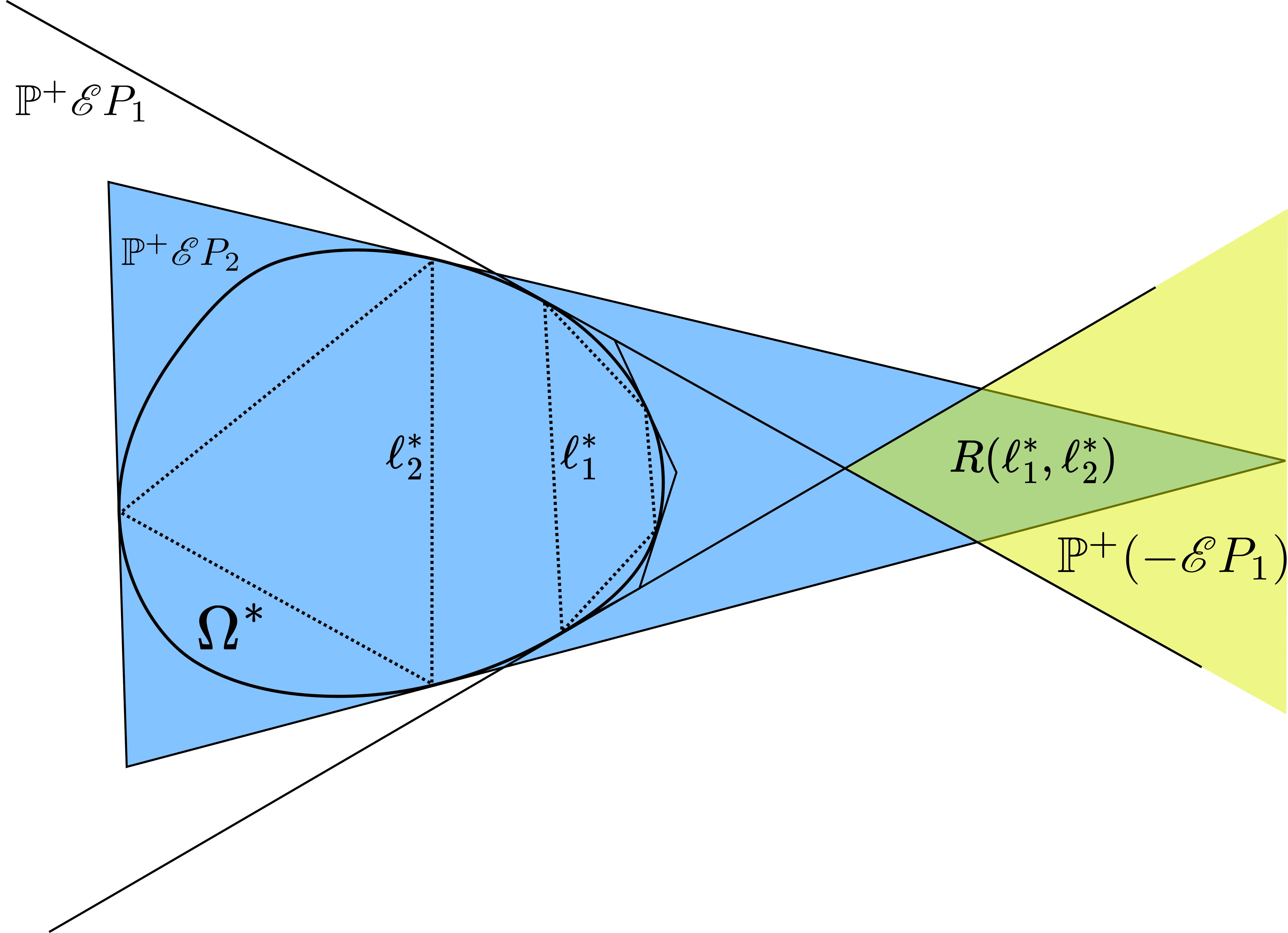}
\centering
\caption{The relationship between $\mathscr E P_1$, $\mathscr E P_2$, and $R(\ell^*_1,\ell^*_2)$ in the \textit{positive} projectivization of $\TAff^*(3,\R)$.}
\label{figure: intersecting cones}
\end{figure}

\begin{proof}
    Figure \ref{figure: intersecting cones} shows an affine chart for the \textit{positive} projectivization of $\TAff^*(3,\R)$, with the convex domain $\Omega^*$. The shaded region $\mathscr E P_1$ is the dual domain of the stratum $P_1$.  It is necessary to work in the positive projectivization to preserve the meaning of signs, which give the notion of convexity. 
    
    The condition in the statement of the lemma
    is equivalent to $\alpha(P_2) - \alpha(P_1) \in \mathscr E P_2$ and $\alpha(P_1) - \alpha(P_2) \in \mathscr E P_1$. By definition of $\mathscr E$, these two conditions are exactly that $\alpha(P_1)$ is positive on $P_2$ and vice versa. If this is satisfied for every pair of two-strata, the function $q$ is convex. 

    The fact that the projectivization of $\mathscr E P_2 \cap (-\mathscr E P_1)$ is equal to $R(\ell^*_1,\ell^*_2)$ is due to convex projective geometry. In particular, since we view $\Omega^*\subset \TAff^*(3,\R)$, we may view $P_i\subset\Omega\subset (\TAff^*(3,\R))^*$. Then $\mathscr E(P_i)\subset \TAff^*(3,\R)$ is the positive dual cone to $P_i$. The intersection  $\mathscr E P_2 \cap (-\mathscr E P_1)$ is exactly the set of hyperplanes which are positive on $P_2$ and negative on $P_1$, which is the separation property defining $R(\ell^*_1,\ell^*_2).$ 
\end{proof}

The following theorem, also found in \cite{Ungemach:thesis}, relates this conversation to coaffine convex cocompact representations.  We supply a proof for completeness.

\begin{theorem}\label{thm: pleated disks}
    When $\varphi  \notin B^1(\pi_1 S, (\R^3)^*_{\rho})$, the boundary $\partial \Xi \setminus \im (\xi)$ is a topologically a union of two disjoint open disks: $\partial\Xi \setminus \im (\xi^1)= \partial\Xi^+\sqcup \partial \Xi^-.$ There exist natural maps
    \[q^\pm :\Omega\rightarrow \partial \Xi^\pm\]
    and laminations $\tlambda^\pm$ so that the $(\Omega, \tlambda^\pm,q^\pm)$ are bent domains. Furthermore, the maps $q^\pm$ are convex $(\rho, \eta)$-equivariant sections; i.e. for all $\gamma\in \pi_1 S$, for all $x\in \Omega$
    \[q^\pm \circ\rho(\gamma) (x) = \eta(\gamma)\circ q^\pm (x).\]
\end{theorem}

\begin{proof}
    The condition that $\varphi$ is not a coboundary is equivalent to requiring that $\eta$ preserves no hyperplane in $\RP^3$ (Lemma \ref{lemma: coboundaries and reducible reps}). Therefore the limit curve $\im(\xi)$ is contained in no hyperplane, its convex hull $\overline \Xi$ is topologically a $3$-disk, and its boundary is topologically a $2$-sphere. The sets $\partial \Xi^\pm$ are then the complements in this $2$-sphere of the topological circle $\im(\xi)$. 

    The set $\Xi$ is a convex body properly contained in an affine chart, and so its boundary is stratified into a poset of faces (see \cite{Rockafellar} for general background on convex geometry). Because $\Xi \subset \RP^3$, the faces of $\Xi$ are of dimension $0$, $1$, or $2$. It is a fact (which is not difficult to prove) that every face $K$ of a proper convex set $X = \hull Y$ is the convex hull of some $Y'\subset Y$. Applying this to the present situation, we conclude that every $0$-face of $\Xi$ must be in $\im (\xi)$. In particular, $\partial \Xi^\pm$ is a union of (open) $1$- and $2$-faces. Define $\tlambda^\pm$ to be the set of $1$-faces in $\partial \Xi^\pm$. The definition of a face of a convex body, together with the previous observation regarding $0$-faces is sufficient to conclude that $\tlambda^\pm$ are laminations on the disks $\partial \Xi^\pm$. 

    Recall from Equation \eqref{eqn: def [Q]} the natural projection $Q: \R^4 \rightarrow \R^4 /[e_4]$, and the induced projection 
    \[[Q]: \RP^3\setminus \{[e_4]\} \rightarrow \P(\R^4/[e_4]) \cong \RP^2.\]
    %The space $\P(\R^4/[e_4])$ is canonically the space of lines in $\RP^3$ containing $[e_4]$. Thus when we consider $\Omega\subset \P(\R^4 /[e_4])$, we are considering $\Omega$ as a subset of the lines in $\RP^3$ through $[e_4]$: namely those that are contained in the cone $\mathscr C \Omega$ (by definition of this cone). 

    It is an exercise in projective geometry to see that the lines through $[e_4]$ contained in $\mathscr C \Omega$ intersect $\Xi$ in a proper segment. 
    Define $q^\pm$ to be the inverse of the restriction of $[Q]$ to $\partial \Xi^\pm$. We may already conclude that $q^\pm$ are bijections. The fact that they are homeomorphisms to their images follows from the fact that their inverses are restrictions of the continuous map $[Q]$ to the graph of a convex section of $\mathscr C \Omega$.

    To check that $[Q](\lambda^\pm)$ are laminations on $\Omega$, one need only note that $[Q]$ is a projective map on each stratum, and that the set of leaves is closed. The closedness of the set of leaves follows from the fact that a Hausdorff limit of closed $n$-dimensional faces of a convex set is a closed face whose dimension is at most $n$. 

    The fact that $q^\pm$ is convex (up or down) in the sense of Definition \ref{def: function convex} follows from Lemma \ref{lemma: two notions of convexity} and the convexity of $\Xi$.
    %To see this: $P,Q$ are plaques of $\Omega\setminus \tlambda^+$. You know that for $P$ every other plaque $R$ is either above it or below it, and the same goes for $R$. You need that it is the same (above or below) for both $P$ and $Q$. But you just compare the two of them, qed.
    
    This completes the Theorem.
    \end{proof}

The difference between the `positive' and `negative' data suggested by the `$\pm$' decorations of the objects defined above is cosmetic, depending only on an orientation of the $\R$-factor of $\mathscr C \Omega$. There is no reason to favor either one over the other. 

\begin{remark}\label{rmk: positive or negative data}
    As it is a vector space,  $(-I)$ acts on $Z^1(\pi_1 S, (\R^3)^*_\rho)$ as an involution (which descends to $H^1(\pi, (\R^3)^*_\rho)$). This induces a nontrivial symmetry of the geometric data which we have been exploring. In particular, it exchanges $\partial \Xi^+$ and $\partial \Xi^-$. Explicitly, this is effected by conjugation by the $\GL(4,\R)$ element $\textrm{diag}(1,1,1,-1)$.
\end{remark}

\subsection{Macroscopic bending data}\label{section: geometric transverse data}

In this subsection, we extract the `macroscopic bending data' from $\partial \Xi$. 
We work with the same set-up and notation from the previous subsection: assume that $\varphi\in Z^1(\pi_1 S,(\R^3)^*_\rho)$ is not a coboundary, and construct the corresponding coaffine convex cocompact representation $\eta=\eta_\varphi$. Let $\tlambda = [Q](\tlambda^+) \subset \Omega$.

Recall that a leaf $\ell$ of a lamination $\lambda$ is \emph{isolated} if there exist a pair of $2$-strata $P_1$ and $P_2$ which both contain $\ell$ in their boundary.

\begin{remark}
    Suppose that $\tlambda$ has an isolated leaf $\ell$ with a pair of adjacent $2$-faces supported by hyperplanes $\alpha$ and $\beta$ (such a leaf does not always exist). The discussion in \S\ref{sec: coaffine geometry} (specifically Corollary \ref{cor: convex iff diff in R_k}) shows that the difference $[\beta-\alpha]\in [\ker\ell]\subset \P\TAff^*(3,\R)$. See Figure \ref{figure: cone_1} for the geometric meaning of $[\ker \ell]$. This section generalizes this localization property of the bending data to the case of non-isolated leaves of the bending lamination on $\partial \Xi$.
\end{remark}

From the discussion in the previous subsection, it is evident that for any $x\in \Omega\setminus \tlambda$, there exists a unique supporting hyperplane $[\ker \alpha(x)]$ to $\Xi$ at $q^+(x)\in \partial \Xi$, where $\alpha(x) \in (\A^3)^*$. 
Thus the following is well-defined: 

\begin{definition}\label{def: bending cocycle}
    Let $[\eta]$ be the $\TAff^*(3,\R)$ conjugacy class of a coaffine representation with Hitchin linear part $L(\eta)=\rho$. The \textit{bending cocycle} $\psi=\psi([\eta])$ is defined as follows.
    
    To an arc $k: [0,1]\rightarrow \Omega$ with endpoints not in $\tlambda$, let
    \[\psi(k) = \alpha(k(1))-\alpha(k(0)) \in \TAff^*(3,\R).\]
\end{definition}

%The map $\psi$ has image in the translation subgroup of $\Aff^*(3,\R)$, which by our previous choice is identified with $\ker(e_4) \subset \R^4$. 
Some elementary properties of $\psi$ are evident from the definitions.

\begin{lemma}\label{lemma: bending cocycle properties}
    The map $\psi$ satisfies the following properties
    \begin{itemize}
    \item Equivariance: $\psi(\gamma.k) = \rho(\gamma).\psi(k)$ for all  arcs transverse to $\tlambda$ and $\gamma\in \pi_1 S$.
    \item Transverse isotopy invariance: $\psi$ only depends on the two-strata containing $\partial k$.
    \item Support: $\psi(k) = 0$ if and only if $\partial k$ is contained in some component of $\Omega \setminus \tlambda$.
    \item Flip equivariance: For an oriented arc $k$, if $\bar k$ denotes $k$ with the opposite orientation, then $\psi(\bar k) = - \psi (k)$.
    \item Additivity: if $k= k_1\cdot k_2$, then $\psi(k) = \psi(k_1) + \psi(k_2)$.
\end{itemize}
\end{lemma}

As one might hope, $\psi$ is related to the cocycle $\varphi$ (defining $\eta$) in the following precise sense.

\begin{lemma}\label{lemma: psi is phi}
    Let $p\in \Omega\setminus \tlambda$, and for all $\gamma\in \pi_1 S$ let $\tgamma$ be the unique lift of $\gamma$ to $\Omega$ based at $p$ (considered up to isotopy relative to its boundary). Then
    \[[\gamma\mapsto \psi(\tgamma\inverse)] = \varphi \in H^1(\pi_1 S,\TAff^*(3,\R)_\rho).\]
\end{lemma}

\begin{proof}
    Let $P\subset\Omega\setminus \tlambda$ be the connected component containing $p$. Conjugacy affects the addition of a coboundary, so it is sufficient to prove the lemma when $\alpha(P)=[\ker e^4]$. In this case, a matrix computation completes the proof:
    \[\eta(\gamma).e^4 = e^4 + \varphi(\gamma\inverse).\]
\end{proof}

Choose any auxiliary Riemannian metric $d_{\mathbb P}$ on $\RP^2$. Having now defined $\psi$, 
revisiting Corollary \ref{cor: convex iff diff in R_k} implies the following useful property.

\begin{figure}[h]
\includegraphics[width=.8\textwidth]{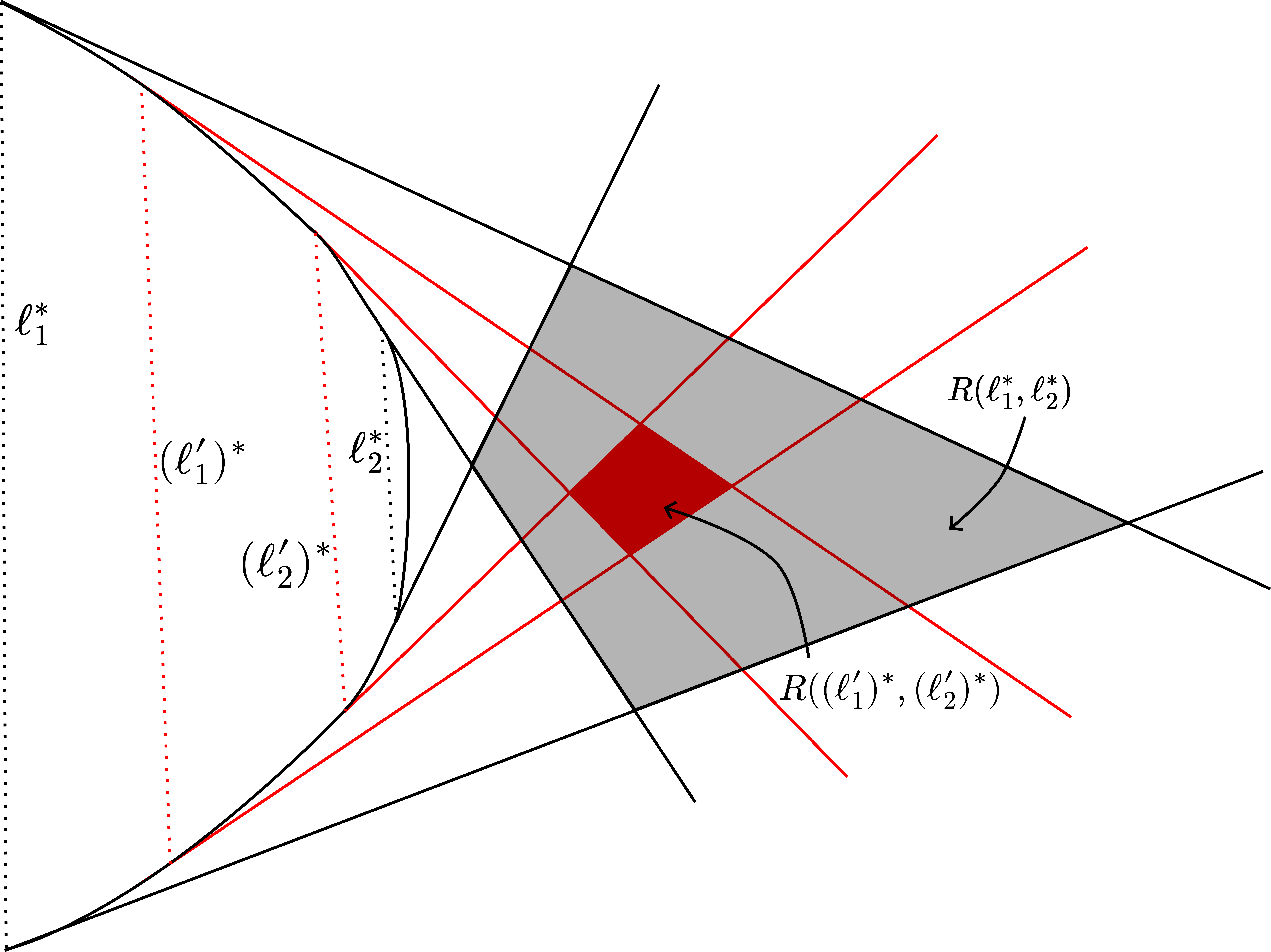}
\centering
\caption{The containment $R((\ell_1')^*,(\ell_1')^*) \subset R(\ell^*_1,\ell^*_2)$. }
\label{fig: rectangle containment}
\end{figure}

\begin{proposition}\label{prop: cocycles localize}
    There is an $a \in (0,1)$ such that when $k$ is an arc transverse to $\tlambda$ and $\ell_x$ is a leaf of $\tlambda$ intersecting $k$ non-trivially, then $d_{\mathbb P}(\psi(k), \ell_x) = O(\diam k^a)$.

    Additionally, there exists a proper positive convex cone $\mathscr E\subset \TAff^*(3,\R)$ so that for all subarcs $k'$ of $k$, $\psi(k')\in \mathscr E$.
\end{proposition}

\begin{proof}
    Let $\ell_1$ and $\ell_2$ be the first and last leaf of $\tlambda$ which $k$ intersects non-trivially. First observe (from the definition) that $\ell_x \in R(\ell_1^*,\ell_2^*)$. So we must only understand how the parallelogram $R(\ell^*_1,\ell^*_2)$ varies in the arc $k$.

    Recall that $\partial \pi_1 S$ has a visual metric after choosing a hyperbolic structure $S\cong X$, which is unique up to bi-H\"older equivalence. Then in the model geometry of  $T^1X$, the map from (oriented) points in a geodesic lamination to their endpoints in $\partial \pi_1 S$ is Lipschitz. The limit curve $\xi:\partial \pi_1 S\rightarrow \RP^2$ is $a$-H\"older, for some $a\in(0,1)$. So $R(\ell^*_1,\ell^*_2)$ is a convex hull of four points in $\RP^2$, the distance between any two of which varies H\"older in $\diam k$. The result then follows as the diameter of $R(\ell^*_1,\ell^*_2)$ varies H\"older in $\diam k$.

    The necessary cone $\mathscr E$ in the final part of the statement is one component of the cone over $R(\ell_1^*,\ell_2^*)\subset \P \TAff^*(3,\R)$. 
    For a subarc $k'$ of $k$, let $\ell'_1$ and $\ell_2'$ be the first and last leaf of $\tlambda$ which $k'$ intersects. Observe, from the definitions, that $R((\ell_1')^*,(\ell_1')^*) \subset R(\ell^*_1,\ell^*_2)$, see Figure \ref{fig: rectangle containment}. Corollary \ref{cor: convex iff diff in R_k} completes the proof.
\end{proof}

\begin{remark}
    We may treat $\psi$ as in Definition \ref{def: bending cocycle} as an element of $Z^1(S,F^\rho)$. Let $k$ be an arc transverse to $\lambda$, and define
    \[\psi(k)= \overline k_* \alpha(k(1))- \alpha(k(0))\]
    where $\overline k_*$ is $\nabla^\rho$-parallel transport along $k$.
\end{remark}

\section{Flow equivariant measures and equivariant transverse measures}\label{sec:equivariant measures}

When the linear part $\rho: \pi_1S \to \SL(3,\R)$ 
of $\eta: \pi_1S \to \Aff^*(3,\R)$ is Fuchsian,
there is a global $\nabla^\rho$-flat section of $E_2^\rho\to T^1S$.
In this classical setting, the bending data on a component of the convex core boundary takes the form of a transverse measure whose support is the bending lamination.
For a fixed hyperbolic metric, measured laminations embed into flow invariant measures on the unit tangent bundle.

In general, there is no global flat section of $E^\rho_2\to T^1S$.
Instead, choose (for the remainder of this section) two auxiliary structures: a smooth norm $\|\cdot \|$ on $E^\rho$ and a hyperbolic metric $X$ on $S$ with geodesic flow $g_t: T^1X \to T^1X$.
Let $\pi: T^1X \to X$ denote the tangent projection and $p\mapsto \overline p$ the antipodal involution of $T^1X$, which satisfies $g_t\overline p = \overline {g_{-t}p}$.

Consider the smooth function $G : T^1X \times \R \to \R_{>0}$ defined by 
\begin{equation}\label{eqn: G defined}
    G(p,t) = \frac{\|\left(g_{[0,t]}\right)_*v\|_{g_tp}}{\|v\|_p}, ~ v\in E_2^\rho|_p,
\end{equation}
which measures the distortion of $\|\cdot \|$ along $g_t$-orbits of the  $\nabla^\rho$-parallel transport.
A measure on the unit tangent bundle of $X$ is $(g,G)$-flow equivariant if it scales by $G$ under the flow $g$ (see Equation \eqref{eqn: G-equivariant}). 
The set of $(g,G)$-flow equivariant measures with support tangent to a lamination (Definition \ref{def: flow equivariant measure}) parameterizes conjugacy classes of coaffine representations with linear part $\rho$ (Theorem \ref{thm: equiv measures sphere}).

Note that the function $G=G(\rho)$ varies continuously in the connection $\nabla^\rho$.
\medskip

In \S\ref{subsec: building equivariant measures}, we characterize which geodesic laminations support flow equivariant measures for a given $\rho$ in terms of the exponential growth of $G$ along its leaves (Corollary \ref{cor: characterization of bendable laminations}).
Theorem \ref{thm: minimal laminations support measures} then demonstrates that \emph{every} non-orientable geodesic lamination is the support of a flow equivariant measure.

In \S\ref{subsec: eq transverse measures}, we disintegrate a flow equivariant measure on a transversal to obtain an \emph{equivariant transverse measure} (Definition \ref{def: transverse measure}), and show that this assignment is a homeomorphism (Proposition \ref{prop: disint is equiv}).
Ultimately, we are interested in these  transverse measures, as the interesting part of a flow equivariant measure lives in the transverse direction to the flow.

\subsection{Building flow equivariant measures}\label{subsec: building equivariant measures}

Let $\mathcal M (T^1X)$ be the space of finite Borel measures on $T^1X$ with its weak-$*$ topology.
We say that $\nu \in \cM(T^1X)$ is $(g, G)$-equivariant if
\begin{equation}\label{eqn: G-equivariant}
    d\left(g_t\right)_*\nu = G(\cdot , -t)d\nu
\end{equation} 
holds for all $t\in \R$.
In other words, $\nu$ and $\left(g_t\right)_*\nu$ are mutually absolutely continuous with Radon-Nykodym derivative given by $G(\cdot, -t)$.
Thus, when $\rho$ is Fuchsian, there is a bilinear form on $E^\rho$ and hyperbolic metric making $G \equiv 1$.

\begin{definition}\label{def: flow equivariant measure}
    Denote by $\ML(T^1X)^{G} \subset \cM(T^1X)$ the set of $(g, G)$-flow equivariant finite Borel measures $\nu$ that are invariant under the antipodal involution and that have support $\hlambda$, where $\lambda\subset X$ is a geodesic lamination and $\hlambda = T^1\lambda$ is its set of tangents. %; $\hlambda\to \lambda$ is also the orientation cover.
\end{definition}
\noindent Note that  $\ML(T^1X)^{G}$ is equipped with the topology of weak-$*$ convergence as a subspace of $\cM(T^1X)$.

The following basic properties of $G$ are easily verified.
\begin{lemma}\label{lem: G cocycle}
    For all $p \in T^1S$ and $s, t \in \R$, we have 
\[G(p,s+t) = G(p,s)G(g_sp,t),\]
and 
 \[G(\overline p , -t) = G(p,t).\]
        Furthermore, for all $T$, we have 
        \[\int_{-T}^T G(p,t)~dt = \int_{-T}^T G(\overline p,t)~dt. \]
\end{lemma}

The key to building $(g, G)$-flow equivariant measured laminations is understanding the accumulation of holonomy along $g_t$-orbits.
Define 
\[\Lambda^+(p) = \limsup_{t\to \infty} \frac1t \log G(p, t), \]
and 
\[\Lambda^-(p) = \liminf_{t\to \infty} \frac1t \log G(p,t),\]
which are Borel measurable functions on $T^1X$.
Note that the \emph{signs} of $\Lambda^\pm$ depend neither on our choice of negatively curved metric (hence parameterization of the geodesic flow) nor on our choice of norm $\| \cdot \|$.
Continuity of $G$ ensures that $-\infty<\Lambda^\pm <+\infty$.

\begin{remark}
The functions $\Lambda^\pm$ measure the exponential growth rate of the middle singular value in matrix products corresponding to an infinite word $\gamma_1\cdot \gamma_2 \cdot ...$ encoding the geodesic trajectory $\{g_t p: t\ge 0\}$ under $\rho$.
We note that if $p$ is generic for a \emph{$g_t$-invariant} ergodic probability measure supported on $\hlambda$, then continuity of $G$ ensures that $\Lambda^+(p) = \Lambda^-(p)$,  and this quantity can be computed by integrating the derivative of $G$; see the proof of Theorem \ref{thm: minimal laminations support measures}.
One can view this in the framework of Lyapunov exponents and the Ergodic Theorem. 
\end{remark}

The following lemma supplies a necessary condition for a geodesic lamination to contain the support of a flow equivariant measure.
\begin{lemma}\label{lem: generic subexponential}
    For a given $\nu\in \ML(T^1X)^{G}$, $\nu$-a.e. point $p$ has $\Lambda^+(p)\le 0$ and $\Lambda^+(\overline p) \le 0$.
\end{lemma}

\begin{proof}
    %Suppose now that $\hlambda\subset T^1X$ is the support of some  $\nu \in \ML(T^1X)^{G}$.
    By equivariance, for all continuous functions $f: T^1X\to \R$, we have
    \[\int f ~d\nu  = \int f\circ g_T \cdot G(\cdot, T)~d\nu\]
    for all $T\in \R$.
    In particular, 
    \begin{equation}\label{eqn: integrate RN}
        \nu(T^1X) = \int G(\cdot ,T) ~d\nu
    \end{equation}
    for all $T\in \R$.
    Consider the set 
    \[B_T^\epsilon = \{p : G(p, T) \ge e^{\epsilon |T|}\}.\]
    Then \eqref{eqn: integrate RN} implies that  \[\nu(B_T^\epsilon) \le \frac{\nu(T^1X)}{e^{\epsilon |T|}},\]
    and so \[\sum_{n = 1}^\infty \nu(B_n^\epsilon) <\infty\]
    for all $\epsilon>0$.
    By the Borel-Cantelli Lemma, 
    \[\nu (\limsup_{n\to \infty}B_n^\epsilon) =0,\]
    which in particular implies that $\Lambda^+(p) < \epsilon$ for all positive $\epsilon$ and $\nu$-a.e. $p$.
    Using the fact that $G(p,-T) = G(\overline p, T)$, a symmetric argument proves also that $ \Lambda^+(\overline p) \le 0$ for $\nu$-a.e. $p$.
    The intersection of two $\nu$-full measure sets has full measure, proving the lemma.
\end{proof}

Now we give sufficient conditions for a minimal geodesic lamination to be the support of a flow equivariant measure.

\begin{proposition}\label{prop: constructing equivariant measures}
    Let $\lambda \subset X$ be a minimal geodesic lamination and $p \in \hlambda$.
    
    If 
    $\int_{-\infty}^\infty G(p,t)~dt<\infty$, then $\hlambda$ is the support of a measure $\nu \in \ML(T^1X)^{G}$ whose mass is concentrated on exactly two $g_t$-orbits that are exchanged by $p\mapsto \overline p$.
    
    If $\int_{0}^\infty G(p,t)~dt=\infty$ and $ 0  \in [\Lambda^-(p), \Lambda^+(p)]$
    then $\hlambda$ is the support of a measure $\nu \in \ML(T^1X)^{G}$.
\end{proposition}

Before proving the proposition, we prove a technical fact.
\begin{lemma}\label{lem: subexponential}
    Suppose $f: \R \to \R$ is continuous and positive, and assume that for every $s>0$, there exists $D(s)>0$ such that when $\abs{u-u'}<s$,  the inequality $D(s)<f(u)/f(u')< 1/D(s)$ holds. 
    Assume also that there is an  $s_0$ and a sequence $T_i\to \infty$ satisfying 
    \[\inf_{T_i}\frac{\int_{T_i-s_0}^{T_i}f(t)~dt}{\int_{0}^{T_i}f(t)~dt}>0.\]
    Then  $\liminf_{i\to \infty} \frac{1}{T_i} \log f(T_i)>0$.
\end{lemma}
\begin{proof}
    Consider the discrete analogue: suppose that $a_n$ is a sequence of positive real numbers so that
    \[\inf_n \frac{a_n}{\sum_{i=0}^{n}a_i}>0.\]
    Then there is some $\epsilon$ so that $0<\epsilon<1$ so that for all $n>0$, $a_n > \varepsilon \sum_{i=0}^{n}a_i$. By an inductive argument using the binomial theorem, one may deduce that for every $n$, 
    \[a_n>\frac{\epsilon}{(1-\epsilon)^n} a_0.\]
    In other words, $\liminf_{n\rightarrow \infty} \frac 1n\log(a_n) >0$.

    To generalize to the continuous case, firstly assume without loss of generality that the sequence $T_i$ is increasing and that $\abs{T_{i+1}-T_i}>2s$. It is then possible to partition each subinterval $[T_{i},T_{i+1}]$ into subintervals $J_j$ of diameter in $[0,s]$ so that for every $i$ there is some $j_i$ with $[T_i-s, T_i]=J_{j_i}$, and so that  $j_i/T_i>1/(2s)$.

    Then consider the sequence 
    $a_j = \int_{J_j}  f~dt$
    and the subsequence $a_{j_i}$, which satisfies the assumption of the discrete case. Therefore there is a positive $\varepsilon$ such that
    \[ \frac{1}{T_i}\log \int_{T_i-s}^{T_i} f~dt>\frac{1}{2s}\frac{1}{j_i} \log a_{j_i} >\varepsilon.\]
    To pass from the statement about the integral over $[T_i-s,T_i]$ to the evaluation $f(T_i)$ requires the assumption about the variation of $f$. By the mean value theorem, there exists some $t_i^* \in [T_i-s,T_i]$ so that $s f(t_i^*)=\int_{T_i-s}^{T_i} f(t) ~dt$.  Then for $i$ large enough, $\frac {1}{T_i} \log f(T_i)>\epsilon$ holds, proving the Lemma.  
\end{proof}

\begin{proof}[Proof of the Proposition]
    If $G(p, \cdot)$ is integrable then so is $G(\overline p, \cdot)$ by Lemma \ref{lem: G cocycle}.  Then
    \[\nu = \int_{-\infty}^{\infty} G(p,t) \delta_{g_tp}~dt + \int_{-\infty}^{\infty} G(\overline p,t) \delta_{g_t\overline p}~dt \]
    is a finite Borel measure on $T^1X$.  The support of $\nu$ is $\hlambda$, because $\{g_tp\}_{t>0}$ is dense in the connected component of $\hlambda$ containing $p$ (see \S\ref{subsec: surface theory}).
    The $(g, G)$-flow equivariance of $\nu$ follows essentially from the cocycle property satisfied by $G$.
    Namely, for any continuous function $f: T^1X \to \R$, we have 
    \[\int f~ d(g_t)_*\nu = \int_{-\infty}^\infty G(p,s) f(g_{t+s}p)~ds + \int_{-\infty}^\infty G(\overline p, s) f(g_{t+s}\overline p) ~ds.\]
    Changing variables $u = t+s$ and applying Lemma \ref{lem: G cocycle}, the right hand side becomes
    \[ \int_{-\infty}^\infty G(p,u) G(g_up,-t)f(g_up)~du + \int_{-\infty}^\infty G(\overline p,u) G(g_u\overline p,-t)f(g_u\overline p)~du = \int f G( \cdot , -t) ~d\nu, \]
    which proves equivariance. 
    
    For the other case, we give a variant of a standard argument of Krylov-Bogoliubov.  
    Suppose now that $\int_0^\infty G(p,t)~dt = \infty$ and $0\in [\Lambda^-(p), \Lambda^+(p)]$.  %The other case is completely symmetric (since $\int_{-\infty}^0G(p,t)~dt = \int_{-\infty}^0G(p,t)~dt$
    For $T>0$, define 
    \begin{equation}\label{eqn: nuT}
        \nu_T =\frac{1}{\int_0^TG(p,t)~dt}\int_0^TG(p,t)\delta_{g_tp}~dt.
    \end{equation} 
    Later, we will symmetrize with respect to the antipodal involution.
    Since $G(p, \cdot)$ is continuous, and $0\in [\Lambda^-(p), \Lambda^+(p)]$, there is a sequence $T_i$ tending to infinity with \[\lim_{i\to \infty} \frac{1}{T_i}\log G(p,T_i) =0.\]
    The cocyle property of $G$ and uniform continuity implies that  
    \[\frac{G(\cdot ,t)}{G(\cdot,t+u)}\]
    is bounded uniformly (in terms of $s$) above and below for all $u \in [0,s]$.  Then
    Lemma \ref{lem: subexponential} implies then that for any $s\in \R$
    \begin{equation}\label{eqn: integral ratio}
        \frac{\int_{T_i-s}^{T_i}G(p,t)~dt}{\int_{0}^{T_i}G(p,t)~dt}\to 0, ~ i\to \infty.
    \end{equation}
    
    We claim that any weak-$*$ accumulation of $\{\nu_{T_i}\}$ is $(g, G)$-flow equivariant.
    Suppose $\nu$ is such an accumulation point (which exists by compactness of $\mathbb P \cM(T^1X)$), and let $f: T^1X\to \R$ be continuous and $\epsilon>0$.
    Then 
    \begin{align*}
        \left|\int f\circ g_s ~d\nu - \int f G(\cdot , -s) d\nu\right| \le & \left| \int  f\circ g_s ~d\nu - \int f\circ g_s~d\nu_{T_i} \right| +\\
        &  \left| \int f\circ g_s ~d\nu_{T_i} - \int fG(\cdot, -s)~d\nu_{T_i} \right| + \left| \int fG(\cdot,-s) ~d\nu_{T_i} - \int f G(\cdot, -s) ~d\nu \right|
    \end{align*} 
    The first and last terms are smaller than $\epsilon$ for $i$ large enough by weak-$*$ convergence of $\nu_{T_i} \to \nu$ (after passing to a subsequence) and because $f\circ g_s$ and $fG(\cdot, -s)$ are both continuous.
    Consider the middle term, which is equal to 
    \[ \frac{1}{\int_0^{T_i} G(p, t) ~ dt}\left| \int_0^{T_i} G(p, t)f(g_{t+s}p)~dt - \int_0^{T_i}G(p,t)G(g_tp,-s)f(g_tp)~dt \right| \]
    After changing variables and applying the cocycle property of $G$, this becomes
    \[ \frac{1}{\int_0^{T_i} G(p, t) ~ dt}\left| \int_{-s}^0 G(p,t)G(g_tp,-s)f(g_tp)~dt - \int_{T_i-s}^{T_i} G(p,t)G(g_tp,-s)f(g_tp)~dt \right|, \]
    which is bounded by the sum of the ratios.
    Then the numerator in \[ \frac{\int_{-s}^0 |G(p,t)G(g_tp,-s)f(g_tp)|~dt}{\int_0^{T_i} G(p, t) ~ dt}\]
    is bounded because of continuity and compactness, while the denominator goes to infinity.
    For the second term, we have 
    \[ \frac{\int_{T_i-s}^{T_i} |G(p,t)G(g_tp,-s)f(g_tp)|~dt}{\int_0^{T_i} G(p, t) ~ dt} \le \|G (\cdot , -s)\|_\infty \|f\|_\infty \frac{\int_{T_i-s}^{T_i}G(p,t)~dt}{\int_{0}^{T_i}G(p,t)~dt},  \]
    which tends to zero by \eqref{eqn: integral ratio}; in particular, both terms are smaller than $\epsilon$ for $i$ large enough.
    Since $\epsilon>0$ was arbitrary, this proves that $\nu$ is equivariant.
    
    Using the property $G(\overline p, -t) = G(p, t)$, one can see that $\overline \nu$, i.e., the pushforward of $\nu$ via the antipodal involution, is $(g, G)$-flow equivariant as well.
    Thus $\nu + \overline\nu$ is symmetric, equivariant, and its support is $\hlambda$.
    This completes the proof.
\end{proof}

Thus we obtain a characterization of which minimal laminations support flow equivariant measures on their orientation covers.  

\begin{corollary}\label{cor: characterization of bendable laminations}
    Let $\lambda\subset X$ be a minimal geodesic lamination.  Then $\hlambda$ is the support of some $\nu \in \ML(T^1X)^{G}$ if and only if there exists $p\in \hlambda$ satisfying $\Lambda^+(p) \le 0$ and $\Lambda^+(\overline p) \le 0$.
\end{corollary}

\begin{proof}
    Sufficiency is proved using Proposition \ref{prop: constructing equivariant measures}: If there exists such a $p$, then either $\int_{-\infty}^\infty G(p,t)~dt<\infty$  or the same integral is divergent and any weak-$*$ accumulation point of $\nu_T$ defined as in \eqref{eqn: nuT} is $(g, G)$-flow equivariant.  After symmetrizing with respect to $p\mapsto \overline p$, this produces a measure $\nu \in \ML(T^1X)^{G}$ with support equal to $\hlambda$. 

    Necessity is the content of Lemma \ref{lem: generic subexponential}.
\end{proof}

Corollary \ref{cor: characterization of bendable laminations} proves Theorem \ref{thm: dynamical characterize bending} from the introduction. 

\begin{remark}
    It is possible also to deduce from  Hurewicz's Ergodic Theorem \cite[Theorem 2.2.1]{Aaronson:ergodic_book} and the Ergodic Decomposition Theorem for finite quasi-invariant Borel measures with a given Radon-Nikodym derivative (\cite{GS:ergodic} or \cite{Aaronson:ergodic_book}) that every ergodic $(g, G)$-flow equivariant probability  measure $\nu$ supported in $\hlambda$ has a generic point.
    That is, there is a point $p$ such that both 
     \[\frac{\int_0^T  G(p,t) \delta_{g_tp}~dt}{\int_0^T G(p,t)~dt} \to \nu \text{ and } \frac{\int_0^T  G(p,-t)\delta_{g_{-t}p}~dt}{\int_0^T G(p,-t)~dt} \to \nu\]
    weak-$*$ as $T\to \infty$.
    The argument can be found in a first course on ergodic theory, and relies on separability of space of continuous functions on a compact metric space.
    See, e.g., \cite[Corollary 4.12]{EW:ergodic}.  
    That Hurewicz's Ergodic theorem for $\Z$-actions can be extended to $\R$-actions can be deduced from the proof of \cite[Corollary 8.15]{EW:ergodic}.
    \end{remark}

Here is another way to build a flow equivariant measure from a geodesic that spirals onto minimal laminations where $G$ decays exponentially in both directions along $\ell$.  Such a measure will have all of its mass concentrated along a single $g_t$-orbit and its time reversal.

\begin{lemma}\label{lem: spiraling leaves}
Suppose $\lambda$ and $\lambda'$ are minimal geodesic laminations on $X$, and suppose that $\ell$ is an isolated leaf spiraling onto $\lambda$ and $\lambda'$ in such a way that $\lambda\cup \ell \cup \lambda'$ is a geodesic lamination.
Let $p \in \widehat \ell$, $q\in \hlambda$ and $q' \in \hlambda'$ such that $\{g_t p\}_{t\ge 0}$ is asymptotic to $\{g_tq\}_{t\ge 0}\subset \hlambda$ and $\{g_t p\}_{t\le 0}$ is asymptotic to $\{g_tq'\}_{t\le 0}\subset \hlambda'$.

If $\Lambda^+(q)<0$ and $\Lambda^+(\overline q')<0$, then there is a $\nu \in \ML(T^1X)^{G}$ giving full measure to $\widehat\ell$ whose support is equal to $\hlambda\cup\widehat \ell \cup \hlambda'$.
\end{lemma} 

\begin{proof}
    Since $\log G (\cdot ,T)$ is uniformly continuous (for any $T$) and the distance between $g_tp$ and $g_{t+s}q$ goes to zero in $t$ for some $s\in \R$, $\Lambda^+(p)=\Lambda^+(q)<0$. Similarly, $\Lambda^+(\overline p)=\Lambda^+(\overline q')<0$.
    Together, this implies that \[\int_{-\infty}^\infty G(p,t)~dt<\infty,\]
    so applying Proposition \ref{prop: constructing equivariant measures} gives the result.
\end{proof}

With help from the previous Lemma, we can now conclude one of the main results from this section.
For context, there is a Thurston measure on $\ML(S)$ in the class of Lebesgue, and almost every measured lamination is minimal, maximal (hence non-orientable), and uniquely ergodic (see \S\ref{subsec: surface theory}).

The following states, in particular, that \emph{every} non-orientable minimal geodesic lamination is the support of a flow equivariant measure, independent of the linear part $\rho$.

\begin{theorem}\label{thm: minimal laminations support measures}
    Every non-orientable minimal geodesic lamination $\lambda$ has that $\hlambda$ is the support of some $\nu \in \ML(T^1X)^{G}$.

    Every uniquely ergodic orientable geodesic lamination $\lambda$ is contained in a geodesic lamination $\lambda'$ satisfying that $\hlambda'$ is the support of some $\nu \in \ML(T^1X)^{G}$.
\end{theorem}

%Recall that Theorem \ref{thm: minimal laminations support measures} claims that for all orientable minimal geodesic laminations $\lambda$, $\hlambda$ is the support of some equivariant measure and that if $\lambda$ is orientable and uniquely ergodic, then there is a $\lambda'\supset \lambda$ such that $\hlambda'$ is the support of an equivariant measure.

\begin{proof}[Proof of Theorem \ref{thm: minimal laminations support measures}]
    Suppose $\lambda\subset X$ is minimal, and consider a $\{g_t\}$-invariant ergodic probability measure $\mu$ with support contained in $\hlambda$.
    The antipodal involution $p\mapsto \overline p \in T^1X$ induces simplicial involution of the space of $\{g_t\}$-invariant probability measures on $\hlambda$; denote by $\overline\mu$ the pushforward of $\mu$ under the involution.\footnote{That the space of $g_t$-invariant probability measures on $\hlambda$ can be strictly larger than the $g_t$ and flip invariant probability measures on $\hlambda$ when $\lambda$ is non-orientable can be deduced from work of Smith \cite{Smith:NUE}.}
    Then $\overline \mu$ is ergodic and if $p$ is generic for $\mu$ then $\overline p$ is generic for $\overline \mu$.
    By generic, we mean that both 
    \[ \frac{1}{T}\int_0^T\delta_{g_tp}~dt \text{ and } \frac1T\int_0^T\delta_{g_{-t}p}~dt\]
    converge weak-$*$ to $\mu$ as $T\to \infty$.
    Moreover, $\mu$-a.e. $p$ is generic.    

    By the choice of a smooth norm $\|\cdot \|$ on $E^\rho$ and because the flow $g_t$ and flat connection $\nabla^\rho$ are all smooth, $G(p,t)$ is continuously differentiable in $t$.
    Denote by 
    \[G'(p,t) = \left.\frac{d}{ds}\right|_{s = t}G(p,s).\]
    Using the cocycle property of $G$, one verifies
    \[G'(p,t) = G(p,t)G'(g_tp,0)~\text{ for all $t\in \R$}.\]
    The Fundamental Theorem of Calculus provides
    \[\frac1T \log G(p,T) = \frac1T\int_0^TG'(g_t p,0)~dt. \]
    Since $p$ is generic for $\mu$, we obtain 
    \[\Lambda(p) := \lim_{T\to \infty} \frac1T\log G(p,t) = \int G'( \cdot, 0) ~d\mu,  \]
    so that $\Lambda^+(p) = \Lambda(p) = \Lambda^-(p)$ holds.

    Genericity of $p$ for $\mu$ implies also that its backward $g_t$ orbit equidistributes to $\mu$, so additionally:
    \[\lim_{T\to \infty} \frac1T\int_0^TG'(g_{-t}p,0)~dt = \int G'(\cdot ,0) ~d\mu = \Lambda(p).\]
    On the other hand, 
    \[\frac1T \int_0^T G'(g_{-t}p, 0)~dt = \frac1T \int_{-T}^0 G'(g_tp, 0)~dt =-\frac 1T \log G(p,-T)=-\frac1T\log G(\overline p,T), \]
    where we used the property $G(p,t) = G(\overline p, -t)$.  
    Taking the limit as $T\to \infty$ yields
    \[\int G'(\cdot, 0) ~d \mu = \Lambda(p) = -\Lambda(\overline p) = -\int G'(\cdot, 0) ~d\overline \mu.\]
    It follows that if $\mu = \overline \mu$, then $\Lambda(p) = 0$ for every $\mu$-generic point $p$.
    Otherwise, $\Lambda(p) = -\Lambda(\overline p)$ holds for all $\mu$-generic points $p$.
    \medskip

    \noindent\textbf{Case: $\lambda$ is non-orientable.}
    That $\lambda$ is non-orientable implies that $\hlambda$ is connected and $g_t$-minimal.
    If $\Lambda(p) = 0 = \Lambda(\overline p)$ we can apply Proposition \ref{prop: constructing equivariant measures}.  
    Indeed, either $\int_{-\infty}^\infty G(p,t)~dt<\infty$ or otherwise one of 
    \[\int_0^\infty G(p,t)~dt \text{ and } \int_{-\infty}^0 G(p,t) ~dt = \int_0^\infty G(\overline p,t)~dt \]
    is infinite.

    Otherwise, $\Lambda(p) \not = \Lambda(\overline p)$ and so $\mu \not= \overline \mu$.
    \begin{claim}
        There is some $p \in \hlambda$ with $0\in [\Lambda^-(p), \Lambda^+(p)]$ and $0\in [\Lambda^-(\overline p), \Lambda^+(\overline p)] $. 
    \end{claim}
    
    \begin{proof}[Proof of the claim]
        Without loss of generality, we assume that $L = \Lambda(p) >0$, so that $-L = \Lambda(\overline p) <0$.
        Since $\frac 1T \log G (\cdot, T)$ is continuous for every $T$, the sets
        \[U_N = \{p : \exists n\ge N  \text { s.t. }\frac 1n \log G (p,n) \ge L/2\} \text { and } V_N = \{p: \exists n\ge N \text{ s.t. }  \frac 1n \log G (p,n) \le -L/2 \} \]
        are open.
        Moreover, $U_N$ and $V_N$ are dense for all $N$.  Indeed, $U_N$ contains the $\mu$-generic points and $V_N$ contains the $\overline \mu$-generic points; genericity is an invariant of the $g_t$-orbit and $\hlambda$ is minimal.  

        Since $\hlambda$ is a Baire-space, the set \[\left( \bigcap_N U_N \right) \bigcap \left( \bigcap_NV_N\right) \] 
        is dense $G_\delta$.  This means that there is a point $p$ satisfying that $\frac 1n \log G(p, n)$ is at least $L/2$ for infinitely many positive values of $n$ and at most $-L/2$ for infinitely many positive values of $n$.  

        Running the same argument in backwards time provides a dense $G_\delta$ set of points satisfying $0\in [\Lambda^-(\overline p), \Lambda^+(\overline p)]$.
        The intersection of dense $G_\delta$ sets is dense $G_\delta$, which proves the claim.
    \end{proof}

    As we did earlier, we can now appeal to Proposition \ref{prop: constructing equivariant measures} to finish the proof in the case that $\lambda$ is non-orientable.

    \noindent \textbf{Case: $\lambda$ is orientable and uniquely ergodic.}
    As before, if  $\Lambda(p) = 0= \Lambda(\overline p)$ for all $\mu$-generic points $p$, then we just appeal to Proposition \ref{prop: constructing equivariant measures} to construct a measure $\nu \in \ML(T^1X)^{G}$ with $\supp \nu = \hlambda$.
    %Since $\lambda$ is orientable,  $\hlambda$ is not connected, so the Claim does not provide us with a point $p$ with 

    Since $\lambda$ is uniquely ergodic and isomorphic to each component of $\hlambda$, each component of $\hlambda$ is uniquely ergodic.
    It follows that $\Lambda(p) >0$ and $\Lambda(\overline p)<0$ holds for every $p$ in one of the components of $\hlambda$.
    
    We claim that there is a geodesic $\ell$ satisfying $\lambda\cup \ell$ is a geodesic lamination, and $\ell$ has negative exponential $G$-growth in both directions along $\ell$, so applying Lemma \ref{lem: spiraling leaves} gives the result in this case.
    The key is just to find such an isolated leaf in the complement of $\lambda$ on $X$ that spirals onto $\lambda$ in its future and past along boundary half-leaves of $\lambda$ that each accumulate negative exponential growth.
    
    To see that such an isolated leaf exists, consider the metric completion $\Sigma$ of a component of $X\setminus \lambda$.  Then $\Sigma$ is a finite area hyperbolic surface with (possibly non-compact) totally geodesic boundary; it is homeomorphic to $S_{g,b}\setminus\{c_1, ..., c_n\}$, where $S_{g,b}$ is a a compact surface of genus $g$ and $b$ boundary components, and $\{c_i\}\subset \partial S_{g,b}$ is a finite collection of ideal points.  Let $A$ be a tubular neighborhood of a boundary component of $S_{g,b}$, and let $A'$ be the image of $ A\setminus \{c_1, ..., c_n\}$ in $\Sigma$; $A'$ is called a \emph{crown}.  The ends of $A'$ are in bijective correspondence with the points $c_i$ contained in $A$, which are called \emph{spikes}.  
    Since $\lambda$ is orientable, the number of spikes on any given crown is even; see Figure \ref{fig: isolated leaves}.

    There are three cases to consider: 
    \begin{enumerate}
        \item $\Sigma$ is contractible, i.e., $\Sigma$ is an ideal polygon.
        \item The interior of $\Sigma$ is an annulus.
        \item The fundamental group of $\Sigma$ contains a non-abelian free group.
    \end{enumerate}

\begin{figure}[h]
    \centering
    \includegraphics[width=.9\linewidth]{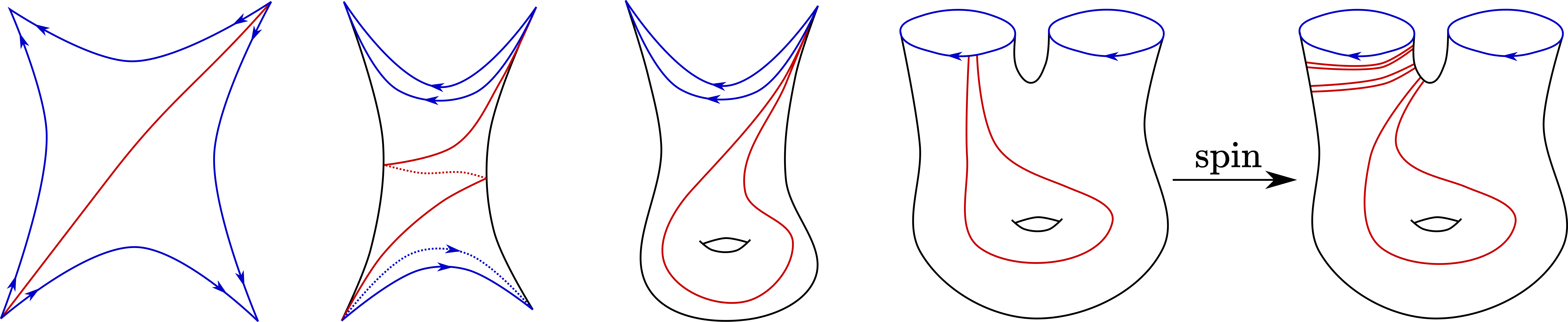}
    \caption{The orientations on the boundary of $\partial\Sigma$ indicates the direction of positive exponential growth.  So the ends of the red isolated leaf are required to travel in the negative direction against these arrows. }
    \label{fig: isolated leaves}
\end{figure}

    The reader should consult Figure \ref{fig: isolated leaves} throughout the rest of the proof.
    In case (1), $\Sigma$ is a $2m$-gon, and it is easy to find an isolated leaf $\ell$ with the desired properties.
    In the annular case (2), note that since $\lambda$ is minimal, both crowns must have a positive, even number of spikes. In this case, it is similarly straight forward to find such an isolated leaf.

    In case (3), if there is a non-compact crown of $\Sigma$, then we can find a properly embedded oriented arc $\alpha$ in $\Sigma$ with endpoints exiting a spike on the same crown with that is not properly homotopic into $\partial \Sigma$.  The geodesic realization of $\alpha$ in $\Sigma$ is the geodesic that we are after.
    Otherwise, we can similarly find a simple, oriented arc joining a boundary component to itself that is not homotopic rel $\partial \Sigma$ into $\partial\Sigma$.
    A ``spinning'' construction then produces the desired isolated leaf.  More precisely, the  Hausdorff limit of geodesic realizations of segments joining $\partial \Sigma$ to itself in the same homotopy class that wrap more and more around that component of $\partial \Sigma$ is the geodesic we are after; see \cite{Thurston:notes} for details about the spinning construction.

    This completes the proof of the Theorem.
\end{proof}

Ultimately, we establish a dictionary between $(g,G)$-flow equivariant measures and bending measures for coaffine representations with linear part $\rho$ (\S \ref{sec: measures and cocycles}). 
\begin{remark}\label{rmk: can and cant do}
     The results of the present section give several constructions of flow equivariant measures and identify their supports.  The constructions are not exhaustive.  In particular, we have not discussed the case that $\lambda$ is orientable and not uniquely ergodic.
    The proof of Theorem \ref{thm: minimal laminations support measures} shows, for example, that if there is a leaf of $\lambda$ which has subexponential growth, then $\lambda$ is the support of a flow equivariant measure.  It also shows that, if every leaf has non-positive (or non-negative) exponential growth, then one can find a geodesic lamination $\lambda'\supset\lambda$, possibly with isolated leaves, and a flow equivariant measure supported on $\lambda'$.  It is not clear what to conclude when these assumptions are not satisfied (or if these assumptions are generally satisfied).
    There is also a more general formulation of Lemma \ref{lem: spiraling leaves} in which the exponential decay of holonomy for two suitable leaves is replaced by an integrability condition.  
    Instead of a complete account, these results supply several rich and interesting classes of examples.
\end{remark}

\subsection{Equivariant transverse measures} \label{subsec: eq transverse measures}

From a flow equivariant measure as in Definition \ref{def: flow equivariant measure}, we would like to extract a \emph{transverse measure} satisfying a certain equivariance condition.
To a $C^1$ embedded arc $k \subset X$ transverse to a geodesic lamination $\lambda$, let $\hk$ be the full preimage in $T^1X$ of $k$ minus those directions tangent to $k$ (thus $\hk$ is two rectangles whose closure is an annulus). 
%In the sequel,

We will often consider oriented arcs. 
An orientation on $k$ and an orientation on $X$ together determine a local co-orientation on the leaves of $\lambda$ that it crosses.  
If $k$ is oriented, we denote by $\hk^+$ the component of $\hk$ whose geodesic trajectories make positive co-orientation with $k$ and similarly define $\hk^-$.

\begin{definition}\label{def: frak o}
    Let $\lambda$ be a lamination on $X$ and $k$ an oriented arc transverse to $\lambda$. For $x\in k\cap \lambda$, let $\fo(x)\in T^1X\vert_x$ be the pre-image of $x$ in $\hlambda\cap\hk^+$. 
\end{definition}

By Lemma \ref{lem: lamination geometry}, $\fo$ is bi-Lipschitz onto its image in $\hk^+$.
%Throughout, $\pi: T^1X \to X$ denotes the projection.

\begin{definition}\label{def: transverse measure}
A $(g, G)$-equivariant transverse measure $\mu$ with support a geodesic lamination $\lambda\subset X$ is an assignment to each unoriented arc $k$ transverse to $\lambda$ a finite positive Borel measure $\mu_k$ with support equal to $k\cap \lambda$ satisfying the following $(g, G)$-equivariance and compatibility conditions:
\begin{itemize}
    \item If $k' \subset k$, then the pushforward of $\mu_{k'}$ under inclusion $k'\to k$ is equal to the restriction of $\mu_k$ to $k'$.
    \item Give $k$ an orientation. 
    %and let $t: k\cap \lambda \to \R$ be a continuous function and let $x\in k\cap \lambda \mapsto \widehat x \in \hk^+\cap \hlambda$ be the positively oriented section.
    Suppose $k'$ is homotopic to $k$ by a homotopy $H$ transverse to $\lambda$ and \[k'\cap \lambda = \{\pi(g_{t(x)} \fo(x)) :  x \in k\cap \lambda\}.\]
    Then  \[dH\inverse_*\mu_{k'} = G(\fo( \cdot), t(\cdot))d\mu_k,\] where $H : k \to k'$ is induced by the homotopy. 
\end{itemize}

The collection of $(g,G)$-equivariant transverse measures is denoted by $ \ML (X)^{G}$; it is equipped with the following topology: 
$\mu_n \to \mu$ if for all arcs $k$ transverse to $\supp \mu$, $\left(\mu_{n}\right)_k\to \mu_k$ weak-$*$ on $k$.
\end{definition}

\begin{remark}
    If $G$ were constant, as in the case of a Fuchsian representation, Definition \ref{def: transverse measure} recovers the usual definition of an \textit{invariant} transverse measure. %An assignment to every arc which is compatible with the inclusion of subarcs but without any assumption on equivariance by homotopies is simply a \textit{transverse measure}.
\end{remark}

We will now describe a continuous map 
\[ \ML(T^1X)^{G}\to \ML(X)^{G}.\]
The procedure will essentially be by local disintegration.  

Let $\nu \in \ML(T^1X)^{G}$ and $k$ be an oriented arc transverse to $\lambda = \pi(\supp \nu)$.
Consider the measurable functions 
$r_+: \hk^+ \to \R_{>0}\cup \{+\infty\}$
and $r_-: \hk^+ \to \R_{<0}\cup \{-\infty\}$ 
recording the first return to time to $\hk^+$, in $+$-time and $-$-time, i.e.,
if $r_+(p)<+\infty$, then $g_{r_+(p)}p \in \hk^+$, but $g_tp\not\in \hk^+$ for all $t \in (0, r_+(p))$ and similarly $g_{r_-(p)}p\in \hk^+$ but $g_tp\not\in \hk^+$ for all $t \in ( r_-(p),0)$.

There is a $\nu$-a.e. defined measurable projection \[\pi_{\hk^+}: T^1X \to \hk^+\]
 by the rule $\pi_{\hk^+}(g_tp)= p$ for all $t \in (r_-(p)/2, r_+(p)/2)$.

For (small) $t>0$ consider the flow box \[B(t,\hk^+) : = \bigcup_{s\in [-t,t]}g_s\hk^+\subset T^1X,\]
and measures
\[\mu_{\hk^+}^t:=\frac{1}{2t}\left(\pi_{\hk^+}\right)_*\nu|_{B(t,\hk^+)}\]
on $\hk^+$.

\begin{lemma} \label{lem: disint}
    For each small enough $t$, $\mu_k^t = \pi_*\mu_{\hk^+}^t$ is a finite Borel measure on $k$.  If $k$ meets $\lambda$ non-trivially, then $\mu_{k}^t$ converge weak-$*$ to a non-zero measure $\mu_{k}$  as $t\to 0$.
    Moreover, $\mu_k$ does not depend on the choice of orientation on $k$.
\end{lemma}

\begin{proof}
    We invoke the Disintegration Theorem which gives us a measure $\eta_k$ on $\hk^+$ 
    satisfying
    \begin{equation}\label{eqn: disint}
        \nu = \int_{p \in \hk^+} \left( \int_{r_-(p)/2}^{r_+(p)/2} G(p,t) \delta_{g_tp}~dt\right)d\eta_k(p).
    \end{equation}
    Formally, disintegration gives us a Borel family of  measures $\nu_p \in \cM(T^1X)$ supported on $\pi_{\hk^+}\inverse (p)$, each defined up to scale.
    But the $(g, G)$-equivariance of $\nu$ can be used to show that $\nu_p$ is a constant multiple of $\int_{r_-(p)/2}^{r_+(p)/2} G(p,t)\delta_{g_tp}~dt$, and we choose the scale factor on each fiber making \eqref{eqn: disint} hold true for some Borel measure $\eta_k$ on $\hk^+$ supported on $\hk^+\cap \hlambda$.

    Let $ f$ be a continuous function on $k$ and let  $\bar f$ be the continuous function on $\hk^+$ obtained by pulling back via $\pi$.
    Then 
    \[\left| \int f~d\mu_{k}^t \right| = \left| \int \bar f~d\mu_{\hk^+}^t \right| \le \frac{1}{2t} \int_{p\in \hk^+}\left(\int_{-t}^t |G(p,s) \bar f(p)|~ds \right) d\eta_k(p) \le \|f\|_\infty \left(\sup_{s\in [-t,t]}\|G(\cdot,s)\|_\infty \eta_k(\hk^+)\right), \]
    which proves that $\mu_{k}^t$ is a bounded linear operator on $C(k)$, hence is a Borel measure.
    
    Since $G(\cdot, t)$ tends uniformly to $1$ as $t\to 0$, a similar computation shows that $\lim_{t\to 0} \mu_{\hk}^t = \eta_k$.
    Setting $\mu_{\hk} = \eta_k$ completes the proof of convergence of $\mu_k^t \to \mu_k$.
    That $\mu_k$ is independent of the chosen orientation of $k$ is a tedious computation but follows from the symmetries of $\nu$, $G$, and of the construction of $\mu_{\hk^+}^t$ under the antipodal involution of $T^1X$.
\end{proof}

    Now that we have extracted a measure $\mu_{k}$ on $k$, we would like to know that it is $(g, G)$-equivariant.

\begin{lemma}\label{lem: flowmotopy}
    Let $k$ be an arc transverse to $\lambda = \pi (\supp \nu)$, and let $t\in \R$. Give $k$ an orientation so that $\fo:k\rightarrow \hk^+$ is $\nu$-almost everywhere defined
    %and consider the positively co-oriented section $x\in k\cap \lambda \mapsto \widehat x \in \hk^+$
    and suppose $k'$ differs from $k$ by a homotopy $H: I \times k \to X$ through transverse arcs satisfying $H(s,x) = \pi(g_{st}\fo(x))$ for all $x \in k\cap \lambda$ and $s\in [0,1]$.
    Then \[dH\inverse_*\mu_{k'} = G(\fo( \cdot), t)d\mu_k.\]
\end{lemma}

\begin{proof}
    One checks from the definitions that 
    \[\pi_{\hk'^+} = g_{t} \circ \pi_{\hk^+} \circ g_{-t},  ~ \nu\text{-a.e.}\]
    Then 
    \[\mu_{\hk'^+}^s = \frac{1}{2s} (g_t \circ\pi_{\hk^+} \circ g_{-t})_*\nu|_{B(s,\hk'^+)} = \frac{1}{2s} H_* (\pi_{\hk^+})_* G(\cdot, t) \nu|_{B(s,\hk^+)} = H_* G(\fo(\cdot), t)\mu_{\hk^+}^s.\]
    As this equality holds for each $s$ sufficiently small, it is also true in the limit as $s\to 0$ and for the pushforward measures to $k$ and $k'$. This completes the proof.
\end{proof}

Checking that the second condition in Definition \ref{def: transverse measure} holds follows from an approximation argument similar to the previous lemma.
That the first condition holds is immediate from the construction.
\begin{proposition}\label{prop: disint is equiv}
    The disintegration map $\nu\in \ML(T^1X)^{G} \mapsto \{\mu_k\}_k \in \ML(X)^{G}$ described above is a homeomorphism.
\end{proposition}

\begin{proof}
    That the map is well defined is the content of Lemmas \ref{lem: disint} and \ref{lem: flowmotopy}.
    Continuity follows along the lines of the proof of Lemma \ref{lem: disint}: $\nu_n \to \nu$ if and only if the disintegration over $k$ converges.
    
    Now we construct an inverse mapping.  Given $\mu \in \ML(X)^{G}$ with support $\lambda$, consider a transverse arc $k$ equipped with an orientation.
    The section $x\in k\cap \lambda \to \widehat x \in \hk^+$ is bi-Lipschitz onto its image, so the pushforward of $\mu_k$ defines a Borel measure $\mu_{\hk^+}$ on $\hk^+$.
    Extending $\mu_{\hk^+}$ to a $(g,G)$-flow equivariant measure with support contained in $\hlambda$ on $T^1X$ is straightforward using \eqref{eqn: disint}.
    Averaging with respect to the antipodal involution produces a flip-invariant and $(g, G)$-flow equivariant measure $\nu \in \ML(T^1X)^{G}$ with support equal to $\hlambda$.
    That this measure did not depend on the choice of $k$ or its orientation can be deduced from the second bullet point (equivariance) in Definition \ref{def: transverse measure}.
    
    To see that this inverse mapping is continuous, consider a convergent sequence $\mu_n \to \mu \in \ML(X)^{G}$ with supports $\lambda_n$ and $\lambda$.
    Call the extended measures $\nu_n$ and $\nu \in \ML(T^1X)^{G}$ constructed in the previous paragraph.  By construction, their supports are $\hlambda_n$ and $\hlambda \subset T^1X$.

    A sequence $z_n \to z$ if and only if every subsequence has a further subsequence converging to $z$.
    A basic fact from measure theory is that if a sequence of Borel measures $\nu_n \to \nu$ weak-$*$ on a compact metric space, then every Hausdorff accumulation point of $\{\supp \nu_n\}$ contains $\supp \nu$.
    So it suffices to consider a subsequence, which we do not rename, with the property that $\lambda_n \to \lambda'\supset \lambda$ in the Hausdorff topology on closed subsets of $X$.
    Then also $\hlambda_n \to \hlambda'$ in the Hausdorff topology on closed subsets of $T^1X$.
    
    Suppose $p_n \in \hlambda_n \cap \hk^+$ converge to $p \in \lambda\cap \hk^+$.
    Since $G$ is continuous on $T^1X\times \R$ and the flow $g_t$ is continuous, the measures $\int G(p_n, t) \delta_{g_tp_n}~dt$ supported on compact segments of leaves of $\hlambda_n$ converge weak-$*$ on $T^1X$ to the corresponding $\int G(p, t) \delta_{g_t p }~dt$ supported on a compact segment of a leaf of $\hlambda$. 
    Since the distributions of the fibers converge (i.e., $(\mu_n)_k \to \mu_k$), the extended measures converge weak-$*$ on $T^1X$, i.e.,  $\nu_n \to \nu$.  This completes the proof.   
\end{proof}

\section{From measures to cocycles and back}
\label{sec: measures and cocycles}
In this section, we will make precise the relationship between transverse equivariant measures on laminations defined in Section \ref{sec:equivariant measures} and the cocycles discussed in Section \ref{section: geometric transverse data}.

To define the notion of an equivariant transverse measure, we used the auxiliary data of a norm on the flat bundle $E^\rho$ and a parameterization of the geodesic flow on $T^1S$ afforded by a hyperbolic metric $X$. This produced a function $G$ describing the expansion of the norm under the geodesic flow (in the subbundle $E^\rho_2$). Now, Definition \ref{def: bundle-valued transverse measure} associates to every element of $\ML(X)^G$ an $E^\rho$-valued cocycle: an object which is independent of the auxiliary data. 

Conversely, in Section \ref{subsec: cocycles to measures} we extract a transverse $G$-equivariant measure from a cocycle (by way of the bending cocycle of Definition \ref{def: bending cocycle}).

Finally, we establish that these constructions are continuously inverse 
to one another:

\begin{theorem}\label{thm: equiv measures sphere}
    The space $\ML(X)^{G}$ is homeomorphic to $H^1(S, F^\rho)\setminus \{0\})$ via a map which is homogeneous with respect to positive scale.
    Consequently, the quotient by positive scale $\PML(X)^{G}$ is a sphere of dimension $6g-7$.
\end{theorem}

\subsection{From measures to cocycles}\label{subsec: measures to cocycles}

Let $\eta$ be a coaffine representation with Hitchin linear part, $\rho$.
 Recall from \S\ref{section: geometric transverse data} that the top component $\partial \Xi^+$ of the convex core boundary is pleated along a $\pi_1 S$-invariant lamination $\tlambda^+$. We choose to work with one component of $\partial\Xi$, though there is a symmetric construction for $\partial \Xi^-$.
 
\medskip
 In the previous Section \ref{sec:equivariant measures}, we made the choice of a norm $\| \cdot \|$ on $E^\rho$ and of a hyperbolic metric $X$ on $S$.
 By $\lambda$, we denote the geodesic straightening of $\lambda^+$ on $X$.
 Lemma \ref{lemma: bending cocycle properties} give us a cocycle $\psi \in Z^1(S, E^\rho)$ satisfying some compatibility properties with respect to $\lambda$.
 Recall that for an oriented arc $k$ transverse to $\lambda$ on $X$, the function $\fo: k \cap \lambda \to T^1X$ from Definition \ref{def: frak o} is a bi-Lipschitz homeomorphism onto its image and satisfies $\pi \circ \fo (x) = x$.

Using the chosen norm on $E^\rho$ and Lemma \ref{lemma: E_2 has sections}, denote by 
\[s^1:T^1X \to E^\rho_2 \]
the unique unit norm positively co-oriented vector in each fiber. This is a H\"older continuous map.

\medskip
Recall from \S\ref{subsec: Anosov property} that the bundle $E^\rho$ is a pullback of $F^\rho\to X$ by $T^1X \to X$, so the $\nabla^\rho$-parallel transport around the fibers of $T^1X \to X$ is trivial. Therefore the parallel transport along $k$ in the integrand below is unambiguous.

\begin{definition}\label{def: bundle-valued transverse measure}
    Suppose $\mu\in \ML(X)^G$ has support $\lambda$.
    For every oriented arc $k$ transverse to $\lambda$, define a Borel measure on $k$ valued in the fiber $F^\rho\vert_{k(0)}$ by 
    \[\mu\otimes s^1(k) = \mu_k\otimes (s^1\circ \fo)|_k.\] 
    More precisely,
    \begin{equation}\label{eqn: measure to cocycle integral}
        \int\mu\otimes s^1(k) =\int s^1 \circ\fo~d\mu_k = \int_k \left(\overline{k_x}\right)_* s^1\circ \fo (x)~ d\mu_k(x) \in F^\rho\vert_{k(0)},
    \end{equation}
    where $k_x$ is the parameterized subarc of $k$ begininng at $k(0)$ and ending at $x$ and $\left(\overline {k_x} \right)_*$  is the $\nabla^\rho$-parallel transport. 
    
    Then $\mu\otimes s^1$ defines an $F^\rho$ valued co-chain: $\left(k\mapsto \int \mu\otimes s^1(k)\right)$.
\end{definition}

\begin{remark}\label{rmk: indicators dense}
    Indicator functions on $k' \cap \lambda$ are dense in the space of continuous functions on $k\cap \lambda$, where $k'$ is a subarc of $k$. 
    Using the restriction property of $\mu \in \ML(X)^G$ (Definition \ref{def: transverse measure}), equation \eqref{eqn: measure to cocycle integral} unambiguously defines a bounded linear map from  continuous functions on $k$ to $F^\rho|_{k(0)}$. 
    %This is what a Borel measure valued in a vector space is.
\end{remark}

We assert that the cochain described in Definition \ref{def: bundle-valued transverse measure} is a cocycle.
%asserts that $\psi = \mu \otimes s^1$ for some rough transverse measure $\mu$ with support $\lambda$.

\begin{lemma}\label{lemma: homotopy invariance of tensor}
    Let $\mu \in \ML(X)^G$ with support $\lambda$, and let $k$ and $k'$ oriented arcs transverse to $\lambda$ related by a homotopy $H$ through transverse arcs. % transverse to $\lambda$.
    Then
    \[\mu\otimes s^1(k') = H_*(\mu\otimes s^1(k)).\]

    In particular, $[\mu\otimes s^1] \in H^1(S,F^\rho)$.
\end{lemma}

\begin{proof}
        One follows the definitions, in particular see equation \eqref{eqn: G defined} and Definition \ref{def: transverse measure}. 
        %By $\mu_k\otimes s^1|_k$, we mean the integral 
        
        The $G$-equivariance of the sections and the measures exactly cancel; we suppress the arguments of $G$ for readability:

        \[H_*(\mu_k\otimes s^1\circ\fo \vert_k) = \left(\frac{1}{G\circ H\inverse} \mu_{k'}\right) \otimes \left(G\circ H\inverse  (s^1\circ\fo)\vert_{k'}\right). \]   
        It is now straightforward to verify that integrating $\mu\otimes s^1$ on oriented transverse arcs is a cocycle.
\end{proof}

\subsection{From cocycles to measures}\label{subsec: cocycles to measures}

From the previous subsection, every equivariant transverse measure defines a cohomology class valued in $F^\rho$. 
Suppose $\eta: \pi_1S\to \Aff^*(3,\R)$ has linear part $\rho$ and bending cocycle $\psi$ with support $\lambda$, as in Definition \ref{def: bending cocycle}.
Now we show that $\psi$  is actually \emph{equal} to integrating some $\mu\otimes s^1$ along oriented arcs.
The first step is to extract the transverse measure.

\begin{remark}\label{rmk: suppress transport}
    In this subsection and the next, we have suppressed the parallel transport along $k$ in the integrals for readability (as in equation \eqref{eqn: measure to cocycle integral}).  Since this parallel transport operator has bounded operator norm that tends to $1$ on small enough segments, this does not affect our estimates in any meaningful way.
\end{remark}

\begin{lemma}\label{lem: bending cocycle is integration}
    For every oriented arc $k$ transverse to $\lambda$ on $X$, there is a unique Borel measure $\mu_k$ with support equal to $k \cap \lambda$ satisfying 
    \[\psi(k) = \int_k s^1\circ \fo ~d\mu_k \in F^\rho|_{k(0)}\]
\end{lemma}

\begin{proof}
    We will construct a family of finite Borel measures $\{\vartheta_\epsilon\}$ on $k$ with support contained in $k\cap \lambda$ indexed by a positive $\epsilon$ and show that they converge as $\epsilon\to 0$.
    %The limit will be the $\mu_k$ from the statement of the lemma.
    
    Consider the $\epsilon$-neighborhood $\cN_\epsilon(\lambda)$ on $X$. 
    Then $k\cap \cN_{\epsilon}(\lambda)$ is the union of its connected components listed in order along $k$:
    \[k\cap \cN_\epsilon(\lambda) = k_\epsilon^1\cup ... \cup k_\epsilon^{m(\epsilon)}.\] 
    Lemma \ref{lem: tt geometry} asserts that $\diam (k_\epsilon^i) = O(\epsilon)$ and $m(\epsilon) = O(\log1/\epsilon)$.
    Let $x_\epsilon^i$ denote the initial point of $k_\epsilon^i$, and define 
    \[\vartheta_\epsilon : = \sum_{i = 1}^{m(\epsilon)} \|\psi (k_\epsilon^i)\|_{x_\epsilon^i} \delta_{x_\epsilon^i}. \]
    %We would like to show that $\vartheta_\epsilon$ converge as $\epsilon\to 0$.

    \begin{claim}\label{clm: convergence criterion}
        If for all oriented subarcs $k' \subset k$, it holds that 
        \[\int_{k'}  s^1\circ \fo~ d\vartheta_\epsilon \to \psi(k'),\]
        then $\vartheta_\epsilon$ converges to a finite Borel measure $\mu_k$ as $\epsilon\to 0$.
    \end{claim}

    \begin{proof}[Proof of Claim \ref{clm: convergence criterion}]
        Suppose that $\{\vartheta_\epsilon\}_{\epsilon>0}$ has two distinct accumulation points $\nu$ and $\nu'$.  We will show that 
        \[ \int_{ k'}  s^1\circ \fo ~d\nu  \not = \int_{ k'} s^1\circ \fo ~d\nu'\]
        for some ark $k'$, hence at least one of them is not $\psi(k')$.
        
        The linear span of indicator functions is dense in $C(k\cap \lambda)$ and $\nu \not= \nu'$, so there is a $k'$ such that $\nu(k') \not= \nu'(k')$, i.e., that the $\nu$ and $\nu'$ integrals of the indicator function on $k'$ disagree. 
        Furthermore, we may find such a $k'$ with arbitrarily small diameter.

        Consider the integrand \[s^1\circ \fo: k'\cap \lambda \to F^\rho|_{k'(0)}.\]
        With respect to a $\|\cdot \|_{k'(0)}$-orthonormal basis $e_1, e_2, e_3$ with $\langle e_2 \rangle  = E_2^\rho|_{k'(0)}$, we have 
        \[ s^1\circ \fo = f_1 e_1 + f_2 e_2 + f_3 e_3,  \]
        where $f_i : k\cap \lambda \to \R$ are $\alpha$-H\"older continuous (see Lemma \ref{lemma: E_2 has sections}).
        Thus
        \[
            \left\|\int_{k'} s^1\circ \fo ~ d(\nu-\nu')\right\|^2 = \sum_{i = 1}^3\left| \int_{k'}f_i~d(\nu-\nu') \right|^2.
        \]
        Proposition \ref{prop: cocycles localize} implies that 
        \[ |f_2 - 1| = O(\diam(k')^\alpha) \text{ and } |f_i|= O(\diam(k')^\alpha) \text{ for $i = 1,3$}.\]
        Hence that 
        \[\left\|\int_{k'}s^1\circ \fo ~ d(\nu-\nu')\right\|^2\ge \left|\int_{k'} f_2 ~d(\nu-\nu') \right|^2 - (\nu(k') - \nu'(k'))^2O(\diam (k')^{2\alpha}).\]
        Furthermore, we obtain 
        \[\left| \int_{k'}f_2 d(\nu-\nu')  \right|^2  \ge [(1-O(\diam(k')^\alpha)( \nu(k') - \nu'(k'))]^2.\]
        All together this gives
        \[\left\|\int_{k'} s^1\circ \fo ~ d(\nu-\nu')\right\|^2\ge (\nu(k')-\nu'(k'))^2 (1- O(\diam(k')^{2\alpha})).\]
        Thus, if $\diam(k')$ is small enough, then this quantity is positive because $\nu(k')\not=\nu'(k')$, which completes the proof of the claim.
        %Then we have 
        %\[\left \| \int \left( \overline {k'_x}\right)_* s^1\circ \fo - f_2 e_2 ~ d\nu   \right\| = O(\diam(k')^\alpha),\]
        %and similarly for $\nu'$.
        %Thus to show that \[\int_{x\in k'} \left( \overline {k'_x}\right)_* s^1\circ \fo (x) ~d\nu \not= \int_{x\in k'} \left( \overline {k'_x}\right)_* s^1\circ \fo (x) ~d\nu',\]
        %it suffices to prove that 
        %\[\left| \int_{k'} f_2 ~d\nu -\int_{k'} f_2 ~d\nu'\right|\]
        %is much larger than $O(\diam(k')^\alpha)$.
    \end{proof}

    It remains to prove that the hypothesis of Claim \ref{clm: convergence criterion} holds.
    We will prove that \[\int_{k}  s^1\circ \fo ~d\vartheta_\epsilon\to \psi(k).\]
    The proof that the analogous claim also holds for all subarcs $k' \subset k$ is identical.

    Repeated application of the cocycle property of $\psi$ and that its support is $\lambda$ (in the sense of Lemma \ref{lemma: bending cocycle properties}) gives 
    \[\psi(k) = \sum_{i = 1}^{m(\epsilon)} \psi(k_\epsilon^i), \]
    while 
    \[\int_{k}s^1\circ \fo ~d\vartheta_\epsilon = \sum_{i =1}^{m(\epsilon)} s^1\circ \fo (x_\epsilon^i) \|\psi(k_\epsilon^i)\|. \]
    %The operator norms of $\left(\overline {k_{x_\epsilon^i}}\right)_*$ is uniformly bounded by some constant $K$ for all $\epsilon$ and all $i$.
    Proposition \ref{prop: cocycles localize} implies that there is a uniform bound on $\left\|\psi(k_\epsilon^i)\right\|$ and that 
    \[\left\|\psi(k_\epsilon^i) - \|\psi(k_\epsilon^i)\| s^1\circ \fo (x_\epsilon^i) \right\| = O(\diam(k_{\epsilon}^i)^\alpha) = O(\epsilon^\alpha)\] 
    for all $i$. 
    
    Therefore, 
    \begin{align*}
        \left\| \psi(k) -\int_{k}  s^1\circ \fo ~d\vartheta_\epsilon \right\| &\le \sum_{i = 1}^{m(\epsilon)}\left\|\psi(k_\epsilon^i) - \|\psi(k_\epsilon^i)\| s^1\circ \fo (x_\epsilon^i) \right\| \\
        & \le m(\epsilon) O(\epsilon^\alpha) = O(\epsilon^\alpha \log 1/\epsilon), 
    \end{align*}
    which tends to $0$ with $\epsilon$. 
    This completes the proof of the Lemma.
\end{proof}

\begin{lemma}\label{lemma: equivariance of measure from cocycle}
    The assignment $\left(k \mapsto \mu_k\right)$ from Lemma \ref{lem: bending cocycle is integration} defines an element of $\ML(X)^G$ with support $\lambda$.
\end{lemma}

\begin{proof}
    Let $k$ and $k'$ be transverse arcs related by a homotopy $H$.
    Lemma \ref{lem: bending cocycle is integration} gives 
    \[\psi(k') = \int s^1~d\mu_{k'} \text{ and } \psi(k) = \int s^1~d\mu_k.\]
    Since $\psi$ is a cocycle, we have 
    \[H_*(\psi(k)) = \psi(k').\]

    We compute \[H_*(\psi(k)) = \int_{k'} (H_*s^1)~ dH_*\mu_k = \int_{k'} s^1 ~(G\circ H\inverse)dH_* \mu_k.\]
    Thus 
    \[\int_{k'} s^1  ~d \mu_{k'}= \int_{k'} s^1 ~(G\circ H\inverse)dH_* \mu_k\]
    A similar computation holds on every pair of subarcs (which are homotopic by $H$). Therefore the two measures agree on the indicator functions of subarcs, which are dense in the space of continuous functions on $\lambda \cap k$. So as Borel measures 
    \[dH\inverse_*\mu_{k'} = G(\fo(\cdot), t(\cdot))d\mu_k\]
    as required.
\end{proof}

Lemmas \ref{lem: bending cocycle is integration} and \ref{lemma: equivariance of measure from cocycle} show that there is a well defined map 
\[\Psi: \{\eta: \pi_1 S\to \Aff^*(3,\R) \text{ with linear part $\rho$}\}/\TAff^*(3,\R) \to \ML(X)^G.\]

\subsection{$\Phi$ is a homeomorphism}

Lemma \ref{lemma: homotopy invariance of tensor} shows that there is a well-defined map
\begin{align}
     \Phi:\ML(X)^{G} &\to H^1(S, F^\rho)\setminus \{0\}\\
      \mu & \mapsto [\mu\otimes s^1] 
\end{align}

\begin{lemma}\label{lem: bijection measure cohomology}
    The map $\Phi$ is a bijection.
\end{lemma}
\begin{proof}
    The inverse map is given as the composition of $\Psi$ with the map from Lemma \ref{lem: cohomology parameterizes reps up to conj} sending a cohomology class $\varphi$ to the bending cocycle $\psi(\eta_\varphi)$ (Definition \ref{def: bending cocycle}).\footnote{We have used the isomorphism $H^1(S,F^\rho) \cong H^1(\pi_1 S, (\R^3)^*_\rho)$ from \S\ref{subsec: cohomology}.} %determined by the coaffine representation with linear part $\rho$ and .
\end{proof}

   We now demonstrate that the map $\mu\mapsto  [\mu\otimes s^1]$ respects the topology of the two spaces.

\begin{proposition}\label{prop: cocycle continuous injective homogeneous}
%    The assignment $\mu\in \ML(X)^{G} \mapsto \Phi(\mu) \in H^1(S,E^\rho)$ 
The map $\Phi$ is continuous and homogeneous with respect to positive scale.
\end{proposition}
\begin{proof}
    Consider a convergent sequence $\mu_n \to \mu \in \ML(X)^{G}$ with supports denoted by $\lambda_n = \supp \mu_n$ and $\lambda = \supp \mu$.
    As in the proof of Proposition \ref{prop: disint is equiv},  every Hausdorff accumulation of $\lambda_n$ contains $\lambda$.  Passing to a subsequence, we assume that $\lambda_n$ converges Hausdorff to some $\lambda'\supset \lambda$.
    %A basic fact from measure theory is that if a sequence of Borel measures $\nu_n \to \nu$ weak-* on a compact metric space, then every Hausdorff accumulation point of $\{\supp \nu_n\}$ contains $\supp \nu$.
    %Any accumulation point for $\{\lambda_n\}$ in the Hausdorff topology on closed subsets of $X$ is again a geodesic lamination.
    
    Every small enough oriented arc $k$ meeting $\lambda'$ transversely meets $\lambda_n$ transversely for $n$-large enough.
    For each $n$, let $\fo_n:k\cap \lambda_n\rightarrow \hk$ be the positively co-orientation lift, and $\fo$ similarly defined for $\lambda$. 
    Denote by $\vartheta_n = \left(\mu_n\right)_k$, and $\vartheta = \mu_k$.
    By definition of the topology on $\ML(X)^{G}$, $\vartheta_n \to \vartheta$ weak-$*$ as $n\to \infty$. 
    
    Defining 
    \[I_n = \int_k s^1\circ\fo_n~d\vartheta_n \text{ and } I = \int_k s^1\circ\fo~d\vartheta,\]
    our goal is to show that $I_n \to I$ as $n\to\infty$.
    \medskip
    
    For $\delta>0$, we have $\lambda\subset\cN_\delta(\lambda')$ and $\lambda_n \subset \cN_\delta(\lambda')$ for all $n$ large enough.
    In particular, $\cN_\delta(\lambda') \cap k \supset \lambda\cap k$ and $\cN_\delta(\lambda')\cap k \supset \lambda_n\cap k$ for all $n$ large enough.
    By Lemma \ref{lem: tt geometry},  $\cN_\delta(\lambda') \cap k$ consists of intervals $k^1, ..., k^{m(\delta)}$, each of diameter $O(\delta)$.
    
    \medskip

    Choose points $x^i \in \lambda'\cap k^i$ and define 
    \[R=\sum_{i = 1}^{m(\delta)} s^1(\fo(x^i)) \vartheta(k^i).\]
    Choose also $x_n^i \in \lambda_n\cap k^i$ converging to $x^i$ as $n\to \infty$, and define 
    \[R_n = \sum_{i = 1}^{m(\delta)} s^1(\fo_n(x_n^i)) \vartheta_n(k^i).\]
    Each map %\mb{These E's are confusing. They want to be $F_2$ but it isn'd defined yet}
    \[x \in \lambda_n \cap k \mapsto s^1(\fo_n(x)) \in E_2^\rho \text{ and } x\in \lambda \cap k \mapsto s^1(\fo(x)) \in E_2^\rho\]
    is $\alpha$-H\"older with uniformly bounded $\alpha$-H\"older norm.  Using this, we estimate

    \[\left\| I_n - R_n \right\| \le  \sum_{i =1}^{m(\delta)} \int_{k^i}\left\|  s^1(\fo_n(x)) - s^1(\fo_n(x_n^i))  \right\| ~d\vartheta_n(x) = O(\delta^\alpha) \vartheta_n(k), \]
    and similarly
    \[\left\| I - R \right\| = O(\delta^\alpha) \vartheta(k). \]
    The triangle inequality and $\vartheta_n(k) \to \vartheta(k)$ give that
    \[\|I - I_n\|\le O(\delta^\alpha) (\vartheta_n(k) +\vartheta(k)) + \|R_n - R\| = O(\delta^\alpha) + \|R_n - R\|.\]
    Given $\epsilon>0$, we can choose $\delta$ such that $O(\delta^\alpha) <\epsilon/2$.
    \medskip

    Now that $\delta$ is fixed, we bound
    \begin{align}%\label{eqn: difference riemann sums}
        \|R_n-R\| &\le \sum_{i =1}^{m(\delta)} \left\|  s^1(\fo_n(x_n^i)) \vartheta_n(k^i) - s^1(\fo(x^i))\vartheta(k^i) \right\|
    \end{align}
    using the general fact that $|xy -zw| \le |y| \cdot |x-z| + |z|\cdot |w-y|$:
    
    By Hausdorff convergence of $\hlambda_n \to \hlambda'$, by our choice of $x_n^i\to x^i$, and by continuity of $s^1$, for any $\epsilon'>0$,
    \[\left\| s^1(\fo_n(x_n^i)) -  s^1(\fo(x^i))  \right\|  <\epsilon'\]
    for large enough $n$.
    By weak-$*$ convergence of $\vartheta_n \to \vartheta$,
    and because $k^i$ is a continuity set for $\vartheta$ (since $\partial k^i$ does not meet $\lambda'$) we have \[|\vartheta_n(k^i) - \vartheta(k^i)| <\epsilon'\]
    for $n$ large enough.

    Since $\vartheta(k^i) \le \vartheta(k)$, we obtain
    \[ \|R_n -R\| \le m(\delta) \left( \vartheta(k) \epsilon' + \|s^1\circ \fo\|_\infty \epsilon' \right) \]
    for all $n$ large enough.
    In particular, since $\delta$ is already fixed, we can take $\epsilon'$ small enough and then $n$ suitably large so that $\|R_n -R\|<\epsilon/2$ holds.
    
    This proves that if $n$ is large enough then $\|I_n - I\|$ is smaller than $\epsilon$.  Since $\epsilon$ was arbitrary, we have proved that $I_n \to I$.  Since $k$ was arbitrary, continuity of $\Phi$ follows.

    That $\Phi$ is homogeneous follows from the definitions.  This completes the proof.
    \end{proof}

    Although the function $G$ depended on the choices of a parametrization of the flow $g_t$ and a norm $\|\cdot\|$ on $ E^\rho$, Lemma \ref{lem: bijection measure cohomology} demonstrates that the resulting space of equivariant measures does not, in the sense that $\Phi$ is canonical; it identifies $\ML(X)^G$ with a cohomology theory defined independently of any choices.

\begin{proof}[Proof of Theorem \ref{thm: equiv measures sphere}]
    Lemma \ref{lem: bijection measure cohomology} gives that $\Phi$ is bijective, while Proposition \ref{prop: cocycle continuous injective homogeneous} states that $\Phi$ is continuous and homogeneous with respect to positive scale.
    We obtain an induced bijective continuous mapping
    \[\mathcal {PML}(X)^{G} \to \mathbb P_{>0}(H^1(S,F^\rho)\setminus \{0\}) \cong \mathbb S^{6g-7}. \]
    Because $\mathcal {PML}(X)^{G}$ can be identified with a closed subset of probability measures on $T^1X$ (Proposition \ref{prop: disint is equiv}) which is compact in the weak-$*$ topology, and $\mathbb S^{6g-7}$ is Hausdorff, this mapping is a homeomorphism. This completes the Theorem.
\end{proof}

\section{Flat vector bundles over geodesic laminations}\label{section: flat bundles over laminations}

Let $\rho: \pi_1S \to \SL(3,\R)$ be a Hitchin representation. Recall the flow-invariant H\"older continuous Anosov splitting $E^\rho = E_1\oplus E_2 \oplus E_3$ over $T^1X$ from \S\ref{subsec: Anosov property}. The involution $p\mapsto \overline p$ preserves the $E_2$ subbundle, while permuting $E_1$ and $E_3$.  Since $E^\rho$ was constructed as a pullback of $F^\rho$ by the tangent projection, $E_2|_p$ is naturally identified with $E_2|_{\overline p}$.
%\[E_2|_{\overline p} = E_2|_p\] 
%(while $E_1|_{\overline p } =E_3|_{p}$).
Therefore, $E_2$ descends to a line bundle $\overline {E_2} $ over $\mathbb PTX$.

Let $\lambda \subset S$ be a geodesic lamination. 

\begin{definition}\label{def: F2 over lamination}
    Let $F_2^\rho(\lambda)=F_2\to \lambda$ be the pullback of $\overline {E_2}$ by the inclusion $\lambda  \cong \mathbb PT\lambda \hookrightarrow \mathbb PTX$.  The flat connection $\nabla^\rho$ restricts to a flat connection along the leaves of $\lambda$.
\end{definition}

Although $\overline {E_2}$ need not be flat, we will nevertheless produce a flat connection on $F_2^\rho(\lambda)$.
Indeed, the goal of this section is to show that there is a natural H\"older extension of the connection $\nabla^\rho$ to a H\"older flat connection on $F_2\to \lambda$.
First, we formalize the notion of a flat connection on a vector bundle over a lamination compatible with a given connection along the leaves (Definition \ref{def: flat connection}), which gives a way to extend parallel transport in directions transverse to the leaves.
%Such an object defines parallel transport of a vector over a point in the lamination to a vector over nearby points, even on different leaves. 

The slithering map studied in \cite{BonDre} furnishes the bundle $F_2$ with a flat connection in this sense (Lemma \ref{lemma: E_2 is flat}), and we obtain the main theorem of the section as a corollary.

\begin{theorem}\label{thm:flat connection over lamination}
    For any Hitchin representation $\rho: \pi_1S \to \SL(3, \R)$ and any geodesic lamination $\lambda\subset S$ the bundle $F_2\to \lambda$ has a H\"older continuous flat connection $\nabla$, called the \emph{slithering connection}, which is compatible with $\nabla^\rho$ along the leaves of $\lambda$.
\end{theorem}
    
From the formalism developed in Section \S\ref{subsec: tt and flat connections}, we extract a combinatorial description of flat bundles over a connected lamination in terms of a snug train track that carries $\lambda$.
In particular, suppose that $\mathcal N$ is a snug train track neighborhood for $\lambda$ equipped with a foliation by ties with quotient $q: \mathcal N \to \tau$ (see \S\ref{subsec: surface theory}). 

\begin{proposition}\label{prop: Hitchin train track bundle}
    There is a character $\chi: \pi_1 \tau\to \R^\times$
    defining a flat bundle $F_\chi\to\tau$
    \[F_{\chi} = ( \widetilde \tau \times \R) / \left[(x, v) \sim (\gamma.x, \chi(\gamma) v)\right]\]
    such that  $F_2 \to \lambda$ is isomorphic to the pullback of $F_\chi$ by $q$.
\end{proposition}

Thus over each  connected component of a lamination $\lambda$, the bundle $F_2$ is determined by a finite collection of data which can be encoded as a train track with holonomy. 

\begin{remark}
    The results in this section have analogues for $E_i\vert_{\hlambda}$, $i=1,\dots 3$ as well as for Borel Anosov representations into $\SL(d,\R)$. The proofs require no significant changes.
\end{remark}

\subsection{Train tracks and flat connections}\label{subsec: tt and flat connections}

In the rest of this section, we will be interested in connected geodesic lamination $\mu \subset \mathcal M$, where $\mathcal M$ is a train track neighborhood of $\mu$ equipped with a foliation by ties.  The leaf space of the foliation of $\cM$ by its ties $\theta = \cM/\sim$ is a train track carrying $\mu$. See \S\ref{subsec: surface theory} for terminology.

\begin{remark}
We treat real vector bundles of arbitrary finite dimension over a lamination, though we will only use the case when the dimension is one. The one-dimensional case is slightly simpler, as the difference between the vector bundle and its frame bundle essentially disappears.
\end{remark}

\begin{definition}[Flat vector bundle over a lamination]
    A (finite-dimensional real) \textit{flat} vector bundle $\zeta: E \to \mu$ is a finite-dimensional vector bundle equipped with an atlas 
    \[\mathcal A = \{\psi_\alpha : \zeta\inverse (U_\alpha) \to U_\alpha \times \R^d\}\]
    of vector bundle charts where the linear part of the transition maps is constant, i.e., $\psi_\beta \circ \psi_\alpha\inverse (p, v) = (\phi_{\beta\alpha} (p), A_{\beta\alpha}v)$, where $\phi_{\beta\alpha}$ is a homeomorphism (where defined) and $A_{\beta\alpha} \in \GL(\R^d)$ is independent of $p$.

    Associated to such a vector bundle is the \textit{frame bundle} 
    \[\zeta_{\mathcal B}: \mathcal B(E) \to \mu,\]
    the fiber of which at $p$ is the set of ordered bases of $E\vert_p$. The atlas of $E$ gives an atlas of charts for $\mathcal B(E)$ with image in $U_\alpha\times \mathcal B(\R^d)$.
    
    The local sections of $\zeta_{\mathcal B}$ that are constant in charts are called the \emph{local flat frames} of $E$. A local section $s_1$ of $E$ is flat if there exist local sections $s_2, \dots, s_d$ on the same neighborhood in $\mu$ so that $(s_1,\dots, s_d)$ is a local flat frame (i.e. $s_1$ is flat when it can be extended to a flat frame).
\end{definition}

It is an exercise to check that given a point $(p,v)\in E$, there exists exactly one local flat section $s$ of $E$ so that $s(p)=v$ (up to extending and restricting the neighborhood on which $s$ is defined).

\medskip
Suppose $E$ is a vector bundle over a space with a H\"older class of metrics. A H\"older structure on $E$ is an atlas of charts such that the transition maps $\phi_{\beta\alpha}$ are H\"older functions into $\GL(d,\R)$; a flat structure on $E$ is compatible with a H\"older structure if its flat sections are H\"older.

\medskip

We may also approach flatness from the perspective of flat connections. 
%We will show that this is ultimately no different.
Let $\zeta: E \to \mu$ be an $\R^d$-bundle over $\mu$, which is not necessarily flat. Let $\nabla^0$ be a flat connection along the leaves of $\mu$: for every leaf $l$ of $\mu$, $\nabla^0$ is a flat connection on the bundle $E\vert_l = \zeta\inverse(l)$, the restriction of $E$ to $l$. 
%That is, for each vector $v  \in E\vert_p$, there is a unique section $\nabla^0_{\ell(p)}(v)$ of $E\vert_{\ell(p)}$ which is flat with respect to $\nabla^0$. 
%Any flat frame of $B(E\vert_{\ell(p)})$ characterizes $\nabla^0$ restricted to $\ell$ in that any two $\nabla^0$-flat frames differ by an element of the general linear group.

%For example, consider any
%line sub-bundle $F$ of $E_\rho$ (such as $E_i$) that is invariant under the geodesic foliation of $T^1S$. %, i.e., the assignment $x\in T^1S \mapsto F_x$ where $F_x\subset \pi\inverse(x) \subset E_p$ is constant along leaves of the geodesic foliation.
%Then $\nabla^\rho$ defines a connection on the $F$-bundle over any leaf of the geodesic foliation of $T^1S$.

To an embedded $C^1$ path $\alpha \subset \mathcal M$ transverse to  $\mu$, denote by $T_\alpha= \alpha \cap \mu$ the \textit{transversal space} of $\mu$ in $\alpha$.
If $\zeta: E\to \mu$ is an $\R^d$-bundle, then $E\vert_\alpha$ is shorthand notation for the bundle $E$ restricted to $T_\alpha$.

\begin{definition}\label{def: flat connection}[Flat connection]
    Let $\mu$ be a lamination in a train track neighborhood $\cM$.
    A \textit{flat connection} $\nabla$ on $\zeta: E\to \mu$ extending a flat connection $\nabla^0$ along the leaves of $\mu$ is an assignment
    $$ \nabla: \alpha\in \{\text{embedded transversals to $\mu$ in $\mathcal M$} \} \mapsto \nabla_\alpha\in \{\textrm{continuous sections}~ T_\alpha\rightarrow \mathcal B(E)\vert_\alpha \}/ \GL(d,\R).$$
    Here $\GL(d,\R)$ is acting on the fibers of $E\vert_\alpha$ via a trivialization. The map $\nabla$ is required to satisfy
    \begin{enumerate}
        \item \textbf{Restriction.} If $\alpha'\subset \alpha$ then 
        \[\nabla_{\alpha'} = (\nabla_\alpha) \vert_{\alpha'} .\]
        %\item \textbf{Parametrization invariance} If $h:I\rightarrow I'$ is a homeomorphism between intervals in $\R$, and $\alpha, \beta$ satisfy $\alpha = \beta\circ h$, then $h^*\circ\nabla_\alpha = \nabla_\beta \circ h$ as maps $\alpha^*\lambda \rightarrow \beta^*E_2$
        \item \textbf{Compatibility with $\nabla^0$}. Suppose $\alpha$ and $\beta$ are transversely isotopic via $H$, %and let $H_*$ denote the induced homeomorphism $\alpha^*\lambda \to \beta^*\lambda$. 
        Then $\nabla$ satisfies
        \[H_* \circ \nabla_\alpha =  \nabla_\beta\]
        (i.e. along leaves, $\nabla$ restricts to $\nabla^0$). Here $H_*$ denotes $\nabla^0$-parallel transport along leaves parameterized by $H$.
    \end{enumerate}

    %Components of 
    The local flat sections of $\mathcal B(E)$ defining $\nabla$ are called the \textit{local flat frames} of $E$, the components of which are called the \textit{local flat sections} of $E$. 
    %For a subset $X$ admitting flat sections, denote by $\nabla_{X,v}$ the unique flat section over $X$ containing the vector $v$.
    We will denote by $X_*(v)$ the flat extension of a vector $v$ along a suitable subset $X$ by $\nabla$.   
\end{definition}

    The connection $\nabla^0$ is part of the data of $\nabla$, so that $\nabla$ may be used to parallel transport vectors along the leaves. of $\mu$.
    %In the presence of a H\"older structure on $E$, a flat connection $\nabla$ is H\"older if its local flat sections are H\"older functions.
    %These two notions of flat bundles over laminations coincide.
    To prove the following lemma, one necessarily defines $\nabla$ as the $\GL(d,\R)$-classes of local flat sections of the frame bundle $\mathcal B(E)$.

\begin{lemma}
    A flat vector bundle $\zeta: E\to \mu$ admits a unique flat connection $\nabla$ which is compatible in the sense that the local flat sections of $\zeta$ are equal to the local flat sections of $\nabla$.  %Conversely, a flat connection $\nabla$ on $\lambda$ induces a flat structure on the bundle $E$.
\end{lemma}

Given a train track neighborhood $\cM$ of $\mu$, and a flat connection $\nabla$ on a bundle $E$ over $\mu$, the following result is a useful consequence of the second item in Definition \ref{def: flat connection}. For a point $p$ on a lamination $\mu$ carried by a traintrack $\cM$, denote by $\ell(p)$ and $t(p)$ the leaf of $\lambda$ and the tie of $\cM$ containing $p$, respectively.
\begin{lemma}\label{lem: rectangles trivial hol}
    Let $(p,v)\in E$, and suppose $q\in \mu$ is such  that $\ell(p)$, $t(q)$, $\ell(q)$, and $t(p)$ bound a rectangle embedded in $\cM$ whose sides are labelled (respectively) $L_p, S_q, L_q, S_p$. 
    Then the $\nabla$-parallel transport around this rectangle is trivial, i.e.
    \[(L_q)_*\circ (S_p)_*(v)= (S_q)_* \circ (L_p)_*(v).\]
\end{lemma}

\begin{proof}

    We recognize the fact that the rectangle is the image of the boundary of a homotopy between $t(p)$ and $t(q)$, and apply the property from the definition of $\nabla$:
    \[(H)_* \circ \nabla_{t(p)} =  \nabla_{t(q)}.\]

    Consider the local flat section $\im(H)_*(v)$. The restriction of this section to $t(p)\cup \ell(q)$ and $\ell(p)\cup t(q)$ are exactly 
    \[(L_q)_*\circ (S_p)_*(v)\]
    and
    \[(S_q)_* \circ (L_p)_*(v),\]
    respectively. We conclude that the Lemma holds because $H_*(v)$ is well-defined.
\end{proof}

\begin{remark}\label{rmk: ties are enough}
    The same proof holds more generally for arbitrary pairs of homotopic embedded transversals. 
    In particular, a flat connection $\nabla$ on $\mu$ extending $\nabla^0$ is determined locally by its flat sections over a tie.
\end{remark}

Recall that collapsing the ties of $\cM$ induces a map $q:\cM \to \theta$ which is $C^1$ restricted to the leaves of $\mu$.
We define \textit{branch charts} for the bundle $E$ in the presence of $\theta$ as follows. Let $U_b =q\inverse(b)$, where $b$ is a branch of $\theta$, and let $V_b$ be the pre-image in $E$ of $U_b$. Let $T_b$ be the transversal space of $U_b$, i.e., the space of leaves of $\mu$ running through $U_b$. %understood as a quotient of $U_b$ (by the coordinate of $b$ itself). 

Given a flat connection $\nabla$ extending $\nabla^0$ as in Definition \ref{def: flat connection}, Lemma \ref{lem: rectangles trivial hol} guarantees the existence of a flat section of the frame bundle $\mathcal B(E)\vert_{U_b}$. Let $A_b$ be the linear map taking this frame to the standard frame in $\R^d$.
%Choose a transversal $t$ to $b$ and choose a $\nabla$-flat frame $B_t =(b_1,\cdots,b_d)\subset B(E)\vert_t$ over $t$. 
%For a point $p\in U$, let $\ell(p)\in T_b$ be the leaf containing $p$, and let $t(p)$ be the tie containing $p$. 
The branch chart is defined as
\begin{center}
    $\psi_b:V_b \to U_b \times \R^d$\\
    $(p,v) \mapsto\left(\ell(p),t(p), A_b v\right).$
\end{center}

A different choice of flat section over $U_b$ would effect $A_b$ by post-composition with an element of $\GL(d,\R)$.
Extending the open sets $U_b$ slightly over each switch to $U_b'$ (with pre-images $V_b'\subset E$), we see that the transitions between branch charts are constant.

\begin{lemma}
    Given a flat connection $\nabla$ on a vector bundle $E$ over a lamination, the atlas (described above) 
    \[\mathcal A(\nabla) = \left\{ \psi_b : V'_b \to U_b' \times \R^d \right\} \]
    of vector bundle charts is flat.
\end{lemma}

\begin{proof}
    One can construct $U_b'$ so that the triple-wise intersection of any such distinct charts is trivial, so the cocycle condition on transition maps is trivially satisfied. A pair of adjacent train track charts $U'_b$ and $U'_{b'}$ overlap on a small isotopy of a tie. The flat sections $A_b$ and $A_b'$ defining the branch charts differ on their overlap by a constant element of $\GL(d,\R)$ (see Lemma \ref{lem: rectangles trivial hol}). 
\end{proof}

Our next objective is to combinatorialize an $\R^d$-bundle $\zeta : E\to\mu$ with flat connection $\nabla$ and build a flat $\R^d$-bundle over a train track $\theta$ that carries $\mu$.
Choosing a point $p \in \theta$ gives a holonomy representation 
\[\hol: \pi_1(\theta, p) \to \GL(E\vert_p)\cong \GL(d,\R),\]
as follows. Each $[\alpha]\in \pi_1(\theta, p)$ lifts to a continuous concatenation of segments of leaves in $\lambda$ and segments of ties of $\mathcal M$, call it $\alpha$. Then let
\[\hol([\alpha]) = \overline\alpha_*\]
the $\nabla$-parallel transport around $\overline \alpha$.
Any two choices of such lifts $\alpha$ and $\alpha'$ differ by a sequence of rectangles, so that Lemma \ref{lem: rectangles trivial hol} guarantees that $\hol$ is well-defined.
This is the usual definition of holonomy.
%Flatness ensures that $\hol$ is independent of the choice of representative of $\alpha$ within its homotopy class.

\medskip

There is a flat $\R^d$-bundle $E^{\hol}$ over $\theta$ constructed by taking $\pi_1 \theta$-orbits:
\[E^{\hol} = ( \widetilde \theta \times \R^d) / \left[(x, v) \sim ([\alpha].x, \hol([\alpha]) v)\right].\] 
The bundle $E^{\hol}$ has a flat connection $\nabla^{\hol}$ obtained by pushing forward the trivial connection on $\R^d \times \widetilde \theta$ by the quotient projection.

Although $\mu$ does not generally have a reasonable universal cover (it is generally not path connected), using the train track $\theta$, we still recover the following direct analogue of the structure theorem for flat bundles (see, e.g., \cite[\S13]{Steenrod:topology}).
Its proof is essentially the content of this section.
%The following is a consequence of the construction of $E^ {\hol}$.

\begin{proposition}\label{prop: combinatorialized lamination bundles}
    In the category of flat bundles, $q^*(E^{\hol}, \nabla^{\hol})$ is isomorphic to $(E, \nabla)$. 
\end{proposition}

The proof is no different than in the case of a manifold, though one works in the universal cover of the traintrack neighborhood $\mathcal M$. The branch charts give the necessary local trivializations.

\subsection{Key examples: Hitchin representations}\label{subsection: Hitchin bundles are flat}

In this section we apply our general discussion to line bundles coming from Borel-Anosov representations (see \S\ref{subsec: Anosov property} for definitions). In particular, we obtain Theorem \ref{thm:flat connection over lamination} and Proposition \ref{prop: Hitchin train track bundle}.
%We apply the \textit{slithering map} defined in \cite{BonDre} for Borel-Anosov representations.
We begin with a brief discussion of the \textit{slithering map} defined in \cite{BonDre} for Borel-Anosov representations.\footnote{Bonahon and Dreyer state their result for Hitchin representations and remark that that their proof ``likely'' extends also to a more general setting that includes all Borel-Anosov representations $\pi_1S \to \SL(d, \R)$.  Recently, In \cite[\S6]{MMMZ:pleated} show that the slithering construction is well-defined in a general setting.}

The limit curves $\xi$ and $\xi^*$ from Section \ref{subsec: Anosov property} may be understood as a single map to the space of full flags in $\R^3$; call this map $x\mapsto \mathcal F(x)$. These flags are transverse: for any distinct $x$ and $y$ in $\partial \pi_1S$, $\xi(x) \oplus \xi^*(y) =\R^3$ is direct \cite{Labourie:Anosov}. In general, Borel Anosov representations have a transverse limit map to the space of full flags, constructed from the Anosov splitting, see \S\ref{subsec: Anosov property}.

Recall that $\pi$ the fundamental group of the surface is a hyperbolic group, and hence has a well-defined boundary $\partial \pi$. For any triple of points $x,y,z\in \partial \pi$, the flags $\mathcal F (x),\mathcal F (y),\mathcal F(z)$ are pairwise transverse. It is an exercise in linear algebra that for a pairwise transverse triple of full flags $\mathcal F_1, \mathcal F_2, \mathcal F_3$ of $\R^d$, there exists a unique element $g_1^{2,3}\in \GL(d,\R)$ which is unipotent, stabilizes $\mathcal F_1$ and has $g_1^{2,3}.\mathcal F_2 =\mathcal F_3$. 

The slithering map is the unique H\"older extension of this map across collections of parallel leaves in a maximal lamination. More precisely, if $g$ and $h$ are oriented leaves in a maximal lamination $\lambda$ which share an endpoint, $g^-=h^-$, then the slithering map $\Sigma_{hg}$ is the natural isomorphism $\R^d\rightarrow \R^d$ which acts trivially on $\mathcal F(g^-)$ and takes $\mathcal F (g^+)\mapsto \mathcal F(h^{+})$. %This map is exactly . 
Bonahon and Dreyer extended this rule to pairs of coherently oriented geodesics that do not share a common endpoint; this map is called the slithering map 
\[(h,g)\mapsto\Sigma_{hg}\in \SL(d,\R).\]
The slithering map obeys a composition rule when a geodesic $g'$ separates a pair of geodesics $g$ and $g''$. Namely 
\[\Sigma_{g''g}=\Sigma_{g''g'}\circ \Sigma_{g'g}.\]
It is also independent of the (coherent) orientation of the geodesics:
\[\Sigma_{\overline g' \overline g} = \Sigma_{g'g}\]
The assumption that $g'$ separates is necessary; generically the slithering map around a triangle is nontrivial. To be precise, if $x,y,z$ are leaves of a lamination forming an ideal triangle, then the composition 
\begin{equation}\label{eqn: slithering holonomy}
    \Sigma_{xy}\circ\Sigma_{yz}\circ\Sigma_{zx}
\end{equation}
is generically not a root of the identity\footnote{for the case that $\rho\in \textrm{Hit}_3$, this element is a square root of $I$ exactly when the triple ratio of the triangle $(x,y,z)$ is trivial.}. 
We direct the reader to Section 5 of \cite{BonDre} for the details. 

\begin{remark}
    While Bonahon-Dreyer define the slithering map only for maximal laminations, there is an analogous construction for arbitrary laminations; the map $(h,g)\mapsto \Sigma_{hg}$ is only defined when all of the gaps between $g$ and $h$ are bounded by asymptotic leaves. The proof of this fact follows with no adjustments at all.
    
    In particular, when $\cN$ is a snug train neighborhood track carrying $\lambda$, then $\Sigma_{hg}$ is defined whenever $h$ and $g$ are joined by a tie of $\widetilde \cN$; see \S\ref{subsec: surface theory}.
\end{remark}

\medskip
Given a lamination $\lambda$ on $S$ and an Anosov representation $\rho$, we will use the slithering map to construct a flat connection on the line bundles $F_2 = F_2^\rho$ over $\lambda$ (recall Definition \ref{def: F2 over lamination}).
%\footnote{Recall that we have embedded $\lambda$ into $\mathbb PTX$ in a canonical way and then pulled back the Anosov line splitting to obtain the line bundles $F_i$ over $\lambda$.} 
%$E_i=E_i^\rho$ over $\hlambda$, which descend to flat structures on $F_i = F_i^\rho$ over $\lambda$.
Flow invariance of the Anosov splitting from \S\ref{subsec: Anosov property} guarantees that $\nabla^\rho$ restricts to a flat connection on $F_2|_{\ell}$ for each leaf $\ell \subset \lambda$.

%Each of these line bundles is trivialized over each leaf $\ell\subset \hlambda$: for $p,q\in \ell$, it is true that $E_i\vert_p=\ell_* E_i\vert_q$, so that $\nabla^\rho$ induces a connection along the leaves of $\hlambda$. 

 %, and let $\widehat\cN\to \cN$ be induced by the orientation covering $\hlambda \to \lambda$.

\begin{definition}\label{def: slithering section}
    Let $\cN$ be any snug train track neighborhood for $\lambda$ equipped with a foliation by ties, and let $t$ be a tie.
    A non-zero section $\omega:t\to F_2\vert_t$ is a \textit{slithering section} if for all $p,q\in t\cap \lambda$
    \[\omega(q) = \Sigma_{\ell(q)\ell(p)}\omega(p) \]
\end{definition}
In other words, a slithering section is obtained by extending a non-zero vector over a tie using the slithering map.

\begin{lemma}[Slithering connection]\label{lemma: E_2 is flat}
    The collection of slithering sections define a flat H\"older connection on $F_2$ (extending $\nabla^\rho$ along the leaves) called the \textit{slithering connection}.
    %There exists a flat H\"older connection $\nabla$ on $F_i\to \lambda$ extending $\nabla^\rho$  in the sense of Definition \ref{def: flat connection}. %On each minimal component of $\lambda$, $\nabla$ is unique.
\end{lemma}

\begin{proof}
    Let $\nabla$ be the local extension of all slithering sections by $\nabla^\rho$-parallel transport. Since $F_2$ is one-dimensional, non-zero sections are also sections of $\mathcal B(F_2)$. Certainly $\nabla$ satisfies the restriction condition and it is H\"older because the slithering map is.

    Over any transversal $t$, there exists a unique slithering section up to scale (because slithering maps are linear).
    Since $\nabla$ extends $\nabla^\rho$ and $\Sigma_{qp}^\rho$ depends only on $\ell(q)$ and $\ell(p)$, it follows that for two ties $t$ and $t'$ and an isotopy $H$ between them that $H_* \circ \nabla_t = \nabla_{t'}$.
\end{proof}

Lemma \ref{lemma: E_2 is flat} and Proposition \ref{prop: combinatorialized lamination bundles} imply Theorem \ref{thm:flat connection over lamination} and Proposition \ref{prop: Hitchin train track bundle} as corollaries. 

%after noting that the slithering construction is compatible with the involution $i$ defining $F^\rho$.

\section{Affine Laminations}\label{sec: affine laminations}

The equivariant transverse measures discussed in \S\ref{subsec: eq transverse measures} depended on a choice of norm on $E^\rho$.
In Section \ref{section: flat bundles over laminations} we found that this bundle restricted to 
%the tangents $\hlambda$ of 
a geodesic lamination $\lambda$ was flat.
In this section, we use the (scalar class) of a flat section determined by the slithering connection on a transversal arc to $\lambda$ to find a preferred \emph{affine} representative of our equivariant measured lamination.
We then study the space of such transversely affine measured laminations coming from a Hitchin representation $\rho$.

\subsection{Affine measures from $\ML(X)^{G}$}\label{subsection: affine measures from Hitchin world}
%We will show in this section that when an equivariant measure is written in terms of the flat section, we can recover an affine lamination.
Given $\mu\in \ML(X)^{G}$ let $\lambda'$ be its support, and choose a maximal lamination $\lambda\supset \lambda'$. %The results that follow will depend on this choice.

%Let $k$ be a transversal arc to $\lambda$.
%and $k$ a transversal arc to $\lambda$. 
By Theorem \ref{lemma: E_2 is flat}, the  bundle $F^\rho_2 \to \lambda$ from Definition \ref{def: F2 over lamination} is flat by way of the slithering connection.
Consider a snug train-track neighborhood $\mathcal N$ of $\lambda$ equipped with a foliation by ties. The collapse map $q:\cN \to \tau$ has image which is a train track, and $q$ exhibits that $\tau$ carries $\lambda$. 
Proposition \ref{prop: Hitchin train track bundle}
gives us a holonomy representation  
\[\chi: \pi_1 \tau\rightarrow \R^\times\] 
and states that $F^\rho_2\to \lambda$ is the pullback by $q$ of the flat bundle $F_\chi\to \tau$ with holonomy $\chi$. %(Proposition \ref{prop: combinatorialized lamination bundles})

We associate to $F_2$ over $\lambda$ the flat ray-bundle 
\[R=R^\rho_2(\lambda) = F^\rho_2(\lambda) /(\pm1).\]
This will ease some of the pressure on keeping track of orientations. The holonomy of $R$ is $|\chi|: \pi_1\tau \to \R_{>0}$. 

Let $k$ be tie of $\cN$ and suppose $\omega: k\cap \lambda \rightarrow R\vert_k$ is a flat section.
Recall that $s^1:T^1X \rightarrow E^\rho_2$ was a global unit norm section depending on a norm on $E^\rho$.
Denote by $[s^1]:\lambda \to R$ the section obtained by taking the class of $s^1(\widehat x)$ in $R|_{x}$, where $\widehat x$ is a unit tangent vector to $x$ in $\hlambda$. 

There is a continuous non-vanishing function $f:\lambda\cap k\rightarrow \R$ satisfying
\[[s^1]=f\omega \text{ on $\lambda\cap k$}.\]

%Recall that $ \mu_{\hk}$ is the flip invariant measure supported on $\hk\cap \hlambda$ satisfying $\pi_*(\mu_{\hk}) = 2 \mu_k$, where $\pi: T^1X \to X$ is the projection.

Unpacking the definitions of an equivariant transverse measure (Definition \ref{def: transverse measure}) and the construction of $R$ and $[s^1]$, we conclude that the bundle valued measure $\mu_{k}\otimes [s^1]$ satisfies for any transverse homotopy $H$ relating $k$ to $k'$

\begin{equation}\label{eqn: flow invariance bundle}
    H_*(\mu_{k}\otimes [s^1]) = (H_*\mu_{k})\otimes (H_* [s^1]) = \mu_{k'} \otimes [s^1].
\end{equation}
Compare with Lemma \ref{lemma: homotopy invariance of tensor}.

We define the measure $\nu_{k}$ with support contained in $k\cap \lambda$ by
\[f\nu_{k} = \mu_{k}.\]
Then the bundle-valued measure satisfies
\[\mu_{k}\otimes [s^1] = \nu_{k} \otimes \omega.\]

 We define a \emph{first return homotopy} with respect to an orientation on $k$
\begin{equation}\label{eqn: first return homotopy}
    H: (k\cap \lambda) \times I \rightarrow \lambda
\end{equation}
as follows, although the parameterization of $H$ is unimportant. The orientation on $k$ locally co-orients the leaves of $\lambda$ meeting $k$. For each $x\in k\cap \lambda$, the image of $H$ is the segment of a leaf of $\lambda$ joining $x$ to its first  forward return.  %according to this co-orientation. 
The first return homotopy is only well-defined when for each $x\in k\cap \lambda$, this leaf of $\lambda$ returns to $k$ in the forward direction.
%\footnote{The first return only defines a discrete dynamical system in $\hlambda$.}. 
%The behavior of the first-return homotopy away from $\lambda$ is unimportant, as is its parametrization.

Using the first-return homotopy, we can associate to each point $x\in k$ a loop $\gamma_x$, which is simply the image under $H$ of $x$ followed by a subarc of $k$. We then treat the holonomy as a locally constant continuous function on $k\cap \lambda$, so that 
\[|\chi(x)| = |\chi (\gamma_x)|.\]

\begin{theorem}\label{thm: nu is affine}
    Assuming that the first return homotopy \eqref{eqn: first return homotopy} is defined, the measure $\nu_k$ defined above as
    \[\nu_k = \left(\frac{\omega}{[s^1]}\right)\mu_k\]
    is affine with holonomy $|\chi|$, i.e.,
    \begin{equation}\label{eqn: measure affine}
    H_*\nu_k = (|\chi|\circ H\inverse) \nu_k. 
    \end{equation} 
\end{theorem}

\begin{proof}
    By following the definition of $|\chi|$, we find that
\[H_* \omega = (|{\chi}|\circ H\inverse)\inverse \omega\]
and thus
\[H_*\left(\nu_{k}\otimes \omega \right) = H_*(\nu_{k}) \otimes H_*\omega =
%H_*(\nu_{k})\otimes (|{\chi}|\circ H\inverse)\inverse \omega =
\frac{1}{\left(|{\chi}|\circ H\inverse\right)}  H_*(\nu_{k})  \otimes \omega\]

From \eqref{eqn: flow invariance bundle} and the definition of $\nu_k$, we have 
\[H_*\left(\nu_{k}\otimes \omega \right) = \nu_{k}\otimes \omega.\]

The Theorem follows directly.
\end{proof}

\begin{remark}
    In the case that $\lambda$ consists of more than one minimal component (i.e., the first return map to a single transversal is possibly not well-defined in forward or backward time), one can amend the above Theorem and its proof as follows. Choose a finite system of ties $\{k_1,\dots, k_n\}$, one for each minimal component of $\lambda$. Then the first return to the union
        \[k=\cup_{i=1}^n k_i\]
    is well-defined.
    The flat connection trivializes the bundle over each $k_i$. One arrives at the same result: the measure $\nu_k$ is affine on each component.  Isolated leaves may be handled similarly.  
\end{remark}

\subsection{Integrating to an AIET}\label{subsec: AIET}
    There is a dynamical relationship between affine measures on laminations and \emph{affine interval exchange transformations} (AIETs) (see \cite{AIETflips} or \cite{Cobo:AIET,MMY:AIET} and there references therein).  The AIET that we construct has a non-trivial involutive symmetry.
    When $\lambda'= \lambda$ is (dynamically) minimal and (topologically) maximal we construct an AIET by integrating the transverse measure (Proposition \ref{prop: AIET from measure}). The same construction works for each minimal components of $\lambda$ which is not a nullset for the transverse measure, though there may be no such minimal component; see Figure \ref{fig: prototypes} or \S\ref{appendix}.

    For simplicity, assume that $\lambda = \lambda'$ is minimal and maximal, and let $k$ be a tie of a snug train track neighborhood of $\lambda$. 
    In the orientation double cover $\hlambda$, there is a first return map $\sigma: \hk \cap \hlambda \to \hk \cap \hlambda$ for the flow.
    Recall that $\hk\cap \hlambda = (\hk^+\cap \hlambda)\cup (\hk^-\cap \hlambda)$ where each of $\hk^\pm\cap \hlambda$ projects 
    homeomorphically to $k\cap \lambda$, and the two components are exchanged by $p\mapsto \overline p$. 
    Let $\widehat \nu$ be the local pushforward of $\nu_k$ under the inverse of the covering projection $\hk \to k$ (so that $\widehat \nu (\hk) = 2\nu_k(k)$).
    
    Let $p=k(0)$ be the initial point of $k$, and let $p^\pm$ be the two lifts of $p$ to $\hk$. Assume that the total $\nu$-mass of $k$ is one, and define a map $\mathcal I:\hk\cap \hlambda \to \R$ by integrating $\widehat\nu$:

    For $x\in \hk^+\cap \hlambda$ let
    \[\mathcal I(x) = \int_{p}^{\pi (x)} ~d\nu_k \]
    and for  $x\in \hk^-\cap \hlambda$ let
    \[\mathcal I(x) = 1+\int_{p}^{\pi(x)} ~d\nu_k,\]
    where the integrals are taken along $k$.
    There is an %$\mathcal I$-equivariant 
    involution on $[0,2]$, denoted by $r\mapsto \overline r$ given by the rule that for $r\in [0,2],$
    \[\overline {r} =1+r \mod{2}\]
    which satisfies
    \[\mathcal I(\overline x) =\overline {\mathcal I(x)}.\]

    Note that $\mathcal I$ is continuous if $\nu$ has no atoms. If $\nu$ has an atom $x\in \hk\cap \hlambda$, then $[0,2]\setminus \im \mathcal I$ contains an interval of length $\nu(x)$; let $J(x)$ be the closure of this interval so that $\mathcal I(x)$ is the right-hand endpoint of $J(x)$.

    \begin{lemma}\label{lemma: measurable inverse}
        There exists a Lebesgue a.e. defined map $\varrho:[0,2]\to \hk\cap \hlambda$ satisfying
        \[\varrho\circ \mathcal I =Id_{\hk\cap \hlambda}, ~ \text{$\widehat\nu$ a.e.}\]
    \end{lemma}
        
    \begin{proof}
        Write $\widehat \nu = \widehat\nu_{a}+\widehat\nu_0$ where $\widehat\nu_a$ is the atomic part of $\widehat\nu$, supported on the countable set of atoms $\{a_i\}_{i\in \N}$.  Note that if $a$ is an atom of $\widehat\nu$, then every point of its $\sigma$ orbit is an atom as well.  Hence the atoms are dense if $\widehat\nu_a$ is non-trivial, by minimality.
        Define $\varrho\vert_{J(a_i)}: r\mapsto a_i$.
        %We claim that there is a countable set $E$ in $\hk\cap \hlambda$ such that $\mathcal I$ is injective away from $E$.
        
        We claim that there is a countable set $\mathcal E$ in the complement of $\cup J(a_i)$, such that if $r \in [0,2] \setminus \cup J(a_i)\cup E$, then $\mathcal I\inverse(r)$ is a singleton.  
        Indeed, suppose $r\in [0,2]\setminus \cup J(a_i)$, and suppose that there exist two distinct points $b,d\in \mathcal I \inverse (r)$ . Suppose further that there exists $c\in \hk\cap \hlambda$ between $b$ and $d$.  
        If $\nu_a$ is non-trivial then there is an atom $a_i$  between $b$ and $d$. This is a contradiction, because it implies that $\mathcal I (a) \neq \mathcal I (b)$. 
        If $\widehat\nu_a$ is trivial, then the existence of such a point $c$ is still a contradiction: by minimality $c$ is not isolated, hence there is a non-empty open set containing $c$ strictly between $b$ and $d$; since $\lambda$ is the support of $\widehat\nu$ this set has positive $\nu$-measure. %so the subarc between $b$ and $d$ has non-zero $\nu$-measure.
    
        We conclude in both cases that pairs of points $b$ and $d$ satisfying $\mathcal I (b) = \mathcal I (d)$ must bound a gap in $k\cap \lambda$.
        This is a countable set.
        Define $\varrho$ on $[0,2]\setminus \cup J(a_i) \cup \mathcal E$ by $r\mapsto \mathcal I \inverse(r)$.
        Standard arguments show that $\varrho$ is Borel measurable.
        %Therefore $\mathcal I$ is measurably injective on the complement $\hk \cap \hlambda \setminus \cup(I\inverse(J(a_i))$ and we may define $\varrho$ to be its measurable inverse on $[0,2]\setminus \cup J(a_i)$; the map $\varrho$ is simply not defined when $\mathcal I$ is not injective.
    \end{proof}
    
    We have recovered the following dynamical data from an affine measure.  It is an affine interval exchange transformation (with flips).

    \begin{proposition}\label{prop: AIET from measure}
        There exists a finite set $\{r_1,\dots,r_n\}\subset [0,2]$ and a piecewise affine map $\mathcal S:[0,2]\setminus \{r_1,\dots,r_n\}\rightarrow [0,2]$ with dense image so that the diagram
                
        \[ \begin{tikzcd}
        \left[0,2\right] \arrow{r}{\mathcal S} \arrow[swap]{d}{\varrho} & \left[0,2\right] \arrow{d}{\varrho} \\
        \hk\cap\hlambda \arrow{r}{\sigma}& \hk\cap\hlambda
        \end{tikzcd}
        \]

        \noindent commutes up to null sets.
        Moreover
        \begin{itemize}
        \item $\varrho_*Leb = \widehat\nu$,
        \item $\varrho(\mathcal S(\overline r))= \sigma(\overline{\varrho(r)})$, and
        \item The derivative $D_r\mathcal S = \chi (\pi(\gamma_{\varrho(r)}))\in \R^\times$, whenever $r\not\in \{r_1, ..., r_n\}$, where $\gamma_p$ is the oriented geodesic segment from $p$ to $\sigma(p)$, and $\pi(\gamma_p)$ is the corresponding loop in $\tau$.
        \end{itemize}
    \end{proposition}

\begin{proof}
    The proposition follows from Theorem \ref{thm: nu is affine} and the construction; the details are left as an exercise for the reader.  See also \cite[\S2]{Cobo:AIET}.
\end{proof}

    \begin{remark}
        The derivative $D_r(\mathcal S)<0$ if and only if $r\in [0,1]$ and $\mathcal S(r) \in [1,2]$ or $r\in [1,2]$ and $\mathcal S(r) \in [0,1]$.
        In these cases, the loop $\pi(\gamma_{\rho(r)})$ is a loop in $\tau$ that begins and ends on the same side of our transversal $k$.  See  Figure \ref{fig:tt cover} for a schematic.
    \end{remark}

\begin{figure}[h]
    \centering
    \includegraphics[width=.9\linewidth]{figures/tt_cover.pdf}
    \caption{The figure depicts the orientation covering of a minimal geodesic lamination $\lambda$.
    The blue arcs represent homotopy classes of geodesic segments joining the relevant transversals; train tracks carrying $\hlambda$ (or $\lambda$) can be obtained by collapsing $\hk^\pm$ (or $k$) to points.}  %The orientation cover of a minimal lamination $\lambda$ represented on a train track with one switch determined by a transverse arc $k$. \mb{I want to discuss this caption and figure and make them cooperate.} }
    \label{fig:tt cover}
\end{figure}
    \begin{remark}
        As in \S\ref{subsection: affine measures from Hitchin world}, the results in this subsection have analogues when $\lambda$ is not dynamically minimal. As before, one would choose a finite system of transversals and integrate the equivariant measure from their union to achieve an affine interval exchange transformation. We omit the details for tedium.
    \end{remark}
    
\subsection{Relationship to affine laminations}\label{subsec: affine laminations}
For convenience choose a hyperbolic structure $X$ on a surface $S$ and a connected geodesic lamination $\lambda \subset X$. 
Let $\mathcal N$ be a snug train track neighborhood of $\lambda$.
Let $\chi:\pi_1 \cN \rightarrow \R_+$ be a homomorphism and $\widetilde{\mathcal{N}}$ be a covering space of $\mathcal N$ for which $\pi_1 \widetilde \cN<\ker \chi $. 
Let $\tlambda \subset \widetilde{\mathcal{N}}$ be the pre-image under the covering map of $\lambda$. 
%which is carried by the train track $\widetilde\tau$ corresponding to $\widetilde{\mathcal N}$\footnote{We will work with flat bundles over $\widetilde \tau$. There is implicitly (by the collapse map) a pullback bundle over $\widetilde{\mathcal{N}}$.}.

Following Hatcher and Oertel \cite{HO:affine}, 
an \textit{affine lamination} is a $\chi$-equivariant transverse measure supported on $\tlambda$. That is, a homotopy invariant transverse measure $\widetilde \nu$ satisfying
\begin{equation}\label{eqn: equiv for affine measure}
    \chi(\gamma)\gamma.\widetilde \nu = \widetilde \nu
\end{equation}
for all $\gamma\in \pi_1 \cN$ acting by deck transformations.\footnote{The relationship between the pushforward by a loop in a bundle and the deck transformation on a cover is inverse. This is the same convention from the previous subsection.}
Equivalently, for every transverse arc $k\subset \widetilde\cN$ and every $\gamma \in \pi_1\cN$, we have \[\widetilde\nu(\gamma.k) = \chi(\gamma)\widetilde\nu(k),\]
where $\widetilde \nu(k) = \int ~d\widetilde\nu_k$.

As always, there is a flat line  bundle $E_\chi\to \cN$ with holonomy $\chi$ defined as

\[E_\chi = \widetilde{\cN}\times \R / \chi.\]

The constants in $\R$ give flat sections of $\widetilde{\cN}\times \R $, i.e. sections $\widetilde w$  such that
\[\chi(\gamma)^{-1} \gamma. \widetilde w = \widetilde w.\]

Then \[\widetilde\nu \otimes \widetilde w\]
is invariant under $\pi_1 \cN$, and descends to a well-defined object on the bundle $E_\chi$. 
%This data on the cover $\widetilde{\tau}$ is an affine lamination.

There is a description of affine measure on a lamination which does not reference a covering space. Following the notation from \S\ref{subsec: tt and flat connections}, let 
\[\mathcal A = \{\psi_\alpha : \zeta\inverse (U_\alpha) \to U_\alpha \times \R\}\]
be an atlas of vector bundle charts for a flat $\R$-bundle $E$ over $\lambda$. In each chart $U_\alpha$, let $\nu_\alpha$ be a transverse measure to $\psi_\alpha(\lambda\cap U_\alpha)$ which is invariant under transverse homotopy in the chart, and $\omega_\alpha: U_\alpha\cap \lambda \rightarrow E\vert_\lambda$ a flat section in the chart. 

An affine measure on $\lambda$ valued in $E$ is then specified by the choice of an object of the form $\nu_\alpha\otimes \omega_\alpha$ on each chart. These choices are of course required to be coherent on overlaps of charts:
\[(\nu_\alpha\otimes \omega_\alpha)\vert_{U_\alpha\cap U_\beta} = (\nu_\beta\otimes \omega_\beta)\vert_{U_\alpha\cap U_\beta}.\]

%Note that if $\cN$ is a snug train track neighborhood of $\lambda$ and $\tau$ is the quotient of $\cN$ by collapsing the ties, then Proposition \ref{prop: combinatorialized lamination bundles} gives us a flat bundle over $\tau$ that is compatible with $E$ over $\lambda$.
%Pulling back this flat bundle over $\tau$ to $\mathcal N$, extends the original flat bundle over $\lambda$ to $\cN$.

If $\cN$ is a snug train track neighborhood of $\lambda$  with $q: \cN \to \tau$ the quotient by collapsing the ties, then by using Proposition \ref{prop: combinatorialized lamination bundles} and pullback by the map $q:\lambda\subset \cN \to \tau$, we may pass between bundles over the traintrack neighborhood, $\cN$, the lamination, $\lambda$, and the train track, $\tau$.

\begin{remark}
    The two notions of affine lamination outlined above are the same. That is, to each affine lamination in the sense of Hatcher--Oertel, there is an affine line bundle and an affine lamination described in local charts (and vice versa). The proof is an exercise in covering theory.
\end{remark}

The purpose of discussing affine laminations is the following result that ties together some of the main objects of study in this paper.

\begin{theorem}\label{thm: affine lamination}
    For $\mu\in \ML(X)^{G}$, let $\mathcal N$ be a (train track) neighborhood of $\lambda=\supp \mu$. Then the product
    \[\mu\otimes [s^1]=\nu\otimes\omega\]
    is an affine lamination on $\lambda$. The relevant bundle is $R^\rho_2\to \lambda$, the ray bundle got as a two-fold quotient of $F^\rho_2$. Its flat sections are the slithering sections constructed in Section \ref{subsection: Hitchin bundles are flat} producing the holonomy representation $|\chi|: \pi_1\tau \to \R_{>0}$.
\end{theorem}

\begin{proof}
    Lemma \ref{lemma: homotopy invariance of tensor} shows that $\mu\otimes [s^1]$ satisfies the required homotopy invariance. To show that it may be represented as a product of a flat section and a measure satisfying \eqref{eqn: equiv for affine measure} on $\widetilde{\mathcal N}$ is 
    Theorem \ref{thm: nu is affine} and Theorem \ref{lemma: E_2 is flat}.
\end{proof}

To each geodesic lamination and Hitchin representation $\rho: \pi_1S\to \SL(3,\R)$, we have a flat ray bundle $R_2^\rho(\lambda) = F_r^\rho(\lambda)/(\pm1)$.

\begin{definition}
For each Hitchin representation $\rho: \pi_1S \to \SL(3,\R)$, define
\[\ML^\rho(S) = \{(R_2^\rho(\lambda), \nu\otimes \omega)\}\]
as the set of pairs where $\lambda \subset S \text{ is a geodesic lamination and } \nu\otimes \omega \text{ is a nonzero affine measure on  }R_2^\rho(\lambda)$.

We consider this space with its natural weak-$*$ topology.\footnote{In the next section, we describe explicitly the weak-$*$ topology on this space as it requires some clarification.} 
\end{definition}

\medskip
If $\eta: \pi_1S \to \Aff^*(3,\R)$ has linear part $\rho$, and bending cocycle $\psi$ with support $\lambda$, then Theorem \ref{thm: equiv measures sphere} gives us a unique $\mu \in \ML(X)^G$ with $\psi = \mu \otimes s^1$.
Theorem \ref{thm: affine lamination} tells us that $\mu \otimes [s^1]$ is an affine measure $ \nu \otimes \omega$  on $R_2^\rho(\lambda)$.

In other words, combining Theorems \ref{thm: equiv measures sphere}  and \ref{thm: affine lamination} proves the following theorem.
\begin{theorem}
There is a map
    \begin{equation}\label{eqn: beta}
    \beta_+: \{\eta: \pi_1 S\rightarrow \TAff^*(3,\R)~\text{not semi-simple}\mid L(\eta)\in\rho\}/\TAff^*(3,\R) \to \ML^\rho(S)
\end{equation}
recording the affine bending data from a $\Aff^*(3,\R)$ conjugacy class of representation with linear part (conjugate to) $\rho$. 
\end{theorem}

The `$+$' decoration on this map reflects the global choice of a component of $\partial \Xi_\eta$. There is a map $\beta_-$ defined analogously for the opposite component. 
The next subsection is devoted to proving continuity of $\beta_+$.

\begin{remark}
    The flat ray bundle $R^\rho_2(\lambda)$ contains enough data to reconstruct the flat line-bundle $F^\rho_2(\lambda)$. The fibers of $F_2$ covers $R_2$ two-to-one, and the holonomy of $F_2$ picks up a minus sign based on topological data which can be seen in a train track carrying $\lambda.$ Therefore no data is lost by considering only the ray bundle.
\end{remark}

\begin{remark}\label{rmk: no dual affine tree}
    The fact that the slithering connection has non-trivial holonomy in general around the boundary components of a snug train track neighborhood of $\lambda$ (see \eqref{eqn: slithering holonomy}) makes it difficult to define an affine structure on the (dual) dendrite which is a  quotient of the leaf space of $\lambda$.
\end{remark}
        
\subsection{Continuity of $\beta_+$}\label{subsec: continuity of weights}
To a train track $\theta$ with branches $b(\theta)$, the \emph{weight space} $W(\theta)$ is the vector subspace of $\R^{b(\theta)}$ cut out by the \emph{switch conditions}.  That is, a given $\phi \in W(\theta)$ satisfies for all switches
\[\sum \phi(b_i^{in}) = \sum \phi (b_j^{out}),\]
where  $b_1^{in} , ..., b_k^{in}$ are the incoming and $b_1^{out}, ..., b_\ell^{out}$ are the outgoing half-branches of $\theta$ at that (oriented) switch  \cite{PH, Thurston:notes}.
There is a non-negative cone $W_{\ge 0}(\theta)$ obtained by intersecting $W(\theta)$ with $\R_{\ge 0}^{b(\theta)}$.
Given a measure $\mu$ with support contained in a lamination carried by $\theta$, denote by $\mu(b)$ the integral of the constant function $1$ over a transversal to $b$. This function $w_\mu : b\mapsto \mu(b)$ is the non-negative weight system of $\mu$ on $\theta$.

The following lemma follows from \eqref{eqn: equiv for affine measure} and the definitions.
\begin{lemma}\label{lem: equivariant weight system}
    An affine lamination with holonomy $\chi$, thought of as a transverse measure $\widetilde \nu$ on $\widetilde \lambda$,  gives a weight system $w= w_\nu: b(\widetilde\tau) \rightarrow \R$ 
    satisfying 
        \[ w (\gamma.b) = \chi(\gamma)w(b) \text{ for all $\gamma \in \pi_1\tau$.}  \]
\end{lemma}

\begin{remark}
    By cutting $\tau$ along a finite collection of points (called \emph{stops}), one can record these data as a weight system on a train track $\tau$ with stops and holonomy; see \cite{HO:affine}.
\end{remark}

Suppose $\mu_n \in \ML(X)^{G}$ have support $\lambda_n'$ and converge weak-$*$ to $\mu$ with support $\lambda'$. 
Choose maximal completions $\lambda_n$ of $\lambda_n'$.  
Passing to a subsequence, we assume that $\lambda_n$ converge in the Hausdorff topology to a geodesic lamination $\lambda$ that, necessarily, is maximal and contains $\lambda'$.
Let $\cN$ be a snug train track neighborhood of $\lambda$ collapsing to $\tau$.
Note that $\lambda_n$ is carried snugly by $\tau$ for large enough $n$.

According to Theorem \ref{thm: affine lamination} and Lemma \ref{lem: equivariant weight system}, for each $n$, there are corresponding affine measures $\nu_n \otimes \omega_n$ and holonomy representations 
\[|\chi_n|: \pi_1 \tau \to \R_{>0}\] 
and $|\chi_n|$-equivariant weight systems $w_n=w_{\nu_n}$.
We also have the corresponding holonomy and weight system, $\chi$ and $w = w_\nu$, for $\nu$.

The following theorem states that equivariant transverse measures map continuously to affine measured laminations with the weak-$*$ topology.
\begin{proposition}\label{prop: convergence in tt charts}
    For all $b\in b(\widetilde\tau)$, the weight systems
    $w_n(b) \to w(b)$, as $n \to \infty$.
\end{proposition}

The proof relies on a continuity property of the slithering map indicated in the following Proposition.  
The techniques of Section 5 of \cite{BonDre} apply readily to give the proof, which we omit.

\begin{lemma}\label{lem: slithering continuous}
Suppose $\lambda_n$ are complete geodesic laminations converging in the Hausdorff topology to $\lambda$.
        Let $k$ be a transversal to $\lambda$, and lift the situation to the universal cover of $S$ so that $\widetilde k$ meets $\widetilde\lambda_n$ transversely for all $n$ large enough.

    Let $\ell^0_n$ and $\ell^1_n$ be the first and last leaves of $\widetilde k\cap \tlambda_n$, ordered by $\widetilde k$. For each $n$, let $\Sigma_n\in \SL(d,\R)$ be the slithering map defined by $\tlambda_n$ from $\ell^0_n$ to $\ell^1_n$.
    Define analogously $\ell^0$, $\ell^1$, and $\Sigma$ for $\lambda$.

    Then $\Sigma_n \to \Sigma\in \SL(d,\R)$.
\end{lemma}

\begin{proof}[Proof of Proposition]\ref{prop: convergence in tt charts}
    Let $k$ be a transversal to some branch $b$. Without loss of generality, assume that for all $n$,  $\omega_n$ is such  that its evaluation on the leftmost leaf of $\lambda_n$ over $b$ is equal to $s^1$ evaluated on the same leaf.
    
    For each $n$, we have $[s^1] = f_n \omega_n$, where $f_n: \lambda_n \cap k \to \R_{>0}$ is continuous. 
    Lemma \ref{lem: slithering continuous} gives that $f_n$ converge in the following sense: if $x_n \in \lambda_n \cap k \to x \in \lambda\cap k$, then $f_n(x_n) \to f(x)$.
    Thus we can extend $f$ continuously to a function $\tilde f : k \to \R_{>0}$, and then find continuous extensions $\tilde f_n : k \to \R_{>0}$ of $f_n$ so that \[\|\tilde f_n -\tilde f\|_\infty \to 0, \text{ as } n \to \infty.\]
    
    %Weak-$*$ convergence of $\mu_n$ then implies weak-$*$ convergence of $\nu_n = f_n \mu_n$ on $k$.    
    %Let $\tilde f$ be any continuous extension of $f: k\cap \lambda\to \R_{>0}$ to $k$.
    Using the triangle inequality, for any continuous $g: k \to \R$, we have

    \[\left| \int gf_n ~d\mu_n - \int gf~d\mu\right|
    \le \left|\int g\tilde f_n ~d\mu_n - \int g\tilde f ~d\mu_n \right| + \left| \int g \tilde f~d\mu_n - \int g\tilde f~d\mu \right|. \]
    The second term goes to zero by weak-$*$ convergence, while the first is bounded by \[\|g\|_\infty \|\tilde f_n - \tilde f\|_\infty \mu_n (k).\]
    Since $\mu_n(k) \to \mu(k)$, this term goes to zero as well, proving the Proposition.
\end{proof}

While $\ML(X)^G$ was constructed from a choice of norm and parameterization of the geodesic flow on $T^1S$, $\ML^\rho(S)$ does not reference any such choice. 

\begin{theorem}\label{thm: MLrho is a sphere!}
    The map $\beta_+$ defined in equation \eqref{eqn: beta} is a homeomorphism that is homogeneous with respect to positive scale.
    Consequently, the quotient by positive scale $\PML^\rho(S)$ is a sphere of dimension $6g-7$.
\end{theorem}

\begin{proof}
    The proof of Theorem \ref{thm: equiv measures sphere} gives that $\ML(X)^G$ is identified with non-semi-simple coaffine representations with linear part $\rho$ up to conjugation.
    Proposition \ref{prop: convergence in tt charts} proves continuity of the map  $\ML(X)^G\to \ML^\rho(S)$. 
    
    To see that $\beta_+$ is injective, consider two distinct points $\mu$ and $\mu' \in \ML(X)^G$. If their supports differ, then the bundle factor in $\beta_+(\mu)$ and $\beta_+(\mu')$ differ. If they are supported on the same lamination, then the measures differ on some transversal, and therefore the measure factors in the codomain differ.

    Integration on sufficiently small transversals distinguishes measures, so the weak-$*$ topology on $\ML^\rho(S)$ is Hausdorff.

    By taking positive scalar classes, the map $\beta_+$ induces a continuous injection from a compact space to a Hausdorff space, hence is a homeomorphism to its image.

    The inverse to $\beta_+$ is given by writing the measure $\nu\otimes \omega = \mu \otimes [s^1]$ and passing the ratio $\omega/[s^1]$ over the tensor. Then $\mu\in \ML(X)^G$, and applying Theorem \ref{thm: equiv measures sphere} again completes the proof.
\end{proof}

\section{Appendix}\label{appendix}
We review a motivating example due to Ungemach \cite{Ungemach:thesis}, with some simplifications and extensions. This appendix is designed to be largely self-contained, apart from the content of Section \ref{sec: background}. 

The apparent paradox of these examples is that while one has bent a representation along an isolated leaf of a lamination which accumulates to a simple closed curve, the data of such a deformation cannot be captured as a transverse measure on a lamination: such a leaf is never in the support of a transverse measure. 

Therefore, some more sophisticated technology must be developed to describe analytically the data of convex cocompact coaffine deformations in terms of bending data on laminations. This paper provides an answer to this question; \textit{equivariant} (affine) measures are the necessary extension, see Sections \ref{subsec: eq transverse measures} and \ref{subsection: affine measures from Hitchin world}.

\begin{example}\label{example: Weston's example}
    Let $\rho\in \Hit_3$ be non-Fuchsian, and let $\Omega\subset \RP^2 = \P(\R^4/[e_4])$ be the domain which $\rho(\pi_1 S)$ divides. Assume that there exist two disjoint non-homotopic simple closed curves $c_1$ and $c_2$ in $S$ so that $1$ is not an eigenvalue\footnote{Every non-Fuchsian $\rho\in \Hit_3$ has such curves; the proof of this fact is orthogonal to the current goals.} of $\rho(c_1)$ or $\rho(c_2)$. One may also choose $c_1=c_2$. Let $\Lambda_j(c_i)$ be the $j^{\textrm{th}}$ eigenvalue of $c_i$, ordered by magnitude.
    
    We may assume that $c_i$ is oriented so that $\Lambda_2(c_i)>1$ for $i=1,2$. Choose a lamination $\lambda\subset S$ consisting of exactly three leaves:
    \[\lambda = c_1 \cup c_2 \cup m\]
    where $m$ is an isolated leaf which accumulates to both $c_1$ and $c_2$. Assume also that for both $c_1$ and $c_2$, $m$ is asymptotic to $c_i$ in the negative direction with respect to their orientations.

    Let $\eta_0$ be the semi-simple coaffine representation with linear part $\rho$ which preserves the hyperplane $[\ker e^4]\subset \RP^3$ and let 
    \[dev_0:\widetilde S \rightarrow \Omega \subset [\ker e^4]\] 
    be a developing map\footnote{We obfuscate the difference between $\Omega\subset \P(\R^4/[e_4])$ and $\Omega\subset [\ker e^4]$.} so that $dev_0(\lambda)$ is a union of projective segments. The complement in $\widetilde S$ of the lamination is a disjoint union of countably many connected components:
    \[\sqcup_{i \in I} X_i = \widetilde S \setminus \tlambda.\]

\begin{figure}[h]
\includegraphics[width=.6\textwidth]{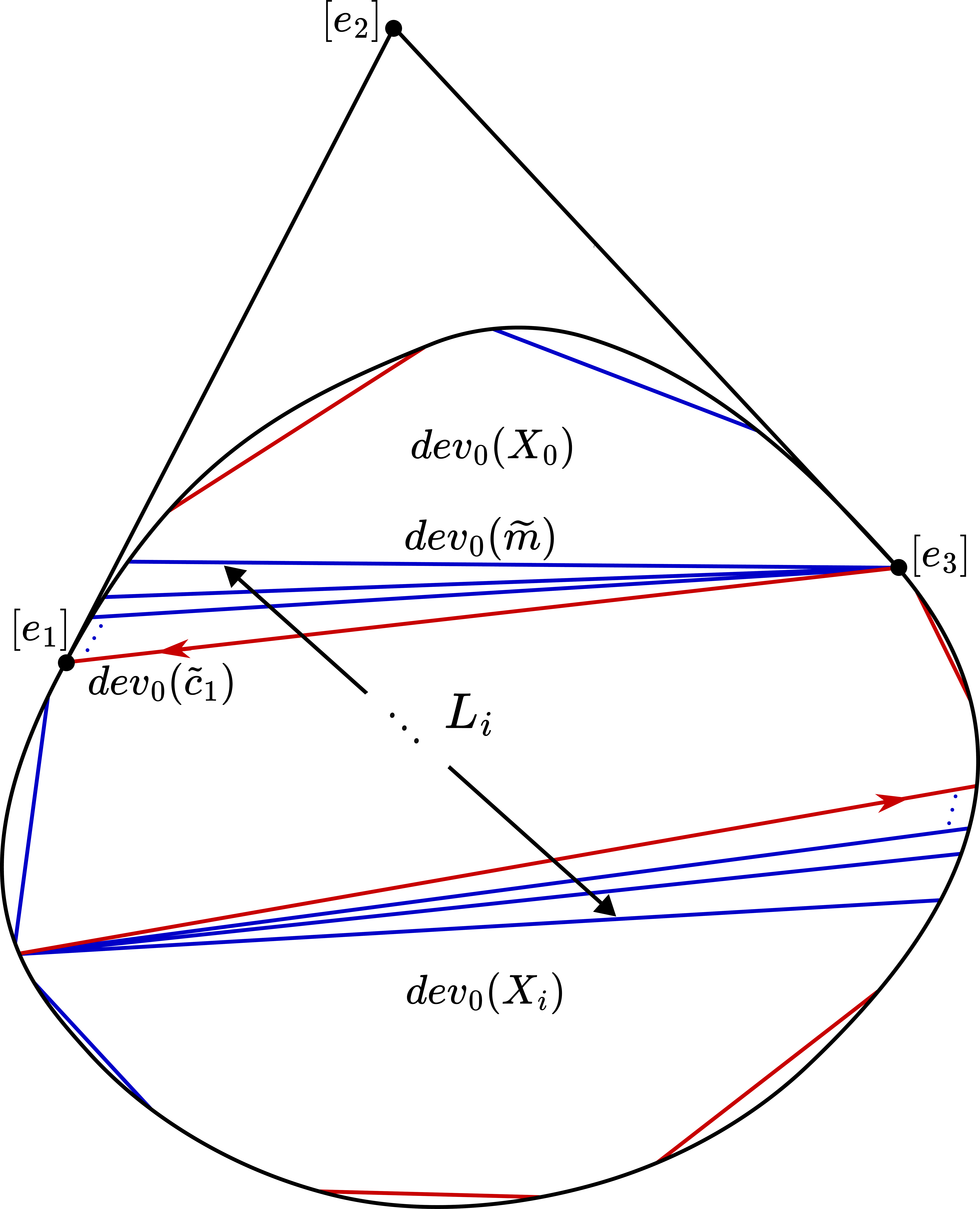}
\centering
\caption{The domain $\Omega$ with the lifts of $c_i$ in red, and $m$ in blue. Two regions and the leaves between them are depicted. }
\label{figure: Weston_develops}
\end{figure}

    Choose one such region $X_0$, and a point $\tilde p\in X_0$ projecting to $p\in S$, where $p$ is the basepoint implicit in $\pi_1 S$.  Assume (without loss of generality) that there is a lift $\tilde m \subset \partial X_0$. See Figure \ref{figure: Weston_develops}.
    
    For each region $X_i$, let 
    \[L_i = \{\ell\subset \tlambda \mid \ell ~\textrm{between}~ X_0 ~\textrm{and}~ X_i.\}\]
    
    Let $M_i\subset L_i$ be those elements of $L_i$ which are translates of $\tilde m$. Because $m$ is not closed, each element of $M_i$ may be written as a $\eta_0$-translate of $\tilde m$ in a unique way, so that 
    \[M_i = \sqcup_{j\in J(i)} \eta_0(\gamma_j).\tilde m\]
    for some countable index set $J(i)$. 

    From the discussion in Example \ref{example: coaffine translation}, there exists a one-parameter subgroup $\{Z_t\}<\TAff^*(3,\R)$ preserving $dev_0 (\tilde m)$ pointwise.

    For each $t\in \R$ we define a developing map $dev_t$ which is equivariant with respect to a representation $\eta_t$.
    Define for each $i$ the cotranslation
    \[B^t_i = \sum_{j\in J(i)} \eta_0(\gamma_j).Z_t \in \TAff^*(3,\R).\]

    Implicit in this definition is the claim that this sum converges. Asking the reader's momentary trust, we then define $dev_t$ and $\eta_t$ as follows:
    \[dev_t\vert_{X_i} = B^t_i.(dev_0\vert_{X_i})\]
    and
    \[\eta_t(\gamma) = B^t_{i(\gamma)}\eta_0(\gamma)\left(B^t_{i(\gamma)}\right)\inverse\]
    where $i(\gamma)\in I$ is the index of $\eta_0(\gamma).X_0$. This is essentially the standard construction of a developing map for a `bent' representation in the sense of Johnson and Millson \cite{JM:deformations}.
    
    It is an exercise to check that $dev_t$ is $\eta_t$-equivariant. The geometry of the cone $\mathscr C \Omega$ discussed in Section \ref{subsec: minimal domain coaffine} may be used to prove that $dev_t$ is a homeomorphism (one passes through the projection $[Q]$). 

    The interesting claim is that $B^t_i$ is well-defined, in particular that the countable sum defining it converges (absolutely). Briefly, it is sufficient to assume that $\tilde m$ and $\tilde c_1$ are asymptotic and to prove that the sum 
    
    \[\sum_{j=1}^\infty \eta_0(c_1)^n.Z_t\]
    converges absolutely. Choosing an ordered eigenbasis for $\eta_0(c_1)$ reveals the equality 
     \[\sum_{j=1}^\infty \eta_0(c_1)^n.Z_t = (b_1\sum_{j=1}^\infty \Lambda_1(c_1)^{-n}, b_2\sum_{j=1}^\infty \Lambda_2(c_1)^{-n}, 0,0)\]
    for some constants $b_1, b_2\in \R$, where we have used the identification of $\TAff^*(3,\R)$ with $\ker e_4$ to write a row vector rather than a matrix.  Certainly, these geometric sums are absolutely convergent since $\Lambda_1(c_1)>\Lambda_2(c_1)>1$. 

    Note the direction onto which $m$ was chosen to accumulate onto $c_1$ and $c_2$ was crucial: otherwise these sums would exponentially diverge. 

    It is also true that the image of $dev_t$ is convex, though it is somewhat tedious to prove. Ungemach makes the necessary argument in Section 6 of \cite{Ungemach:thesis}, or one may employ the techniques developed in Section \ref{subsec: minimal domain coaffine}. Thus, every lamination constructed in this way appears at the boundary of the convex hull for some convex cocompact coaffine representation. We omit the details of the proof for length. 
\end{example}

\bibliography{biblio}{}
\bibliographystyle{amsalpha.bst}
\Addresses
\end{document}